\documentclass[oneside,reqno,11pt,a4paper]{amsart}

\usepackage[foot]{amsaddr}
\usepackage[T1]{fontenc}

\usepackage[margin=2.4cm]{geometry}

\usepackage{etoolbox}
\patchcmd{\abstract}{\null\vfil}{}{}{}

\usepackage{caption}
\usepackage[ruled]{algorithm2e}
\usepackage{subcaption}
\usepackage{doi}
\usepackage{color}
\usepackage[skins]{tcolorbox}
\usepackage{framed}
\usepackage{verbatim}
\usepackage{fancyvrb}
\usepackage{newtxtext}
\usepackage[final,nopatch=footnote]{microtype}
\usepackage[final]{pdfpages}
\usepackage{pgfgantt}
\usepackage[super]{nth}
\usepackage{comment}

\usepackage{amsfonts}
\usepackage{mathtools}% loads amsmath; load before hyperref/cleveref for showonlyrefs
\usepackage{amssymb}
\usepackage{amstext}
\usepackage{amscd}
\usepackage{amsthm}
\usepackage{array}
\usepackage{booktabs}
\usepackage{csquotes}
\usepackage{hyperref}
\hypersetup{
breaklinks=true, bookmarksopen=true, pdftitle={}, % insert title
pdfauthor={}, % insert author
colorlinks=true,  % false: boxed links; true: coloured links
linkcolor=blue, % color of internal links (change box color with linkbordercolor)
citecolor=dgreen,        % color of links to bibliography
filecolor=black,      % color of file links
urlcolor=blue           % color of external links
}
\usepackage{cleveref}
\crefname{section}{Sect.}{Sects.}
\Crefname{section}{Sect.}{Sects.}
\crefname{subsection}{Sect.}{Sects.}
\Crefname{subsection}{Sect.}{Sects.}
\crefname{subsubsection}{Sect.}{Sects.}
\Crefname{subsubsection}{Sect.}{Sects.}
\crefname{algorithm}{Alg.}{Algs.}
\Crefname{algorithm}{Alg.}{Algs.}
\crefname{figure}{Fig.}{Figs.}
\Crefname{figure}{Fig.}{Figs.}
\crefname{table}{Table}{Tables}
\Crefname{table}{Table}{Tables}

\usepackage{makecell}
\usepackage[dvipsnames,svgnames,dvipsnames]{xcolor}
\usepackage{colortbl}
\usepackage{emptypage}
\usepackage[shortlabels]{enumitem}
\usepackage[english]{babel}
\usepackage{eqparbox}
\usepackage{extarrows}
\usepackage{float}
\usepackage{graphicx} \graphicspath{{figs_used/}}
\usepackage{latexsym}
\usepackage[nohyperlinks]{acronym}
\usepackage{listings}
\usepackage{longtable}
\usepackage{marvosym}
\usepackage{mathrsfs}
\usepackage{stmaryrd} 
\PassOptionsToPackage{dvipsnames,svgnames}{xcolor}
\usepackage{multirow}
\usepackage{setspace}
\usepackage{siunitx}
\usepackage{soul}
\usepackage{todonotes}
\usepackage{tabularx}
\usepackage{textgreek}
\usepackage{tikz}
  \usetikzlibrary{matrix,calc}
  \usetikzlibrary{decorations.markings,decorations.pathreplacing}
  \usetikzlibrary{patterns}
\usepackage{url}
\usepackage{xspace}
\usepackage[backend=biber, defernumbers=true, maxbibnames=5,
style=numeric-comp, isbn=false, bibencoding=utf8, safeinputenc,
url=false, doi=true, giveninits=true, backref=false, sorting=nyt, date=year,
sortcites=true, % enables sorting within multiple citations
]{biblatex}

\hypersetup{
breaklinks=true, 
bookmarksopen=true, 
pdftitle={Automatic Anisotropic h-Adaptive Quadrature for Neural Network Approximation of PDEs}, 
pdfauthor={Sai Krishna Kishore Nori},
colorlinks=true,  % false: boxed links; true: coloured links
linkcolor=blue, % color of internal links (change box color with linkbordercolor)
citecolor=blue,        % color of links to bibliography
filecolor=black,      % color of file links
urlcolor=blue           % color of external links
}
\definecolor{bg}{rgb}{0.93,0.93,0.93}

\DeclareFieldFormat{journaltitle}{\mkbibemph{#1}\isdot}
\renewbibmacro*{journal+issuetitle}{%
  \usebibmacro{journal}%
  \setunit*{\addcomma\space}%
  \iffieldundef{series}
    {} {\newunit \printfield{series}%
     \setunit{\addspace}}%
  \usebibmacro{volume+number+eid}%
  \setunit{\addspace}%
  \usebibmacro{issue+date}%
  \setunit{\addcolon\space}%
  \usebibmacro{issue}%
  \newunit}

\renewbibmacro*{doi+eprint+url}{%
    \printfield{doi}%
    \newunit\newblock%
    \iftoggle{bbx:eprint}{%
        \usebibmacro{eprint}%
    }{}%
    \newunit\newblock%
    \iffieldundef{doi}{%
        \usebibmacro{url+urldate}}%
        {}%
}

\usetikzlibrary{quotes, angles, patterns, calc}

\newcolumntype{L}{>{\centering\arraybackslash}m{3cm}}
\DeclareBibliographyCategory{cited}
\AtEveryCitekey{\addtocategory{cited}{\thefield{entrykey}}}

\usepackage{stackengine}
\stackMath
\newlength\matfield
\newlength\tmplength

\setlist[enumerate]{label = (\arabic*), ref = \arabic*}
\crefformat{enumi}{item~(#2#1#3)}
\Crefformat{enumi}{Item~(#2#1#3)}
\crefrangeformat{enumi}{items~(#3#1#4) to~(#5#2#6)}
\Crefrangeformat{enumi}{Items~(#3#1#4) to~(#5#2#6)}

\definecolor{Gray}{gray}{0.9}
\definecolor{dgreen}{rgb}{0,.8,0}
\definecolor{red}{HTML}{D62728}
\definecolor{blue}{RGB}{ 0, 109, 219}
\theoremstyle{definition}

\newtheorem{example}{Example}[section]

\newtheorem{remark}{Remark}[section]
\newtheorem{assumption}{Assumption}[section]
\newtheorem{lemma}{Lemma}[section]
\newtheorem{corollary}{Corollary}[section]

\newcommand{\normalarray}{\renewcommand{\arraystretch}{1.1}}

\normalarray

\newcommand{\TT}{\mathcal{T}}

 {\begin{list}{}%
    {\setlength{\leftmargin}{#1}}%
  \item[]%
 }
{\end{list}}

\tikzset{square matrix/.style={
    matrix of nodes,
    column sep=-\pgflinewidth, row sep=-\pgflinewidth,
    nodes={draw,
      minimum height=15pt,
      anchor=center,
      text width=13pt,
      align=center,
      inner sep=0pt
    },
  },
  square matrix/.default=2cm
}

\colorlet{trans}{blue!50}
\colorlet{trans2}{red!20}

\newcommand{\norm}[1]{\left\lVert #1 \right\rVert}

\DeclareMathOperator{\atan}{atan}

\DeclareMathOperator*{\argmin}{\arg\!\min}
\DeclareMathOperator*{\argmax}{\arg\!\max}

\newcommand{\nn}{neural network}

\newcommand{\abs}[1]{\left\lvert#1\right\rvert}
\newcommand{\btheta}{\pmb{\theta}}
\newcommand{\T}{\mathcal{T}}

\makeatletter
\newcommand*\bigcdot{\mathpalette\bigcdot@{.5}}
\newcommand*\bigcdot@[2]{\mathbin{\vcenter{\hbox{\scalebox{#2}{$\m@th#1\bullet$}}}}}
\makeatother

\allowdisplaybreaks

\begin{document}
 
\acrodef{pde}[PDE]{partial differential equation}
\acrodef{fe}[FE]{finite element}
\acrodef{fem}[FEM]{finite element method}
\acrodef{ad}[AD]{automatic differentiation}
\acrodef{dof}[DOF]{degree of freedom}
\acrodefplural{dof}[DOFs]{degrees of freedom}
\acrodef{agfe}[agFE]{aggregated finite element}
\acrodef{agfem}[AgFEM]{aggregated finite element method}
\acrodef{cg}[CG]{continuous Galerkin}
\acrodef{dg}[DG]{discontinuous Galerkin}
\acrodef{nn}[NN]{neural network}
\acrodef{pinn}[PINN]{Physics Informed Neural Network}
\acrodef{dls}[DLS]{Deep Least Square}
\acrodef{pdls}[PDLS]{Primal Deep Least Squares}
\acrodef{dgm}[DGM]{Deep Galerkin Method}
\acrodef{feinn}[FEINN]{Finite Element Interpolated Neural Network}
\acrodef{vpinn}[VPINN]{Variational Physics-Informed Neural Network}
\acrodef{ivpinn}[IVPINN]{interpolated variational physics-informed neural network}
\acrodef{elm}[ELM]{Extreme Learning Machine}
\acrodef{rfm}[RFM]{Random Feature Method}
\acrodef{kan}[KAN]{Kolmogorov-Arnold Network}
\acrodef{dgipnn}[DGIPNN]{DG Interpolated Piecewise NN}
\acrodef{autodiff}[AD]{automatic differentiation}
\acrodef{ffnn}[FFNN]{fully-connected feed-forward Neural Network}
\acrodef{mlp}[MLP]{multi-layer perceptron}
\acrodef{lhc}[LHC]{Latin hypercube sampling}
\acrodef{mc}[MC]{Monte Carlo}
\acrodef{qmc}[QMC]{Quasi-Monte Carlo}
\acrodef{amr}[AMR]{adaptive mesh refinement}
\acrodef{sciml}[SciML]{Scientific Machine Learning}
\acrodef{aq}[AQ]{adaptive quadrature}
\acrodef{wan}[WAN]{weak adversarial network}
\acrodefplural{wan}[WANs]{weak adversarial networks}
\acrodef{fdm}[FDM]{finite difference method}
\acrodef{fvm}[FVM]{finite volume method}
\acrodef{feinn}[FEINN]{finite element interpolated neural network}
\acrodef{amr}[AMR]{adaptive mesh refinement}
\acrodef{api}[API]{application programming interface}
\acrodef{gpu}[GPU]{Graphics Processing Unit}
\acrodef{cpu}[CPU]{Central Processing Unit}
\acrodef{gl}[GL]{Gauss--Legendre}
\acrodef{aq}[AQ]{adaptive quadrature}
\acrodef{d2rm}[D2RM]{Deep Double Ritz Method}
\acrodef{drm}[DRM]{Deep Ritz Method}

\title[Adaptive anisotropic composite quadratures for residual minimisation]{Adaptive anisotropic composite quadratures for residual minimisation in neural PDE approximations}
\author[S. Badia]{Santiago Badia$^{\dagger}$}
\address{$^\dagger$School of Mathematics\\Monash University\\Clayton\\Victoria 3800\\Australia}
\email{santiago.badia@monash.edu}
\author[K. Nori]{Kishore Nori$^{\dagger}$}
\email{sai.nori@monash.edu}
% \date{\today}  
 
\begin{abstract}
We study the role of numerical quadrature in residual-minimisation methods for neural network approximation of partial differential equations. We first present an abstract error framework that separates approximation, quadrature and optimisation errors, and derive a nonlinear Strang-type estimate quantifying how inaccuracies in the discrete loss affect the final approximation. Motivated by this analysis, we propose an anisotropic adaptive composite quadrature strategy that controls the relative quadrature error of the residual loss using richer reference quadratures and bisection-based refinement. We then introduce a refresh-based training methodology that rebuilds the quadrature only when an online error indicator exceeds a prescribed threshold, balancing accuracy and computational cost. Numerical experiments on a range of benchmark problems show that the proposed approach narrows the gap between training and reference losses, uses quadrature points more efficiently and delivers strong approximation accuracy relative to non-adaptive quadrature strategies.
\end{abstract} 

\maketitle 

\section{Introduction}

In this article, we study how numerical quadrature affects the quality of
approximations of \acp{pde} obtained with fully connected, feed-forward neural
networks, i.e., multi-layer perceptrons. We focus on
strong-form residual minimisation methods, in which the \ac{nn} is
trained by minimising a residual-based loss. In particular, we consider the
deep least squares method~\cite{Dissanayake1994, DGM2018, DLS-Cai-2020}, whose loss
involves integrals of the strong residual over the physical domain. 

Early works
approximated these integrals using uniform collocation points
\cite{Lagaris1998}, while later \ac{mc} and \ac{qmc} integration
became standard in what came to be known as \acp{pinn}~\cite{PINNs2019}.
Although such approaches have been successfully applied to a wide range of
\acp{pde}~\cite{UrbanPINNOpt2025}, the quality of
the induced quadrature is often treated as a secondary concern.
The central issue is that, if the quadrature is too coarse or poorly
distributed, the optimiser may reduce the discrete loss while the underlying
continuous residual loss remains poorly approximated. This mismatch between training and true losses can lead to severe overfitting with respect to the quadrature, even when the optimiser appears to converge. 

Because \acp{nn}
are global, highly nonlinear approximators, constructing accurate quadratures
is non-trivial~\cite{Feischl2025nnhard}. Fine uniform or quasi-uniform composite
quadrature rules that accurately integrate all relevant realisations of a
\ac{nn} architecture are computationally impractical for reasonably
sized networks. This makes adaptive control of the quadrature a central component of the numerical method.

Several adaptive ideas have been proposed for neural \ac{pde} solvers,
including fixed-interval Monte Carlo refresh~\cite{Wang2023PINNsGuide} and 
mesh-partition-based approaches~\cite{DLS-Cai-2020, Liu-Cai-Ramani-2023}; see
also~\cite{Toscano2025adaptivity} for a variational viewpoint on
residual-driven adaptivity. Focusing on adaptive quadratures,~\cite{RiveraPINNQuadrature2022} proposed a
1D adaptive mesh-refinement algorithm based on local integration error.
That approach already illustrates the potential of
quadrature adaptation in 1D driven by local error, but it is limited to single-step refinement,
does not explicitly enforce quadrature error control before training resumes,
and does not include coarsening, which can be important for redistributing
rather than hierarchically accumulating points~\cite{AdaptiveFEINNs2024}.
An alternative line of work exploits the \ac{nn} structure itself to
achieve adaptivity~\cite{MagueresseBadiaAdaptive2024, LiuAdaptiveFA2022}. These approaches can yield accurate
integration and interpretability, but extracting a domain-aware partition
becomes computationally expensive for deep networks and complex architectures
\cite{PirateNets2024} and may not capture other coefficients and terms in the
loss.

The contributions of this work are threefold. First, we formulate a relative perturbation framework for residual-minimisation losses and derive a nonlinear Strang-type estimate that separates approximation, quadrature and optimisation errors. Second, motivated by this estimate, we propose an anisotropic $h$-adaptive composite quadrature strategy that explicitly controls an estimator of the residual-loss quadrature error for the current network realisation. In particular, we advocate a bisection-based, anisotropic $h$-adaptive quadrature~\cite{GenzMalik1980} that adapts to the current \ac{nn} realisation and aims to control the quadrature error in the residual loss throughout training. Third, we introduce a refresh-based training procedure in which the quadrature is rebuilt only when an online training/reference loss indicator exceeds a prescribed threshold.

In contrast to existing approaches, our focus is not only on improving quadrature accuracy locally, but on explicitly controlling the relative perturbation of the residual loss throughout training. This leads to a training-aware adaptive quadrature strategy, in which the quadrature is dynamically rebuilt based on a global error indicator tied to the optimisation process. The theoretical results justify the importance of controlling quadrature perturbations in residual-minimisation losses, while the numerical experiments show that the proposed adaptive strategy can substantially reduce training/reference loss mismatch and improve accuracy compared with standard fixed quadrature and sampling strategies.

The outline of the paper is as follows. 
In~\Cref{sec:abstract-formulation}, we establish an abstract framework
for residual minimisation and discuss representative residual losses arising in
neural \ac{pde} approximation. In~\Cref{sec:manifold-approximation}, we
formalise minimisation over a nonlinear manifold, identify the relative
quadrature perturbation and optimisation errors and derive a nonlinear
Strang-type estimate for the resulting approximation. 
We also provide a weaker result that can still apply when considering variational crimes that are commonplace in the \acp{pinn} literature. 
In~\Cref{sec:adaptive-quad}, we present an adaptive quadrature framework leading
to an anisotropic, bisection-based $h$-adaptive composite quadrature with
non-nested primal and reference rules. In~\Cref{sec:methodology}, we introduce
a refresh-based training strategy that monitors an online quadrature error
indicator and rebuilds the quadrature only when needed. Finally, in
\Cref{sec:numexps} we report numerical experiments on a range of \ac{pde}
benchmark problems from scientific machine learning and \ac{amr}, comparing the proposed
approach with uniform, \ac{mc} and \ac{qmc}-based quadratures, including \ac{lhc}
and Halton sampling.

\providecommand{\bbR}{\mathbb{R}}
\providecommand{\bbN}{\mathbb{N}}
\providecommand{\bbZ}{\mathbb{Z}}
\providecommand{\bbQ}{\mathbb{Q}}
\providecommand{\bbC}{\mathbb{C}}

\DeclarePairedDelimiter{\ip}{\langle}{\rangle}
\DeclarePairedDelimiter{\Norm}{\lVert}{\rVert}
\DeclarePairedDelimiter{\tnorm}{\lvert\!\lvert\!\lvert}{\rvert\!\rvert\!\rvert}

\providecommand{\vecx}{\mathbf{x}}
\providecommand{\vecn}{\mathbf{n}}
\providecommand{\vtheta}{\pmb{\theta}}
\providecommand{\vbeta}{\pmb{\beta}}

\providecommand{\calN}{\mathcal{N}}
\providecommand{\calM}{\mathcal{M}}
\providecommand{\calF}{\mathcal{F}}
\providecommand{\calT}{\mathcal{T}}
\providecommand{\calQ}{\mathcal{Q}}
\providecommand{\calE}{\mathcal{E}}
\providecommand{\calL}{\mathcal{L}}
\providecommand{\calJ}{\mathcal{J}}
\providecommand{\calR}{\mathcal{R}}

\providecommand{\scrJ}{\mathscr{J}} 
\providecommand{\scrL}{\mathscr{L}}
\providecommand{\scrE}{\mathscr{E}}

\newcommand{\op}[1]{\mathcal{#1}}

\section{Abstract formulation} \label{sec:abstract-formulation}

In this section, we introduce the abstract residual minimisation framework used
throughout the paper. The framework is based on minimising a residual norm of
the problem at hand. We then illustrate this setting with representative
examples arising in neural approximation of \acp{pde}.

\subsection{Residual minimisation framework}

Let $X$ and $Y$ be Banach spaces and $T:X \rightarrow Y$. We consider the
following problem: given $\ell \in Y$, find $u \in X$ such that $Tu = \ell$. The
problem can be restated as a minimisation problem: find
\begin{equation}\label{eq:abstract_problem}
  u \in \argmin_{v \in X} \scrJ(v).
\end{equation}
Here $\scrJ(v) \doteq \| T v - \ell \|_Y^p$, with $p \geq 1$. Under the
following assumptions, the problem is well-posed~\cite{GuermondLp2004}.
\begin{assumption}\label{assumption:well-posed}
  The operator $T$ is linear, bounded above, i.e.,
  \begin{equation}\label{eq:boundedness}
    \|Tu\|_Y \leq C_T \|u\|_X \quad \forall u \in X,
  \end{equation}
  for some constant $C_T > 0$, bounded below, i.e.,
  \begin{equation}\label{eq:coercivity}
    \alpha \|u\|_X \leq \|Tu\|_Y \quad \forall u \in X,
  \end{equation}
  for some constant $\alpha > 0$, and its adjoint $T' : Y' \rightarrow X'$ is injective. 
\end{assumption}
\begin{lemma}
  Under \Cref{assumption:well-posed}, the solution $u$ of (\ref{eq:abstract_problem}) is unique and satisfies $\alpha \|u\|_X \leq \|\ell\|_Y$.
\end{lemma}
\begin{proof}
  Uniqueness holds if and only if $T$ is bijective. By the Banach closed-range
  theorem and the open mapping theorem, $T$ is bijective if and only if it is
  bounded below and its adjoint is injective, which hold by \Cref{assumption:well-posed}. The bound by the data follows from
 ~\eqref{eq:coercivity} and the identity $Tu = \ell$.
  \end{proof} 
  
  Under \Cref{assumption:well-posed}, if $Y$ is a Hilbert space, then the problem can be restated in a least-squares form:
  \[
    u \in X \ : \ (Tu,Tv)_Y = (\ell,Tv)_Y \quad \forall v \in X.
  \]
  This problem is the first-order optimality condition for the minimisation of $\scrJ^{2/p}$, which is differentiable (and strictly convex). Since $\scrJ^{2/p}$ and $\scrJ$ share the same minimiser (the map $t \mapsto t^{2/p}$ is strictly increasing for $p>0$), the two problems are equivalent.  
  
We now illustrate the abstract setting with two representative \acp{pde}. Let
$\Omega \subset \mathbb{R}^d$ be a bounded domain with Lipschitz boundary
$\partial \Omega$. We consider the Poisson problem and the linear transport
problem on $\Omega$. 

\paragraph{\underline{\textbf{Poisson problem}}} Consider the Poisson problem
$-\Delta u = f_\Omega$ in $\Omega$ with Dirichlet boundary conditions
$u = g$ on $\partial \Omega$. This problem admits a unique solution for
$X = H^1(\Omega)$ and
$Y = H^{1}(\Omega)' \times H^{1/2}(\partial \Omega)$. We set
$\ell \doteq (f_\Omega, g)$ and define the operator
$T u \doteq (-\Delta u, u|_{\partial \Omega})$. The loss functional is then
given by:
\begin{equation}
  \scrJ(v)^{2/p} = \| T v - \ell \|_Y^{2} = \| -\Delta v - f_\Omega \|_{H^{1}(\Omega)'}^2 + \| v|_{\partial \Omega} - g \|_{H^{1/2}(\partial \Omega)}^2.
\end{equation}
Assuming elliptic regularity, the problem admits a unique solution in the
smoother functional setting
$X = H^2(\Omega)$ and
$Y = L^2(\Omega) \times H^{3/2}(\partial \Omega)$. In that case, the loss
functional becomes:
\begin{equation}
  \scrJ(v)^{2/p} = \int_{\Omega} (-\Delta v - f_\Omega)^2 d\Omega + \gamma \| v|_{\partial \Omega} - g\|_{H^{3/2}(\partial \Omega)}^2.
\end{equation}
where $\gamma > 0$ is a numerical parameter.
We note that the first loss functional involves dual norms, which are hard to
compute numerically. This has motivated the use of the ill-posed settings
discussed in~\Cref{sub:variational_crimes}. 

\paragraph{\underline{\textbf{Linear transport problem}}} 
Consider a pure transport problem
\[
{\pmb{\beta} \cdot \nabla} u = f_\Omega \quad \text{in } \Omega,
\] 
where $\pmb{\beta}$ is a solenoidal vector field. We impose inflow boundary conditions on 
$$\Gamma_- = \{ x \in \partial \Omega : \pmb{\beta} \cdot \mathbf{n}(x) <0 \},$$ 
where $\mathbf{n}$ is the outward unit normal to $\partial \Omega$. We set 
$$X = L^2_{\pmb{\beta}}(\Omega) \doteq \left\{ u \in L^2(\Omega) \, : \, \pmb{\beta}\cdot \nabla u \in L^2(\Omega)\right\}$$ 
and $Y = L^2(\Omega) \times L^2_{|\pmb{\beta}\cdot \pmb{n}|}(\partial \Omega)$, where $L^2_{|\pmb{\beta}\cdot \pmb{n}|}(\Gamma_-)$ is the weighted $L^2$ space on the inflow boundary with respect to $| \pmb{\beta} \cdot \pmb{n} |$. The operator $T$ is defined as $(\pmb{\beta} \cdot \nabla u, u|_{\Gamma_-})$ and $f$ is defined as $(f_\Omega, g)$, where $f_\Omega$ is the source term and $g$ is the inflow boundary condition. As a result, the loss functional is given by:
\begin{equation}
  \scrJ(v)^{2/p} = \int_{\Omega} (\pmb{\beta} \cdot \nabla v - f_\Omega)^2 d\Omega + \gamma \int_{\Gamma_-} (v|_{\Gamma_-} - g)^2 |\pmb{\beta} \cdot \mathbf{n}| d\Gamma_-.
\end{equation}
\section{Functional manifold minimisation} \label{sec:manifold-approximation}

In this section, we analyse the minimisation of a nonnegative loss functional
over a manifold of functions. Since the exact computation of the functional may
not be possible, we study the effect of the error introduced in its
evaluation. Our focus is the relative perturbation setting relevant to the
residual minimisation problems considered in this paper, together with the
error introduced by the optimiser when approximating the minimiser.

Let us consider a family of maps $\mathcal{R}_{\#}: \mathbb{P}_{\#} \rightarrow
\mathcal{N}_{\#} \subset X$, where $\mathbb{P}_{\#}$ is a finite-dimensional
Euclidean space and $\mathcal{N}_{\#}$ is a manifold embedded in $X$,
parameterised by hyper-parameters $\# \in \mathcal{H}$. $\mathcal{R}_{\#}$ can
be the \ac{nn} realisation map that assigns to each parameter vector
$\pmb{\theta} \in \mathbb{P}_{\#}$ a function $\mathcal{R}_{\#}(\pmb{\theta})
\in \mathcal{N}_{\#} \subset X$ and $\mathcal{H}$ can include standard
hyper-parameters that encode the \ac{nn} architecture, such as the
number of layers, the number of neurons per layer, the activation function,
etc. The goal is to minimise a nonnegative loss functional $\scrJ$ over the
manifold $\mathcal{N}_{\#}$. In practice, $\mathcal{N}_{\#}$ is neither convex
nor closed in $X$ and uniqueness does not hold in general; this is the case
when $\mathcal{N}_{\#}$ is a \ac{nn}. Instead, we define the argmin
space $\mathcal{M}_{\#}(\scrJ)$ of minimisers of $\scrJ$ over
$\operatorname{cl}(\mathcal{N}_{\#})$:  
\[
  \mathcal{M}_{\#}(\scrJ) \;=\; \argmin_{u\in\operatorname{cl}(\mathcal{N}_{\#})} \scrJ(u).
\]
We have:
\begin{equation}\label{eq:best-approx-error}
  \scrJ(v) \leq \inf_{w \in \mathcal{N}_{\#}} \scrJ(w),
  \qquad \forall v \in \mathcal{M}_{\#}(\scrJ).
\end{equation}
In this section, we treat $\scrJ$ as a general loss functional. Later, we will
particularise the discussion to the case in which $\scrJ$ is some power of a
norm of the residual of the problem. In practice, the exact
evaluation of $\scrJ$ is often unavailable because the loss contains integrals
or dual norms that must be approximated numerically. In the present work, the
main source of error is the replacement of these integrals by numerical
quadrature, which leads to a discrete functional $\tilde{\scrJ}$. Our focus
below is therefore the case in which the functional is approximated and the
optimisation algorithm attains the minimum only up to a prescribed
tolerance.

\subsection{Relative perturbation error} \label{subsec:rel-perturbation-error}

For semi-positive definite functionals, such as those arising in residual
minimisation approaches, the first class of errors we consider is the inexact
evaluation of the loss functional $\scrJ$. In our setting, these
perturbations arise mainly from approximating the integrals in the loss by
quadrature rules and, when relevant, from inexact evaluation of dual norms.
They can be quantified by a relative error,
\begin{equation}\label{eq:rel_int_error}
  |\scrJ(v) - \tilde{\scrJ}(v)| \leq \varepsilon_{\mathrm{rel}} \scrJ(v), %\quad \forall v \in \mathcal{N}_{\#},
  \tag{A1}
\end{equation}
where $\varepsilon_{\mathrm{rel}}$ quantifies the relative error introduced by
the approximated functional. In practice, enforcing~\eqref{eq:rel_int_error} uniformly over the entire manifold $\mathcal{N}_{\#}$
 is not feasible. Instead, the adaptive quadrature strategy introduced later aims to control this perturbation condition locally at the network realisations encountered during training.

\subsection{Optimisation error}

Assuming that we replace the minimisation of the exact functional $\scrJ$ by
the minimisation of an approximated functional $\tilde{\scrJ}$, the second
error that we face in practice is the one related to the (non-convex)
optimisation of the functional $\tilde{\scrJ}$ over the manifold
$\mathcal{N}_{\#}$. In general, we cannot guarantee that the optimisation
algorithm converges to a global minimiser of $\tilde{\scrJ}$ over
$\mathcal{N}_{\#}$, due to the non-convex loss landscape exhibiting multiple
local minima. Instead, we denote by $\tilde{u}_{\#} \in \mathcal{N}_{\#}$ the
output of the optimisation algorithm. Let $\mathcal{M}_{\#}(\tilde{\scrJ})$ be
the set of minimisers of $\tilde{\scrJ}$ over
$\operatorname{cl}(\mathcal{N}_{\#})$. We assume that:
\begin{equation}\label{eq:optimisation-error}
  \min_{v \in \mathcal{M}_{\#}(\tilde{\scrJ})} | \tilde{\scrJ}(\tilde{u}_{\#}) - \tilde{\scrJ}(v) | \leq \varepsilon_{\mathrm{opt}}, \tag{A2}
\end{equation}
where $\varepsilon_{\mathrm{opt}}$ quantifies the optimisation error, i.e., how
close the output of the optimisation algorithm is to a global minimiser of
$\tilde{\scrJ}$ over $\operatorname{cl}(\mathcal{N}_{\#})$. 

When the perturbation in $\tilde{\scrJ}$ relative to the exact loss
functional $\scrJ$ is measured as in
\Cref{subsec:rel-perturbation-error}, we can establish the following bound on
the value of the functional at the output of the optimisation algorithm.
\begin{lemma} \label{lemma:nonlinear-strang-rel-pert}
  Let us assume that the minimisation problem (\ref{eq:abstract_problem}) admits a unique solution $u \in X$. Let $\tilde{u}_{\#}$ be the output of the optimisation algorithm. If (\ref{eq:rel_int_error}) and (\ref{eq:optimisation-error}) hold, we have:
  \[
    \scrJ(\tilde{u}_{\#}) \leq  \frac{1}{1 - \varepsilon_{\mathrm{rel}}} \left((1 + \varepsilon_{\mathrm{rel}}) \inf_{w \in \mathcal{N}_{\#}} \scrJ(w) + \varepsilon_{\mathrm{opt}} \right).
  \]
\end{lemma}
\begin{proof}
First, we readily check that (\ref{eq:rel_int_error}) implies:
\begin{equation}\label{eq:rel_int_error_bounds}
  (1 - \varepsilon_{\mathrm{rel}}) \scrJ(v) \leq \tilde{\scrJ}(v) \leq (1 + \varepsilon_{\mathrm{rel}}) \scrJ(v), \quad \forall v \in \mathcal{N}_{\#}.
\end{equation}
Using the definition of the training error in (\ref{eq:optimisation-error}) and $\mathcal{M}_{\#}(\tilde \scrJ)$ and the upper bound in (\ref{eq:rel_int_error_bounds}), we obtain:
\begin{equation} \label{eq:rel-pert-abstract-strang}
  \tilde{\scrJ}(\tilde{u}_{\#})
  \leq \tilde{\scrJ}(v) + \varepsilon_{\mathrm{opt}} 
  \leq (1 + \varepsilon_{\mathrm{rel}}) \scrJ(v) + \varepsilon_{\mathrm{opt}}, 
\end{equation}
for some $v \in \mathcal{M}_{\#}(\tilde{\scrJ})$ that attains the minimum in
(\ref{eq:optimisation-error}). Invoking (\ref{eq:best-approx-error}) and the
lower bound in (\ref{eq:rel_int_error_bounds}) on the LHS, we prove the result.
\end{proof}
The estimate is meaningful only when $\varepsilon_{\mathrm{rel}} < 1$, which corresponds to a regime in which the quadrature-induced perturbation does not dominate the loss functional.
\subsection{Nonlinear Strang's lemma} \label{subsec:nonlinear-strang-lemma}

We now particularise the abstract result above to the residual minimisation
setting. Here, \emph{nonlinear} refers to approximation on a
manifold $\mathcal{N}_{\#}$ rather than on a vector space. This yields a
nonlinear Strang-type estimate that separates approximation, quadrature and
optimisation errors. We then discuss a related interpretation for \ac{pinn}-like
methods based on variational crimes.

The following nonlinear version of Strang's Lemma applies to the quadrature
approximation of the residual functional introduced above. We recall the
notation: $u$ is the exact solution of the problem
(\ref{eq:abstract_problem}) and $\tilde{u}_{\#}$ is the output of the
optimisation algorithm. 
\begin{corollary} \label{corollary:nonlinear-strang-rel}
  Let us assume that the continuity (\ref{eq:boundedness}) and coercivity
  (\ref{eq:coercivity}) of $T$ hold under $X$, and the problem admits a unique
  solution $u \in X$. If the approximate loss functional ($\tilde{\scrJ}$)
  obeys assumption (\ref{eq:rel_int_error}) and the output $\tilde{u}_{\#}$ of
  the optimisation algorithm satisfies (\ref{eq:optimisation-error}), then we have:
\begin{align}
  \| \tilde{u}_{\#} - u\|_X  
  & \leq \frac{1}{\alpha} \left(\frac{1}{\left(1 - \varepsilon_{\mathrm{rel}}\right)} \left( \left(1 + \varepsilon_{\mathrm{rel}}\right) \inf_{w \in \operatorname{cl}(\mathcal{N}_{\#})} \scrJ(w) + \varepsilon_{\mathrm{opt}} \right)\right)^{1/p} \\ 
  & \leq  \frac{1}{\alpha} \left(\frac{1}{\left(1 - \varepsilon_{\mathrm{rel}}\right)} \left( \left(C_T\right)^p \left(1 + \varepsilon_{\mathrm{rel}}\right) \inf_{w \in \operatorname{cl}(\mathcal{N}_{\#})} \|u - w\|^p_X + \varepsilon_{\mathrm{opt}} \right)\right)^{1/p}.
\end{align} 
\end{corollary}
\begin{proof}
  The setting and assumptions are the same as in
 ~\Cref{lemma:nonlinear-strang-rel-pert}, so we start from the assertion of
  that lemma. Since $\scrJ(v) = \norm{T(v -u)}^p_Y$ because $T u = \ell$, applying
  the coercivity estimate~\eqref{eq:coercivity} on the left-hand side yields
  the first bound. Applying the continuity estimate~\eqref{eq:boundedness} on
  the right-hand side then gives the final bound.
\end{proof}
The estimate shows that, in stable residual-minimisation settings, the relative quadrature error enters the bound through multiplicative factors that scale the best-approximation error, rather than as a purely additive perturbation. 
In particular, poor quadrature can significantly amplify approximation errors by distorting the loss functional and therefore the optimisation landscape on which the \ac{nn} is trained. In addition, the optimisation error contributes additively. In particular, this motivates controlling the quadrature error in a relative sense, as in~\eqref{eq:rel_int_error}, so that the distortion of the loss functional remains proportional to its magnitude during training.

\subsection{PINN-like methods}\label{sub:variational_crimes}

\ac{pinn}-like methods often use the residual-minimisation framework together with
variational crimes to avoid the explicit computation of dual norms in the loss
functional. A common approach is to replace $Y$ by a larger space
$\overline{Y}$ endowed with a weaker norm $\|\cdot\|_{\overline{Y}}$, and to
view the operator as $T: X \rightarrow \overline{Y}$. In this relaxed setting,
one cannot in general prove~\eqref{eq:boundedness} and~\eqref{eq:coercivity}
with the weaker norm. Nevertheless, one can define the weaker (semi-)norm
$\| \cdot \|_{0,X} \doteq \| T \cdot \|_{\overline{Y}}$, for which analogous
bounds are immediate. If the data in~\eqref{eq:abstract_problem} satisfy
$\ell \in \overline{Y}$, then the solution $u \in X$ of the original problem is
also a minimiser of the relaxed functional
$\| T v - \ell \|^p_{\overline{Y}}$ over $X$, although uniqueness may fail.

Based on the assumption (\ref{eq:rel_int_error}) and
(\ref{eq:optimisation-error}), beginning with
\Cref{lemma:nonlinear-strang-rel-pert}  and proceeding similarly to the proof
of the nonlinear Strang's lemma, in~\Cref{corollary:nonlinear-strang-rel}, we
obtain the following corollary.
\begin{corollary}
  Let $u$ be the solution of~\eqref{eq:abstract_problem} for data
  $\ell \in \overline{Y}$. Assume that $u$ is also the unique minimiser of the
  relaxed functional $\| T v - \ell \|_{\overline{Y}}$ over $X$. Let
  $\tilde{u}_{\#}$ be the output of the optimisation algorithm. Then we have:
  \[
    \| T\tilde{u}_{\#} - Tu\|^p_{\overline{Y}} \leq \frac{1}{1 - \varepsilon_{\mathrm{rel}}} \left( \left(1 + \varepsilon_{\mathrm{rel}}\right) \inf_{w \in X} \|T u - T w\|^p_{\overline{Y}} + \varepsilon_{\mathrm{opt}} \right).
  \]
\end{corollary}
\begin{example}[Poisson problem]\label{ex:poisson-relaxed-loss} Let us consider the statement of the Poisson
problem introduced above. Consider the space $\overline{Y} = L^2(\Omega) \times
L^2(\partial \Omega)$. The relaxed functional (with $p = 2$) is given by:
  \[
   \int_{\Omega} (\Delta v + f_\Omega)^2 d\Omega + \gamma \int_{\partial \Omega} (v - g)^2 ds, 
  \]
which is easy to compute using numerical integration. However, one cannot bound $\|u\|_{H^2(\Omega)}$ by $\|\Delta u\|_{L^2(\Omega)} + \| u\|_{L^2(\partial \Omega)}$. Despite the ill-posedness of the problem, this is the common setting in \ac{pinn} formulations.
\end{example}

\section{Adaptive quadratures}\label{sec:adaptive-quad}

In this section, we discuss an effective practical strategy for satisfying
\eqref{eq:rel_int_error}. More precisely, given the perturbation assumption
\eqref{eq:rel_int_error} introduced in~\Cref{subsec:rel-perturbation-error},
our goal is to formulate a local quadrature-estimator condition, denoted below
by~\eqref{eq:local_quad_reliability}, from which the relative perturbation
bound~\eqref{eq:rel_int_error} for the quadrature-based approximation
$\tilde{\scrJ}$ of the exact loss functional $\scrJ$ follows for the current
\ac{nn} realisation during training. For standard numerical methods
posed on vector spaces, e.g., finite element methods, one can often construct a
fixed quadrature rule with sufficiently high polynomial accuracy. In those
settings, the relevant integrands involve piecewise polynomials and known
analytical functions such as forcing terms, boundary conditions and physical
coefficients.

For discretisation on nonlinear manifolds, such an approach is no longer
feasible. Different neural network realisations can exhibit different
localised features depending on the parameters. As a result, it is
impractical to construct a fixed quadrature rule that satisfies
(\ref{eq:rel_int_error}) uniformly over all relevant realisations of the
\ac{nn}, especially for deep architectures with many parameters. This
creates a form of overfitting in which the optimisation may find a solution
$\tilde{u}_{\#}$ that minimises the approximated functional $\tilde \scrJ$
while $\scrJ(\tilde{u}_{\#})$ remains large. Accordingly, we consider an
adaptive quadrature rule that depends on the current realisation of the neural
network. More precisely, we define
$\tilde{\scrJ}$ by replacing the exact domain and boundary integrals in
$\scrJ$ with composite numerical quadrature. Hence, for each
$v \in \mathcal{N}_{\#}$, the quadrature used to evaluate $\tilde{\scrJ}(v)$
is adapted to that current realisation rather than chosen uniformly over the
whole manifold $\mathcal{N}_{\#}$, which would be prohibitively expensive or infeasible.

The key idea is to pair a primal quadrature rule with a richer quadrature that
serves as a reference and yields cell-wise error estimates. This reference
quadrature plays a role analogous to a validation set in machine learning,
while also guiding local improvements to the training quadrature itself.
Unlike standard machine learning settings, we can, in principle, draw
arbitrarily many quadrature points, limited only by computational cost.
The local cell-wise estimates are then aggregated into a global
indicator controlling the perturbation of each integral contribution to the
loss. This local-to-global construction is what allows the adaptive procedure
to target assumption~\eqref{eq:rel_int_error} in a computable way.

As a model case, let us consider a term of $\scrJ$ that can be written as an
integral over a domain $\omega$ raised to some power. We denote this
contribution by $\scrJ_\omega(v)$ for $v \in \mathcal{N}_{\#}$ and write
$I_\omega(f) \doteq \int_{\omega}^{} f$ for a nonnegative integrand
$f(v) \in L^{1}(\omega)$, and $\scrJ_\omega(v) = I_\omega(f)^{1/p}$ for
$p \geq 1$.

\subsection{Adaptive composite quadratures}\label{subsec:anisotropic-adaptive-quadratures}
We consider the automatic generation of an anisotropic $h$-adaptive composite quadrature for the computation of $I(f)$. We start with a coarse partition $\mathcal{T}$ of $\omega$, so that
\[
I(f) = \sum_{K \in \mathcal{T}} I_K(f),
\qquad
I_K(f) \doteq \int_K f(x).
\]
Let $\widehat K$ 
be a fixed reference element and let
  $F_K : \widehat K \to K \subset \mathbb R^d$ be a $C^{p+1}$ diffeomorphism
  onto the physical element $K \in \mathcal{T}$, with $J_K(\hat x) \doteq \det DF_K(\hat x)$.
  Let $\widehat Q$ be an operator that approximates the integral of a function on $\widehat K$ and is exact on
  $\mathbb P_p(\widehat K)$, and define the quadrature approximation on $K$ via the pullback of the integral to the reference element $\widehat K$ as follows:
  \begin{equation}\label{eq:def-QK}
    Q_K(f)
    \doteq \widehat Q\bigl( f\circ F_K \, J_K \bigr)
    \qquad \forall f \in L^1(K).
  \end{equation}
This defines an approximation of the integral $I_K(f)$. Classical error
estimates generally assume standard shape-regularity assumptions (see, e.g.,
\cite[Lemma 30.9]{Ern-Guermond-FE-II}); however, weaker assumptions regarding
shape regularity may suffice given the anisotropic interpolation error
estimates (cf.~\cite{Acosta2006interpolation}). The following result provides
an estimate on anisotropic partitions. 
% \begin{lemma}[Quadrature error on anisotropic partitions]\label{lem:quadrature-general}
%   Let
%   $F_K : \widehat K \to K \subset \mathbb R^d$ be a $C^{p+1}$ diffeomorphism
%   onto the physical element $K$, with $J_K(\hat x) \doteq \det DF_K(\hat x)$.
%   Assume that the geometric maps satisfy the following uniform regularity
%   estimate: there exists a constant $C_{\mathrm{geo}}>0$, independent of $K$,
%   such that, for all $f \in H^{p+1}(K)$,
%   \begin{equation}\label{eq:mapping-estimate}
%     \bigl\| f\circ F_K \, J_K \bigr\|_{H^{p+1}(\widehat K)}
%     \;\le\;
%     C_{\mathrm{geo}}\, h_K^{p+1}\, \|f\|_{H^{p+1}(K)},
%   \end{equation}
%   where $h_K \doteq \operatorname{diam}(K)$.
%   Then there exists a constant $C>0$, depending only on $\widehat K$
%   and $\widehat Q$, but not on $K$ or $f$, such that for all $f\in H^{p+1}(K)$,
%   \begin{equation}\label{eq:local-quad-error}
%     \left|
%       \int_K f(x)\,dx - Q_K(f)
%     \right|
%     \;\le\;
%     C h_K^{p+1}\, \|f\|_{H^{p+1}(K)}.
%   \end{equation}
%   \end{lemma}
  
  \begin{lemma}[Anisotropic quadrature error for axis-aligned boxes]\label{lem:anisotropic-box-quad-error}
  Let $\widehat K = \prod_{i=1}^d [0,1]$ and $K = \prod_{i=1}^d [a_i, a_i + h_i]$ be axis-aligned boxes in $\mathbb{R}^d$, and let $F_K : \widehat K \to K$ be the affine map $F_K(\hat x) = (a_1 + h_1 \hat x_1, \ldots, a_d + h_d \hat x_d)$. Let $Q_K$ be the tensor-product Gaussian quadrature of degree $p$ in each direction, and $f \in H^{p+1}(K)$. Then there exists a constant $C > 0$, depending only on $d$ and $p$, such that
  \begin{equation}\label{eq:anisotropic-box-quad-error}
      \left|
        \int_K f(x)\,dx - Q_K(f)
      \right|
      \;\le\;
      C \sum_{|\alpha|=p+1} \prod_{i=1}^d h_i^{\alpha_i+1} \| D^\alpha f \|_{L^\infty(K)}. 
  \end{equation}
  \end{lemma}
  \begin{proof}
  The proof follows by applying the standard tensor-product Gaussian quadrature error estimate in each direction, using the affine map $F_K$ and the multi-dimensional Taylor expansion. The error is a sum over multi-indices $\alpha$ with $|\alpha| = p+1$, and each term involves the product of the mesh sizes $h_i$ raised to $\alpha_i+1$ and the corresponding mixed derivative of $f$. 
  \end{proof}

We would like to compute the cell-wise quadrature error
$\abs{I_K(f) - Q_K(f)}$, which is not available in general. We therefore
estimate it by considering a richer quadrature $\widehat Q'$ that is exact on
$\mathbb P_{p'}(\widehat K)$ and the corresponding $Q'_K$ obtained via
pullback, usually referred to as the reference quadrature. The
absolute integration error incurred by $Q_K$ can then be estimated by
$\delta_K \coloneqq \abs{Q'_K - Q_K}$ for all $K \in \mathcal{T}$, with global
estimator $\delta \coloneqq \sum_{K \in \mathcal{T}} \delta_K$. We assume that this estimate is efficient
and reliable, i.e., there exist constants  $c_{\mathrm{eff}} > 0$ and
$C_{\mathrm{rel}} > 0$ such that
\begin{equation}\label{eq:local_quad_reliability}
  c_{\mathrm{eff}} \delta_K \leq \abs{I_K(f) - Q_K(f)} \leq C_{\mathrm{rel}} \delta_K,
  \quad \forall K \in \mathcal{T}.
  \tag{A1$'$}
\end{equation}
This assumption is standard in adaptive quadrature, where the integration error
is estimated by comparing quadrature rules of different orders. Its
reliability, however, typically relies on regularity properties of the
integrand and is not guaranteed in full generality. In our setting, it should
therefore be understood as an idealised condition that guides the design of the
adaptive strategy. The remainder of the construction explains how
\eqref{eq:local_quad_reliability}, together with the adaptive stopping
criterion, provides a mechanism to enforce the global perturbation estimate
\eqref{eq:rel_int_error}.

Conceptually, we mark cells with the largest quadrature errors and refine the
partition accordingly. The new quadrature rule is then defined by applying the
same local quadrature rule on the refined partition
$\mathcal{T}_{\mathrm{ref}}$. Assigning $\mathcal{T} \leftarrow
\mathcal{T}_{\mathrm{ref}}$ and repeating this process yields increasingly
accurate approximations of $I_\omega(f)$ as the cell diameters decrease. 
\begin{lemma}\label{lemma:A1prime-implies-A1}
Assume that the local estimator-control condition~\eqref{eq:local_quad_reliability}
holds and that the adaptive loop terminates at the current network realisation
$v_{\#}$ once
\begin{equation}\label{eq:rel-stopping-criterion}
\delta^{1/p} \leq \xi \tilde \scrJ_\omega(v_{\#}),
\end{equation}
is satisfied for a prescribed relative tolerance $\xi > 0$. If $\xi$ is small enough so that $C_{\mathrm{rel}}^{1/p}\xi < 1$, then the perturbation bound~\eqref{eq:rel_int_error} holds for 
\[
\varepsilon_{\mathrm{rel}}(v_{\#})
\doteq
\frac{C_{\mathrm{rel}}^{1/p}\xi}{1-C_{\mathrm{rel}}^{1/p}\xi}.
\]
\end{lemma}
\begin{proof}
Summing~\eqref{eq:local_quad_reliability} over $K \in \mathcal{T}$, using
the concavity of $x \mapsto x^{1/p}$ and invoking~\eqref{eq:rel-stopping-criterion}, we obtain
\[
\abs{\scrJ_\omega(v_{\#})-\tilde \scrJ_\omega(v_{\#})}
\leq
C_{\mathrm{rel}}^{1/p}\delta^{1/p}.
\leq
C_{\mathrm{rel}}^{1/p}\xi\,\tilde \scrJ_\omega(v_{\#}).
\]
Hence
\[
\tilde \scrJ_\omega(v_{\#})
\leq
\scrJ_\omega(v_{\#})
+
\abs{\scrJ_\omega(v_{\#})-\tilde \scrJ_\omega(v_{\#})}
\leq
\scrJ_\omega(v_{\#})
+
C_{\mathrm{rel}}^{1/p}\xi\,\tilde \scrJ_\omega(v_{\#}),
\]
so $(1-C_{\mathrm{rel}}^{1/p}\xi)\tilde \scrJ_\omega(v_{\#}) \le \scrJ_\omega(v_{\#})$.
Combining this with the previous estimate yields
\[
\abs{\scrJ_\omega(v_{\#})-\tilde \scrJ_\omega(v_{\#})}
\leq
\frac{C_{\mathrm{rel}}^{1/p}\xi}{1-C_{\mathrm{rel}}^{1/p}\xi}\,\scrJ_\omega(v_{\#}),
\]
which is exactly~\eqref{eq:rel_int_error} for this contribution.
\end{proof}
In particular, for small $\xi$ , the resulting relative perturbation is of order $\mathcal{O}(\xi)$.
\begin{remark}
In practice, one also needs an absolute tolerance $\rho$ on $\delta$ to account
for rounding errors and as a safety net for the case in which the integral is
very small. Accordingly, the adaptive loop is stopped once either the relative
criterion~\eqref{eq:rel-stopping-criterion} or the absolute criterion is met,
that is,
\begin{equation}\label{eq:abs-stopping-criterion}
\delta^{1/p} \leq \max(\rho, \xi \tilde \scrJ_\omega(v_{\#})).
\end{equation}
\end{remark}
\begin{remark}
We are implicitly assuming here, and elsewhere in the discussion, that
$\tilde \scrJ_\omega(v_{\#}) \geq 0$. Indeed, the exact contribution
$\scrJ_\omega(v_{\#})$ is a norm of the corresponding residual and is therefore
nonnegative. Likewise, if the primal quadrature rule has nonnegative weights,
then $\tilde \scrJ_\omega(v_{\#}) \geq 0$ as well, and the assumption is satisfied. 
In our experiments, we have only considered positive-weight quadratures, which we have observed experimentally to be essential for the success of the neural solvers.
\end{remark}
\subsection{On the choice of the refined quadrature}

A common and highly effective choice for adaptive quadrature is to use nested rules, in which the refined quadrature reuses the quadrature points of the primal rule. A standard example is provided by Gauss-Kronrod rules, where additional points are added to a Gauss rule to increase its degree. However, in the context of \ac{nn} training, nested quadrature-based adaptive approaches have major disadvantages. First, because the low-order and high-order rules reuse points, they are less effective at detecting overfitting scenarios, which are common in \ac{nn} training and are precisely one of the reasons for using adaptive quadrature in the first place. Second, making such quadratures efficient in terms of the total number of evaluation points per cell often leads to negative weights. In \ac{nn} training, this can lead to severe instabilities, since the loss may become negative and unbounded from below, thereby encouraging the optimiser to overfit by driving the solution towards large negative blow-ups.

To remedy these drawbacks while retaining the (anisotropic) bisection refinement strategy, we propose using polynomial quadratures with positive weights. Tensor-product Gauss-Legendre rules are of particular interest here; they are accurate for tensor-product polynomial spaces up to degree $2k - 1$, where $k$ is the number of points in each spatial variable (hereafter referred to as the quadrature order). We also explore (quasi-) optimal polynomial quadratures, such as the Witherden-Vincent~\cite{WitherdenVincent2015} and Xiao-Gimbutas~\cite{XiaoGimbutas2010} families. These similarly feature positive weights, with only minor exceptions in the Xiao-Gimbutas family. Ensuring positive weights is also crucial from an analysis perspective, as it guarantees that the quadrature-based approximations of norms and inner products remain well-defined, and thus the functional $\tilde\scrJ$.

\subsection{Anisotropic refinement
rule}\label{subsec:anisotropic-refinement-rules}

The fundamental idea of the $h$-adaptive composite quadrature approach is to begin
with a coarse partition of the integration domain $\TT_{\texttt{init}}$ and to
refine the cell with the largest error estimate among $(\delta_K)_{K \in \T}$,
repeating this process on the updated partition until the stopping criterion is
met. Instead of using isotropic refinement of the marked cells, we consider
anisotropic refinement in order to reduce the number of quadrature points
generated after refinement and to mitigate the curse of dimensionality.\footnote{We note that the partition is required only for integration, which is why it is straightforward to handle highly anisotropic partitions. As a result, this approach offers advantages over the adaptive finite element interpolation strategies in~\cite{AdaptiveFEINNs2024,CompatibleFEINNs}, which usually involve mesh constraints such as 2:1 balance and shape regularity.} It is noteworthy
that the bisection refinement strategy results in far fewer partitions to
achieve the same relative tolerance and is hence computationally cheaper.
% (see~\cref{subsec:uniform-vs-bisection} for comparisons). 

However, this requires a strategy for selecting the bisection plane for a
``marked'' cell. We formalise this strategy as follows. Assume that
$\widehat K$ is a $d$-cube, and let its coordinate axes be indexed by
$j \in \{1,2, \ldots, d\}$. To each axis we assign a positive value
$\zeta^j_K$ that measures the expected reduction in the quadrature error
estimate on $K$ when it is bisected along the central plane orthogonal to
direction $j$. When $K$ is marked for refinement, the quantities $\zeta^j_K$
are computed for all $j \in \{1, \ldots, d\}$, and the axis corresponding to
the largest value of $\zeta^j_K$ is selected for bisection.

Ideally, one would bisect the element $K$ along each central axis and compute
the resulting reduction in the quadrature error estimate for every candidate
partition. These reductions define the ideal quantities
$\widehat \zeta^j_K$, which determine the optimal bisection axis.
\begin{equation} \label{eq:ideal-dir-estimate}
\widehat \zeta^j_K \doteq \delta_{K^{+}_j} + \delta_{K^{-}_j}, \quad \forall j \in \{1,2, \ldots, d\}
\end{equation}
where $K^{+}_j$ and $K^{-}_j$ are child cells obtained by bisecting $K$
along the plane corresponding to axis $j$. (In the algorithms, the child cells
obtained from $K$ are denoted by the set $\mathcal{C}(K)$.)
However, computing the ideal directional error estimates,
$\{\widehat{\zeta}^j_K\}_{j=1}^d$, can become computationally prohibitive as
the problem dimensionality $d$ increases. To ensure computational efficiency,
it is necessary to derive alternative, efficient surrogate indicators,
$\{\zeta^j_K\}_{j=1}^d$. These cheaper equivalents must serve the same purpose:
providing a reliable heuristic for the reduction in quadrature error along each
axis without requiring the full evaluation of candidate sub-quadratures.

Alternatively, the role of $\widehat{\zeta}^j_K$ can be reinterpreted in terms
of the difficulty of estimating the integral along direction $j$, as measured
by $\zeta^j_K$. Refinement along the corresponding central axis plane is then
expected to provide the best reduction in the quadrature error estimate among
all possible bisections. This can be formulated through the implication
\begin{equation}
  \zeta^j_K \leq \zeta^l_K \implies \widehat{\zeta}^j_K \leq \widehat{\zeta}^l_K, \quad \forall j, l \in \{1,2, \ldots, d\}.
\end{equation}
This leads to the heuristic approach of Genz-Malik
\cite{GenzMalik1980}, based on the work of~\cite{VDooren1976}, in which
$\zeta^j_K$ is defined through fourth differences,
that is, symmetric finite-difference approximations of the fourth derivative of
the integrand. These quantities are intended to measure the irregularity of the
integrand, and hence the difficulty of integration, in the corresponding
direction. The fourth differences are constructed in a standard way, using $5$ points
arranged symmetrically around the centroid, including the centroid itself,
along the coordinate axis $j$ in the reference cell $\widehat{K}$. If $Q'_K$
involves quadrature points
available symmetrically along the coordinate axes, we can reuse these points
for computing $\{\zeta^j_K\}_{j=1}^d$. This is the case for quadratures used in
Genz-Malik and also tensor-product Gauss-Legendre quadratures. However, for
Witherden-Vincent and Xiao-Gimbutas families we need to additionally construct
these points, typically performed using 1D Gauss-Legendre quadrature points of
order $5$ or above, along each central coordinate axis.

It is important to note that this procedure can also be extended to simplices,
provided the anisotropic bisection is defined appropriately, as proposed in
\cite{Genz1991simplex,GenzCools2003simplex}. Alternatively, in $2$D one can handle simplices within
the framework of~\cref{alg:h-adapt-quad}, which is designed for quadrilaterals,
by incorporating a Duffy transform into the reference-to-physical map $F_K$.
This allows the refinement process on the reference square to induce the
corresponding partitions on the simplex in physical space.

\subsection{Global adaptive quadrature algorithm} \label{sec:adaptive-quad-algo}
\Cref{alg:h-adapt-quad} outlines the global adaptive procedure for
constructing the quadrature rule using the ingredients introduced above.
In particular, $\xi$ is the algorithmic relative tolerance appearing in the
stopping rule~\eqref{eq:rel-stopping-criterion}, while $\rho$ is the
corresponding absolute safeguard for very small integral contributions.
\begin{algorithm}[ht]
\DontPrintSemicolon
\caption{The (anisotropic) bisection-based $h$-adaptive quadrature}\label{alg:h-adapt-quad}
\KwSty{Requirements:} \\
\texttt{rule} (\texttt{degree} ($p$), \texttt{degree\_ref} ($p'$)), \texttt{stopping\_criterion} (\texttt{rtol} ($\xi$, relative tolerance),
\texttt{atol} ($\rho$, absolute tolerance), \texttt{maxevals})\\
\KwIn{$f$, $\TT_{\texttt{init}}$}
\KwOut{$S$, $E$, $\TT$}

SET \texttt{$\TT \gets \TT_{\texttt{init}}$}\;

COMPUTE the local quadrature approximations $Q_K$ and $Q'_K$ of $I_K(f)$ using
    polynomial quadratures of \texttt{degree} and \texttt{degree\_ref} over
    all cells in $\TT$\;

ESTIMATE the local quadrature error $\delta_K$ for each element $K \in
\TT$ by $\lvert Q_K - Q'_K \rvert$\;

SET $S \gets \sum_{K \in \TT} Q_K$ and
    $E \gets \sum_{K \in \TT} \delta_K$\;

\While{\upshape \texttt{stopping\_criterion}($\texttt{rtol}, \texttt{atol}, \texttt{maxevals}$) is not met}{ 

    MARK the element with highest error estimate $K^{\ast} \coloneqq \argmax_K
    \delta_K$\;

    BISECT $K^{\ast}$ into child elements, $\mathcal{C}(K^{\ast})$, along
    the plane of \emph{anisotropy} determined by the corresponding dominant
    fourth difference axis: $\hat{\pmb{n}}_{K^\ast} \coloneqq F_{K^\ast} \circ (\argmax_j
    \{\zeta^j_{K^\ast}\})$, resulting in the refined partition
    $\TT \gets (\TT \setminus \{K^\ast\}) \cup \mathcal{C}(K^\ast)$\; 

    COMPUTE $Q_K$, $Q'_K$, and $\delta_K$ for each $K \in \mathcal{C}(K^{\ast})$\;

    UPDATE the current integral estimate $S \leftarrow S - Q_{K^\ast}
    + \sum_{K \in \mathcal{C}(K^{\ast})} Q_K$ and current global error estimate
    $E \leftarrow E - \delta_{K^\ast} + \sum_{K \in
    \mathcal{C}(K^{\ast})} \delta_K$\;
    }

\KwRet{$S$, $E$, $\TT$}
\end{algorithm}
It is noteworthy that the adaptive quadrature construction procedure is tied to
a specific state of the loss functional and hence to a specific \ac{nn}
instance. Since we start from a coarse mesh and refine cells one by one, there
is no need for explicit coarsening within a single adaptive quadrature build.
As we will see in the training methodology, rebuilding the quadrature from a
base mesh at each refresh makes coarsening and redistribution an emergent
feature of the overall procedure.
This is precisely the connection with~\Cref{sec:methodology}:
the refresh strategy there repeatedly reconstructs the adaptive quadrature so
that the quadrature control targeted in the present section is maintained as
the network evolves during training.
The error estimate in~\Cref{lem:anisotropic-box-quad-error} implies convergence of the adaptive
algorithm under uniform refinement and also under the anisotropic refinement
rule, provided the directional error indicator reliably selects the most
effective refinement direction.
While the extension of the \ac{aq} strategy to simplicial meshes (triangles and
tetrahedra) is a direct generalisation based on the refinement rules discussed
in~\Cref{subsec:anisotropic-refinement-rules}, it is omitted here for brevity
and reserved for future investigation.   

\section{Methodology and Implementation} \label{sec:methodology}

To balance computational efficiency and avoid overfitting, we propose an
automatic training strategy that monitors the relative global quadrature error
and triggers a refresh only when the estimated error exceeds a prescribed
threshold. The estimate is computed using the primal (training) and reference
quadratures extracted at the previous AQ trigger.

In addition to the target relative tolerance $\xi$ for the adaptive quadrature itself, we introduce an AQ refresh threshold ($\tau \ge \xi$). The tolerance $\xi$ controls the level of relative perturbation targeted when a new quadrature is constructed, whereas $\tau$ controls when the existing quadrature is considered no longer reliable and must be refreshed. Let $\eta$ denote the relative quadrature error indicator obtained by comparing the primal and reference discrete losses, i.e., a computable approximation of the relative perturbation of the loss functional at the current network realisation. We rebuild the quadrature only when the relative global integration error estimate $\eta$ exceeds $\tau$ during training.
This strategy aims to maintain control of the relative perturbation of the loss functional during training, by triggering quadrature reconstruction only when the current quadrature no longer provides sufficient control for the evolving network state.

\Cref{alg:dls-training} summarises the resulting training strategy. A
refresh is performed whenever the error indicator $\eta$ exceeds the
user-specified threshold $\tau$. At each refresh, the anisotropic
bisection-based $h$-adaptive algorithm is run from a user-provided base
partition until the target relative tolerance $\xi$ is reached. The resulting
composite quadrature then defines the discrete loss used for training and for
estimating the relative global quadrature error.
\begin{algorithm}[ht]
\DontPrintSemicolon
\caption{DLS training with adaptive quadrature (\Cref{alg:h-adapt-quad})}\label{alg:dls-training}
\KwSty{Requirements:} \\
\texttt{AQ\_algorithm (rtol ($\xi$), atol ($\rho$))}, \texttt{refresh\_tol}
($\tau$), \;\texttt{stopping\_criterion} (\texttt{max\_epochs},
\texttt{time\_limit}, \texttt{unsatisfactory\_progress}) \\
\KwIn{$u_{\btheta_0}$, $\scrJ$ (ideal loss representative), $\tilde{\scrJ}$
(primal discrete loss, based on $Q^p$), $\tilde{\scrJ}'$ (reference discrete
loss, based on $Q^{p'}$), $\TT_{\texttt{init}}$} \KwOut{$u_{\btheta_{L+1}}$,
$\TT_{\texttt{history}}$}

SET $\eta^0_\TT \gets \tau$ \tcp*[r]{force initial adaptive step}

\For{\upshape $i = 0, 1, 2, \ldots, \texttt{max\_epochs}$}{

    \If{\upshape $\eta^i_\TT \geq \tau$}{ ADAPT $\TT \gets$
        \texttt{AQ\_algorithm} ($u_{\btheta_i}$, $\scrJ$,
        $\TT_{\texttt{init}}$, $\xi$, $\rho$) \\
        UPDATE the perturbed primal loss $\tilde{\scrJ} \gets (\TT, Q^p)$ and
        perturbed reference loss $\tilde{\scrJ}' \gets (\TT, Q^{p'})$ \\
        COLLECT $\TT_{\texttt{history}} \gets \left(i, \btheta_i,
        \TT\right)$
    }

    TRAIN using $\tilde{\scrJ}$ as the loss function, updating the NN
    parameters from $\btheta_i$ to $\btheta_{i+1}$\;
    
    ESTIMATE the error indicator $\eta^{i+1}_\TT$ by applying
    $\tilde{\scrJ}$ and $\tilde{\scrJ}'$ to the updated approximation
    $u_{\btheta_{i+1}}$\;

    \If{\upshape \texttt{stopping\_criterion}(\texttt{max\_epochs},
    \texttt{time\_limit}, \texttt{unsatisfactory\_progress}) is
    met}{\KwSty{break}} }

    SET $I \gets i$;

\KwRet{\upshape $u_{\btheta_{I+1}}$, $\TT_{\texttt{history}}$}
\end{algorithm}
\Cref{alg:dls-training} is fully deterministic and uses full-batch
training. The adaptive quadrature algorithm (Algorithm
\ref{alg:h-adapt-quad}) is executed on the CPU, where the primal and reference
quadrature points and weights are accumulated in physical space. These data are
then transferred to the target device, which is a GPU in our implementation.
Training and quadrature-error estimation are subsequently performed on the GPU
for the current partition. At the end of each epoch, the indicator $\eta$ is
recomputed; whenever $\eta \ge \tau$, the adaptive quadrature is rebuilt on the
CPU. For efficiency, an AQ refresh is never triggered during line-search
iterations. ~\
Between refresh steps, the quadrature is kept fixed while the network evolves, so the error indicator reflects the degradation of the quadrature accuracy over the training trajectory.

Although single-GPU~\cite{Pagani2021} and multi-GPU
\cite{tonarelli2025AQmultigpu} extensions for adaptive quadrature exist, we do
not use them in this work. Numerical experiments indicate that adaptive
quadrature is only a minor contributor to the total compute time, since the
number of AQ executions is a small fraction of the total number of training
epochs. Moreover, for ambient dimensions $d \leq 4$, the speed-up offered by a
GPU for adaptive quadrature construction is not significant compared with a
CPU.

A key feature of our approach is that automatic differentiation is not required to propagate
through the adaptive quadrature algorithm itself.
Instead, the algorithm is used only to construct the quadrature points and
weights, which are then used to approximate the loss functional for training
and quadrature error estimation.

Because \ac{nn} training is performed exclusively with the primal
quadrature, we also reduce computational cost by avoiding the calculation of
parametric derivatives of the loss evaluated on the richer reference
quadrature.

\section{Numerical Experiments}\label{sec:numexps}

In this section, we assess the proposed anisotropic $h$-adaptive composite
quadrature on a collection of benchmark problems that spans steady, time-dependent and parametric linear and nonlinear
\acp{pde} with sharp layers, localised structures and non-trivial boundary
conditions. Our aim is to determine whether explicit control of a reliable
estimate of the global integration error reduces overfitting and improves
approximation quality relative to standard non-adaptive quadrature strategies,
especially in regimes where classical discretisations often require
specialised adaptivity or problem-specific enhancements.

\subsection{Experimental setup}

We now summarise the comparison protocol and the implementation choices shared
across the experiments. Throughout this section, we refer to the richer
quadrature and the corresponding higher-fidelity loss evaluation as the
\emph{reference} quadrature and \emph{reference} loss, respectively.
Close agreement between training and reference losses is therefore essential
for a reliable method, whereas a persistent mismatch indicates that the
training quadrature is no longer reliable.

\medskip
\noindent
\emph{Neural network approximation ansatz.}
All experiments use fully connected, feed-forward \acp{nn} with $\tanh$
activation and a linear output layer. We describe each network by its input and
output dimensions, the number of hidden layers, and the common width of those
hidden layers. Tanh \acp{nn} are globally $C^\infty$ smooth and have
favourable approximation properties~\cite{TanhDeRyck2021}, which makes them
natural candidates for strong-form residual minimisation. Their smoothness also
permits direct evaluation of the residual without differentiability issues, in
contrast to non-smooth activations that can be problematic for loss
functionals involving nested parametric and ambient derivatives
\cite{MagueresseBadiaAdaptive2024}.

For each problem, the network architecture and initialisation are kept fixed
across the different quadrature strategies in order to isolate the effect of
the integration method. Specifically, parameters are initialised using the
Glorot uniform strategy~\cite{GlorotBengio2010} with a fixed seed. When
comparing with existing benchmarks, we use the same architecture or a smaller
one when restricted by available hardware.

\medskip
\noindent
\emph{Loss function.}
Throughout the numerical experiments, we use relaxed residual losses based on
$L^2$ norms, following the Poisson example in
\Cref{ex:poisson-relaxed-loss}. More precisely, for a residual operator
$\mathcal{R}_\Omega$ in the domain (or space-time domain) and boundary or
initial residual operators $\mathcal{R}_\Gamma$, the loss has the generic form
\[
  \scrJ(v)
  =
  \sqrt{\int_{\Omega} \bigl|\mathcal{R}_\Omega(v)\bigr|^2 \, d\Omega
  +
  \sum_{\Gamma} \gamma_\Gamma \int_{\Gamma} \bigl|\mathcal{R}_\Gamma(v)\bigr|^2 \, d\Gamma}.
\]
Here the sum runs over all relevant boundary and initial-condition
contributions, so several terms may appear simultaneously. In particular, this
allows mixed boundary conditions, such as combined Dirichlet and Neumann terms,
each with its own penalty weight. This reduces to the $L^2$ norm of the strong
residual in the domain together with $L^2$ penalty terms for the boundary
conditions. The same idea is used for the initial and inflow conditions that
appear in the time-dependent and transport examples.\footnote{Observe that we
are using $p = 1$ in the definition of $\scrJ$ for all the numerical
experiments.}

\medskip
\noindent
\emph{Optimiser.}
We use the SSBroyden optimiser~\cite{Al-Baali1998}
throughout, together with the Hager-Zhang inexact line-search algorithm based
on a weak Wolfe condition~\cite{HagerZhang2006}. This
choice is motivated by recent empirical and theoretical work indicating that
quasi-Newton methods can substantially outperform first-order methods and
L-BFGS for high-accuracy training in residual-based neural \ac{pde} solvers
\cite{UrbanPINNOpt2025, KiyaniOpt2025, Wang2025gradalign}. Moreover, as noted
in~\cite{Wang2025gradalign, KiyaniOpt2025}, second-order methods often balance
the different terms in the loss more effectively, thereby reducing the need for
explicit loss reweighting~\cite{WangPathology2021, Wang2022PINNsFailure}.
Our numerical pipeline uses 64-bit floating-point arithmetic to avoid the loss
of accuracy that can degrade second-order optimisation
\cite{KiyaniOpt2025}. When the adaptive quadrature is refreshed,
the optimiser resumes from its previous state and preserves the accumulated
inverse Hessian approximation.

\medskip
\noindent
\emph{Numerical quadrature.}
In most numerical experiments, we employ tensor-product Gauss quadratures with
odd-even order pairings. A typical choice is order $7$ for training and order
$10$ for the reference quadrature. This pairing is used to avoid shared
odd-even artefacts across the two quadratures and to reduce
numerical cancellations along the central axes. In three spatial dimensions,
where tensor-product rules become increasingly expensive, we instead use the
quasi-optimal polynomial quadrature families in
\cite{WitherdenVincent2015,XiaoGimbutas2010}.

\medskip
\noindent
\emph{Comparison with non-adaptive quadrature strategies.}
To assess the effect of quadrature adaptation, we compare the proposed
\ac{aq} algorithm (\Cref{alg:h-adapt-quad}) with several standard non-adaptive
quadrature strategies: composite Gauss-Legendre quadrature on a uniform partition (uniform quadrature), \ac{mc} quadrature, \ac{lhc} (as used in  
\cite{PINNs2019}) and Halton sampling. The latter is a
deterministic low-discrepancy strategy known to perform well in moderate
dimensions~\cite{MorokoffQMC1995}, which covers the ambient dimensions
considered here. To ensure a fair comparison, we match the computational budget of the
non-adaptive baselines to the adaptive runs. The
total number of quadrature points for \ac{mc}, \ac{lhc} and
Halton sampling is fixed from the start to the $90\%$ quantile of the adaptive
history of total quadrature points. In practice, this is usually close to the
maximum attained during training. For uniform quadrature, the number of
partitions is fixed analogously as the nearest square number exceeding the
$90\%$ quantile of the adaptive partition counts. The corresponding richer
quadratures are constructed in the same way, with sizes prescribed by the
$90\%$ quantile of the adaptive history of richer quadrature points.

\medskip
\noindent
\emph{Visualisation.} 
For visualisation of the computed solutions and approximation errors, we use
fine point sets chosen independently of the final training quadrature in order
to avoid misleading presentations of the generalisation behaviour.

\medskip
\noindent
\emph{Computational time.} 
The reported compute times include the cost of evaluating the true
approximation errors ($L^2$ and $H^1$), whenever available, on a fixed fine
mesh at every training iteration. In one and two spatial dimensions, we also
evaluate the loss on this mesh in order to verify the richer quadrature,
although this is not required in practice. The timings should therefore be
interpreted comparatively rather than as absolute benchmarks. All experiments
were run on an NVIDIA V100 GPU.

\medskip
\noindent\emph{Computational framework.}
Our framework is implemented in Julia~\cite{Julia2017}. \acp{nn} are
built with Flux.jl~\cite{FluxOSS2018} and operated via Optimisers.jl, with
automatic differentiation handled by Zygote.jl~\cite{Zygote2019} and
ChainRules.jl. The SSBroyden optimiser is implemented within the Optim.jl
framework~\cite{Optimjl2018}. The anisotropic bisection-based
$h$-adaptive quadrature builds on HCubature.jl~\cite{HCubature} and Gridap.jl
\cite{GridapBadia2020,GridapVerdugo2022} and is executed primarily on the CPU,
whereas \ac{nn} training is performed on the GPU using CUDA.jl
\cite{besard2018juliagpu}. All visualisations were produced with CairoMakie.jl
\cite{DanischKrumbiegel2021}. This implementation split matches the
methodology described in~\cref{sec:methodology}, where quadrature construction
and training are separated computationally.

\subsection{The neural network approximation of the 2D arc-tan well function}
\label{subsec:fa-2d-problem}

We begin with the simplest possible problem: function approximation using a
\nn{}. Specifically, we consider a scalar-valued function in 2D defined on
$\Omega = [0,1]^2$ as follows:
\begin{equation}\label{eq:fa-2d-problem}
f(x,y) = \atan\left(200 \left(\sqrt{(x - 0.35)^2 + (y - 0.45)^2} - 0.2\right)\right).
\end{equation}
The ideal continuous loss function for function approximation is simply the
$L^2$ misfit error in this case.
For this experiment, we employ a $\tanh$ \nn{} with width $25$ and depth $4$,
containing a total of \num{2051} parameters (weights and biases across all
layers). The \ac{aq} algorithm uses tensor-product Gauss-Legendre quadrature
pairs of orders $(7, 10)$, with a base uniform mesh partition of $3 \times 3$
for the adaptive quadrature to build upon, targeting a relative tolerance of
$10^{-2}$. Note that the base mesh provides quadrature point distribution
independent of the solution structure. The tolerance threshold for triggering
adaptive quadrature refinement is set to $5 \times 10^{-2}$, which is $5$ times
the target relative tolerance. Furthermore, the \acp{nn} are trained for
a maximum of \num{10000} epochs, regardless of the quadrature strategy. We
construct a fixed fine mesh of size $100 \times 100$, equipped with
Gauss-Legendre quadrature of the same order as the training (primal) quadrature
employed by the \ac{aq} algorithm, to estimate the $L^2$ and $H^1$ errors
throughout training.
\begin{figure}
  \centering
  \begin{subfigure}[t]{0.45\linewidth}
    \centering \captionsetup{width=.8\linewidth}
    \includegraphics[width=\linewidth]{
        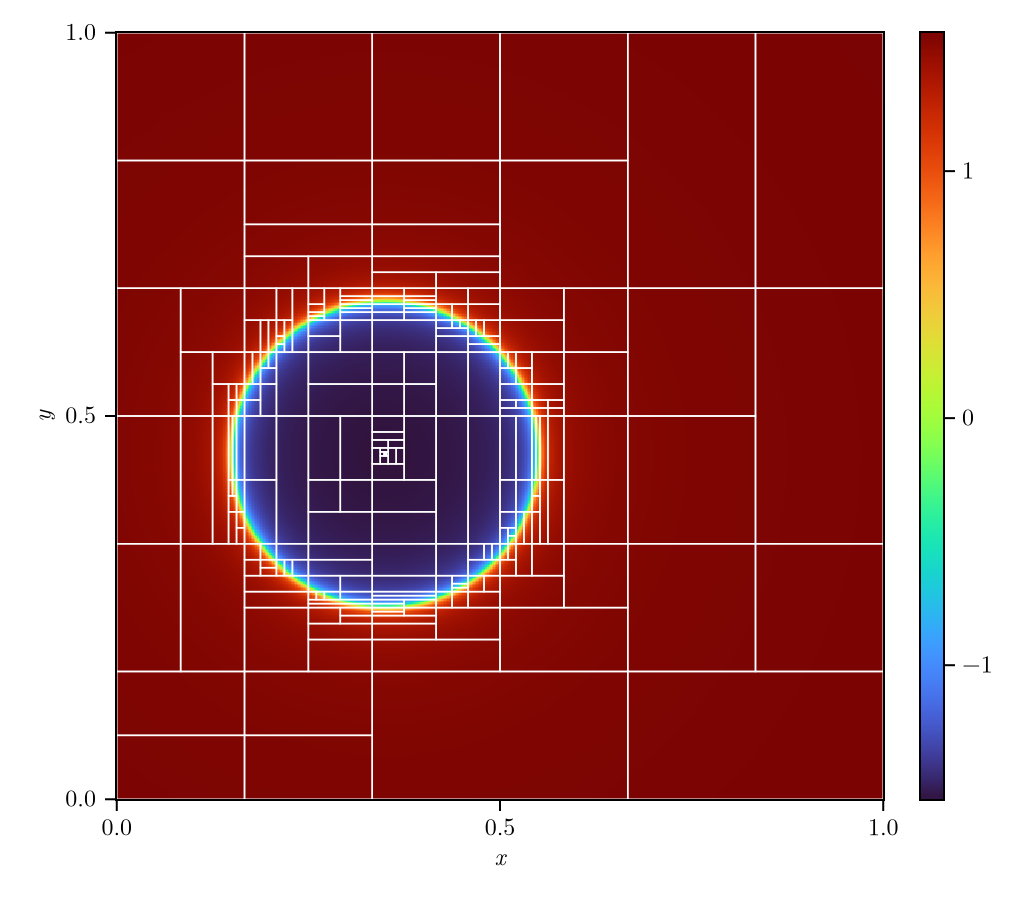
    }
    \caption[NN approximation]{NN approximation and the final adaptive mesh.}
    \label{fig:fa-2d-aq-sol-mesh}
  \end{subfigure}
  \begin{subfigure}[t]{0.47\linewidth}
    \centering \captionsetup{width=.8\linewidth}
    \includegraphics[width=\linewidth]{
        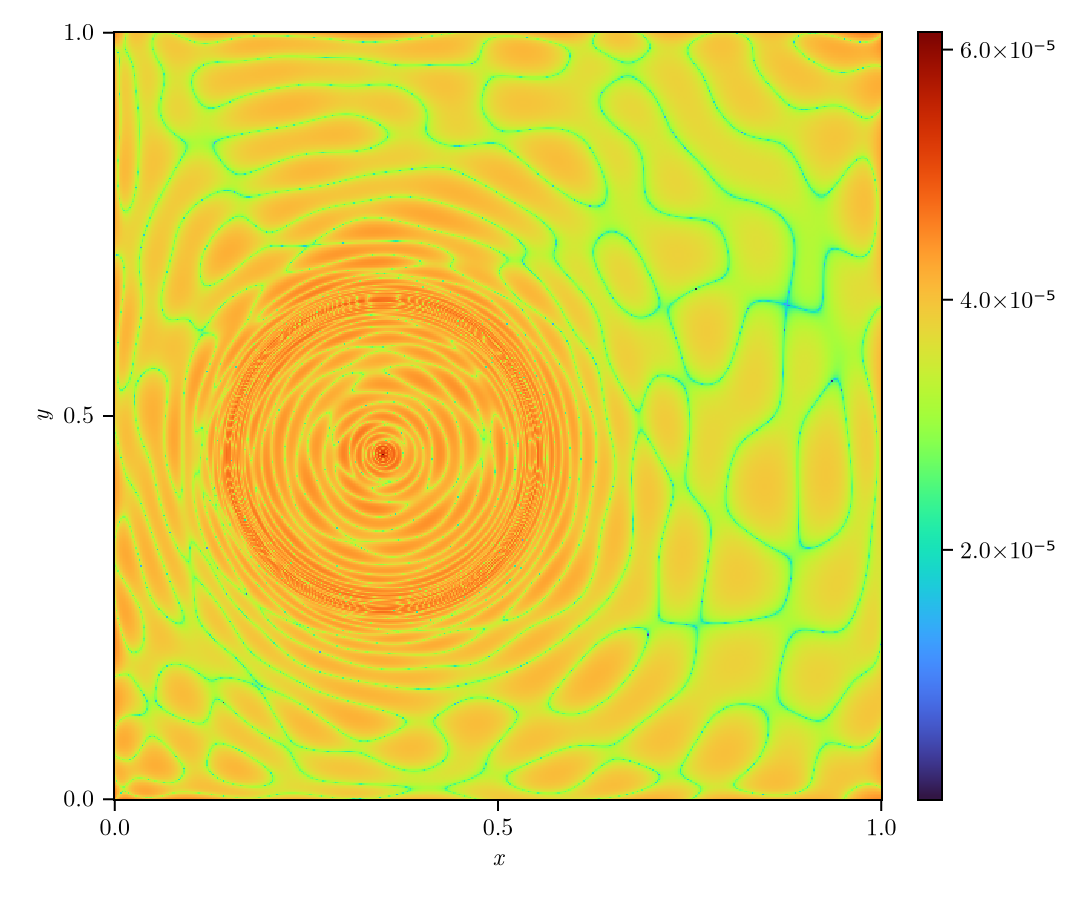
      }
        \caption[Absolute point-wise error]{Absolute point-wise error.}
    \label{fig:fa-2d-aq-error}
  \end{subfigure}
  \caption[NN solution for the function-approximation
  problem]{Adaptive quadrature solution, final adaptive partition and absolute point-wise errors for 
  the function approximation problem
        \protect\eqref{eq:fa-2d-problem} using the AQ algorithm.} 
  \label{fig:fa-2d-aq-results}
\end{figure}

Examining the \ac{nn} approximation and final adaptive quadrature mesh
in~\cref{fig:fa-2d-aq-sol-mesh}, we observe that the \ac{aq} mesh successfully
traces the circular arc where the target function~\eqref{eq:fa-2d-problem}
exhibits steep gradients, achieving highly localized refinement through the
anisotropic nature of the \ac{aq} algorithm. Notably, the \ac{aq} mesh also
captures the center of the well associated with the target function.
\cref{fig:fa-2d-aq-error} shows that the point-wise errors remain
small across the domain. 
\begin{figure}
  \centering
  \begin{subfigure}[t]{0.48\linewidth}
    \centering \captionsetup{width=.8\linewidth}
    \includegraphics[width=\linewidth]{
        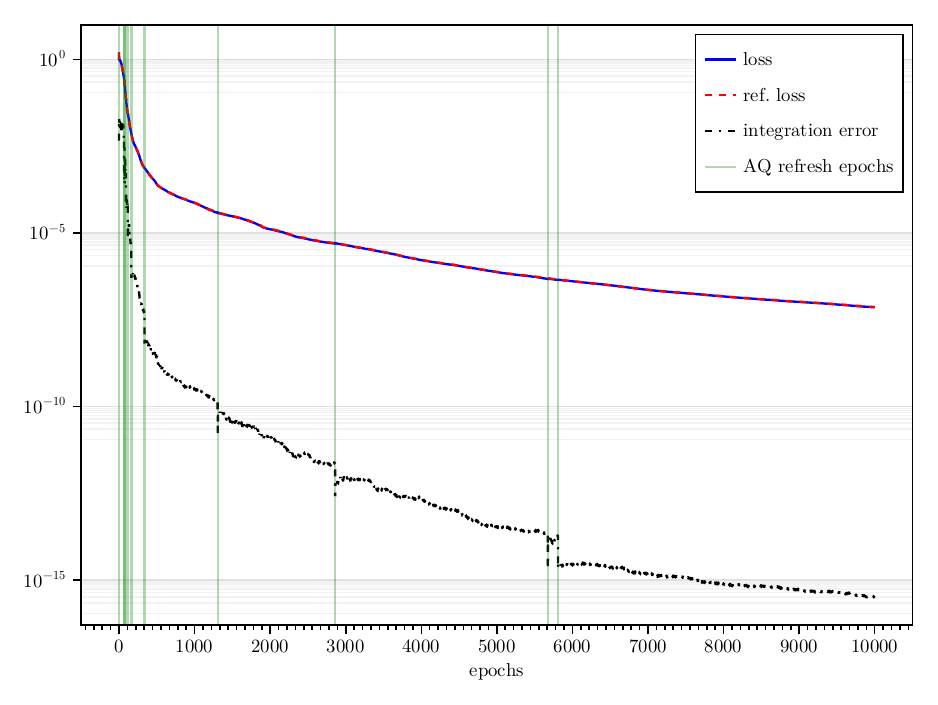
    }
    \caption{Training and
  reference loss histories together with the cell-wise cumulative integration
  error of the primal quadrature relative to the reference quadrature.}
    \label{fig:fa-2d-aq-training}
  \end{subfigure}
  \begin{subfigure}[t]{0.48\linewidth}
    \centering \captionsetup{width=.8\linewidth}
    \includegraphics[width=\linewidth]{
        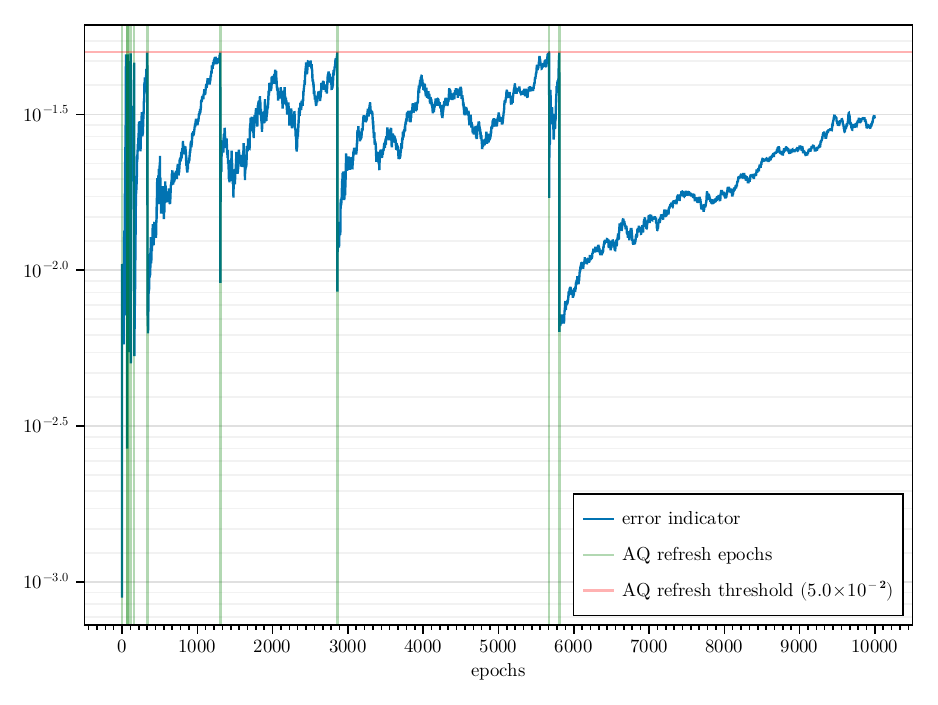
        }
    \caption{Error-indicator history, which triggers AQ refreshes when it exceeds the
  prescribed threshold.}
    \label{fig:fa-2d-aq-error-ind}
  \end{subfigure}
  \caption[Adaptive-quadrature training diagnostics for the 2D ArcTanWell function-approximation problem]{Adaptive training diagnostics for the 
  function approximation in \protect\eqref{eq:fa-2d-problem}.}
  \label{fig:fa-2d-aq-training-results}
\end{figure}

The training curve, and more importantly the reference
curve, are informative indicators of the \ac{nn} performance. In
\cref{fig:fa-2d-aq-training}, the training curve overlaps the
reference curve. Moreover, the systematic downward trend in the
cumulative cell-wise absolute integration error, driven by the marked
\ac{aq} refreshes, supports the reliability of the training process and shows
no indication of overfitting.~\Cref{fig:fa-2d-aq-error-ind} provides a
detailed history of the error indicator throughout training, with immediate
drops following \ac{aq} refreshes triggered when the indicator exceeds the
prescribed threshold. In this experiment, adaptive quadrature refinement occurs
only $11$ times over the \num{10000} training epochs, with refreshes becoming
less frequent as training progresses. Similar AQ training-diagnostic figures arise throughout
the remaining experiments, so we do not repeat them systematically. Instead, we
show them again only for the most challenging Navier-Stokes benchmarks.
\begin{figure}
  \centering
  \begin{subfigure}[t]{1.0\linewidth}
    \centering
    \includegraphics[width=0.62\linewidth]{
        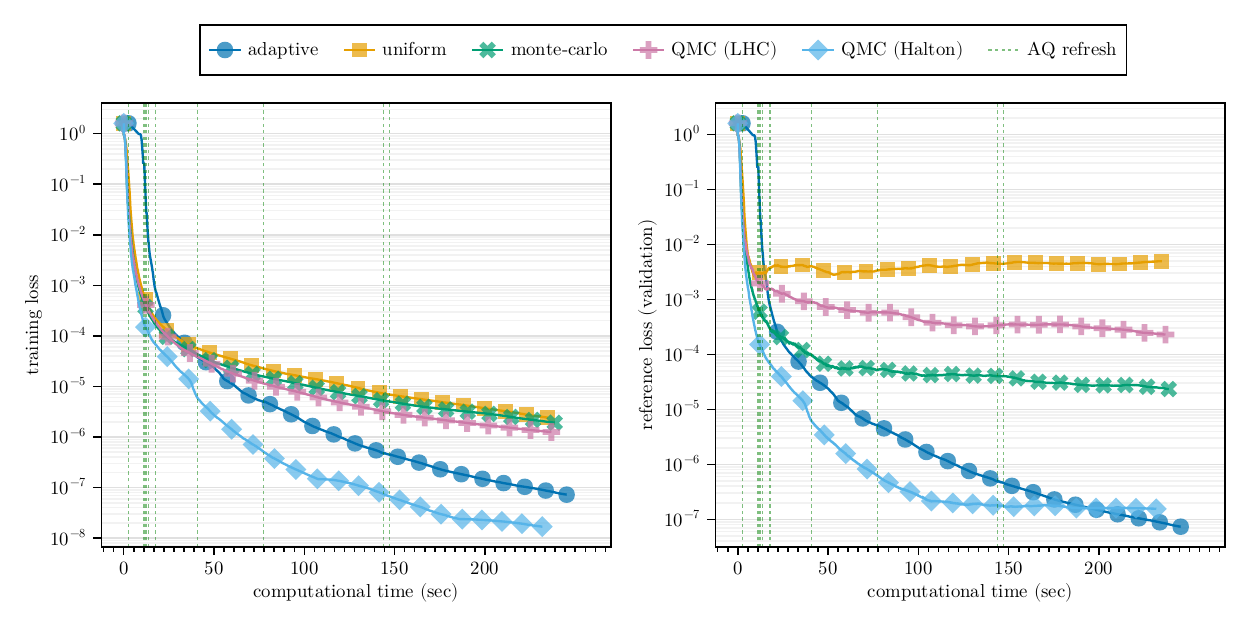
    }
    \caption{Training and reference loss histories across all quadrature strategies.}
    \label{fig:fa-2d-loss-history-panel}
  \end{subfigure}

  \medskip

  \begin{subfigure}[t]{0.62\linewidth}
    \centering
    \includegraphics[width=\linewidth,trim={0 0 0 40},clip]{
        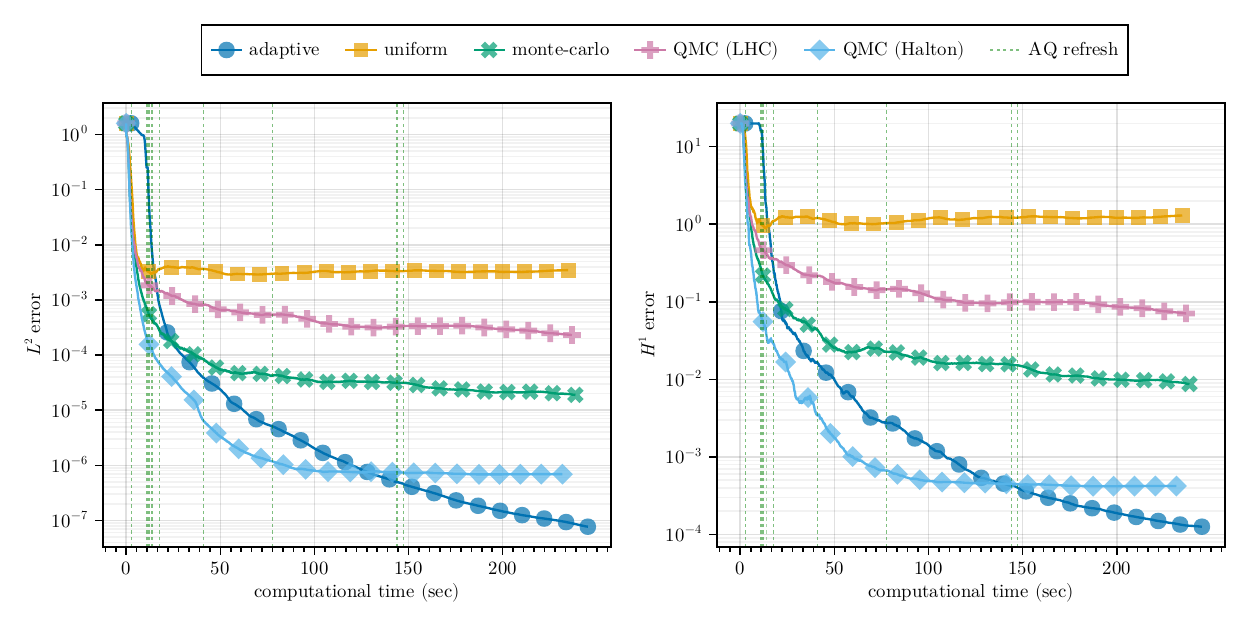
    }
    \caption{$L^2$ and $H^1$ error histories across all quadrature strategies.}
    \label{fig:fa-2d-error-history-panel}
  \end{subfigure}
  \hfill
  \begin{subfigure}[t]{0.34\linewidth}
    \centering
    \includegraphics[width=\linewidth]{
        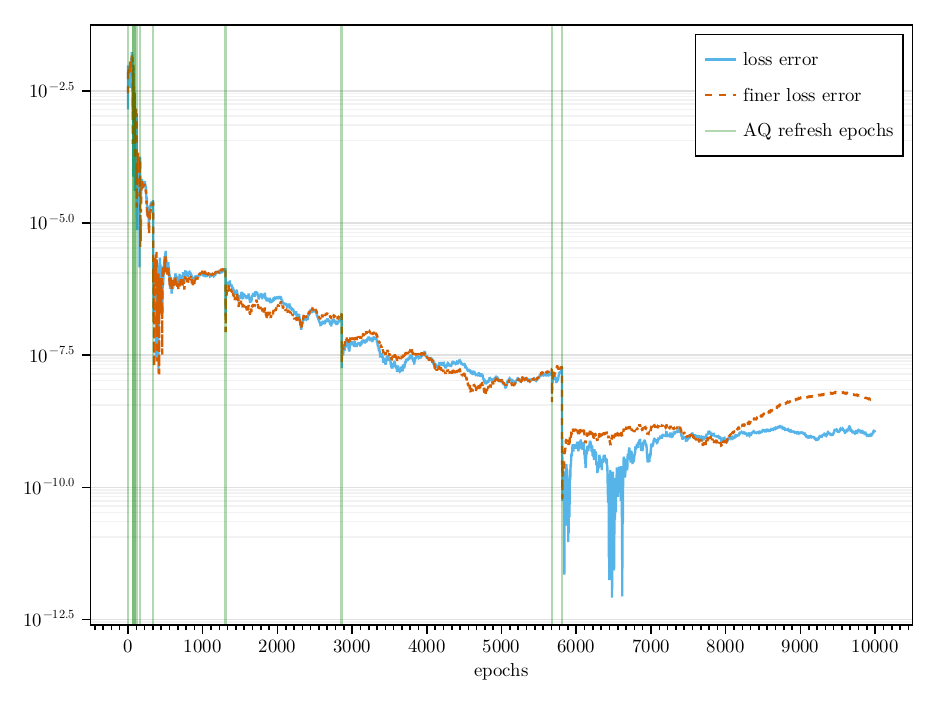
    }
    \caption{Reference quadrature verification.}
    \label{fig:fa-2d-aq-finer-panel}
  \end{subfigure}
  \caption[Loss, error, and verification histories for the function-approximation problem]{Loss, error and verification histories for the function approximation in \protect\eqref{eq:fa-2d-problem}.}
  \label{fig:fa-2d-error-history-comparison}
\end{figure}

We now present a comprehensive comparison of loss curves and approximation
errors across all quadrature strategies in
\cref{fig:fa-2d-loss-history-panel,fig:fa-2d-error-history-panel}: adaptive
quadrature, uniform quadrature, \ac{mc}, \ac{qmc} (\ac{lhc}), and \ac{qmc}
(Halton). We start by observing that the proposed \ac{aq} strategy achieves
superior $L^2$ and $H^1$ errors, followed closely by \ac{qmc} (Halton). 
From~\cref{fig:fa-2d-loss-history-panel}, we observe that only the adaptive
quadrature strategy maintains a monotonic decrease in both training and
reference losses. For the other quadrature approaches, the training loss
continues to decrease while the reference loss stagnates, indicating a loss of
generalisation. Importantly, the reference loss correlates strongly with the
$L^2$ and $H^1$ errors, making it a useful practical indicator of training
progress and reliability. It is also notable that we obtain good $H^1$ errors
even though the loss functional is based only on the $L^2$ error.
Despite employing the most quadrature points, the uniform
quadrature strategy yields the poorest errors by a significant margin,
exposing the inherent disadvantages of uniform partitioning and revealing
Runge phenomena. This is evident in the escalating $H^1$ error progression for
most part of the training, as seen in~\cref{fig:fa-2d-error-history-panel}.

Finally, since the approximation quality of the reference quadrature is
essential for accurately estimating cell-wise errors in the primal quadrature,
we validate its effectiveness by comparing with a finer quadrature. We reuse
the uniform composite quadrature constructed for $L^2$ and $H^1$ error
estimation to perform an independent, accurate loss computation (interpreted
as a test quadrature). As illustrated in
\cref{fig:fa-2d-aq-finer-panel}, we
compare the error incurred by the primal composite quadrature with that of the
reference quadrature (using the same partition) and the finer uniform composite
quadrature. The results demonstrate that the reference quadrature estimates the
error reliably and tracks the same trends throughout training, with only minor
discrepancies after the final adaptive quadrature refinement.

\subsection{The 1D Advection-Diffusion problem} \label{sec:1d-adv-diff}

From this section onwards, we focus on \ac{pde} benchmarks solved by
$L^2$ residual minimisation with \acp{nn}. In this section, we use the relaxed loss
functional described in~\Cref{ex:poisson-relaxed-loss}; for simplicity, we
refer to this setting as the deep least squares approach.
In this section, we consider a 1D advection-diffusion problem with homogeneous
boundary conditions characterized by high P{\'e}clet numbers. We consider two
values of the diffusion coefficient, $\epsilon = (0.005, 0.001)$, with an
advection coefficient $\beta = 1$ and source term $f(x) = 1$ on the
domain $(-1,1)$:
\begin{align} \label{eq:1d-adv-diff}
- \epsilon \Delta u + \pmb{\beta}\cdot \nabla u &= f \quad \text{in \;\; } \Omega\,,\\
u &= 0 \quad \text{on \; } \partial\Omega \,.
\end{align}
Notably, the problem has a closed-form solution,
\begin{equation}
u(x) = 2\frac{\left(1 - e^{(x - 1)/\epsilon}\right)}{\left(1 - e^{-2/\epsilon}\right)} + x - 1.
\end{equation}
This Dirichlet boundary value problem is of particular interest due to the
presence of a strong boundary layer at the right boundary that sharpens
as $\epsilon \to 0$, making it difficult to resolve without adaptivity.
Note that the $\epsilon = 0.005$ case was considered in
\cite{RobustVPINNs2024}, although in the context of a weak-form-based
variational PINN approach. 

We begin with the $\epsilon = 0.005$ case, with
\cref{fig:dls-1d-adv-aq-sol-0.005,fig:dls-1d-adv-aq-error-0.005}
showcasing an accurate approximation of the solution with a very low and
uniform error distribution, achieved with the adaptive quadrature strategy.
\Cref{fig:dls-1d-adv-aq-sol-0.005} also shows the final quadrature partition
distribution, displaying higher mesh density at the right boundary, where the
boundary layer is located.
\begin{figure}
    \centering
    \fbox{%
      \footnotesize
      \begin{tabular}{@{}c@{\hspace{0.9em}}c@{\hspace{1.4em}}c@{\hspace{0.9em}}c@{\hspace{1.4em}}c@{\hspace{0.9em}}c@{}}
        \tikz[baseline=-0.6ex]\draw[blue, line width=0.9pt] (0,0) -- (1.1em,0); &
        exact &
        \tikz[baseline=-0.6ex]\draw[orange, dashed, line width=0.9pt] (0,0) -- (1.1em,0); &
        approx &
        \tikz[baseline=-0.6ex]\draw[red, line width=0.9pt] (0,0) -- (1.1em,0); &
        error
      \end{tabular}%
    }

    \medskip

    \begin{subfigure}[t]{0.4\linewidth}
        \centering
        \includegraphics[width=\linewidth,trim={0 0 226 44},clip]{
            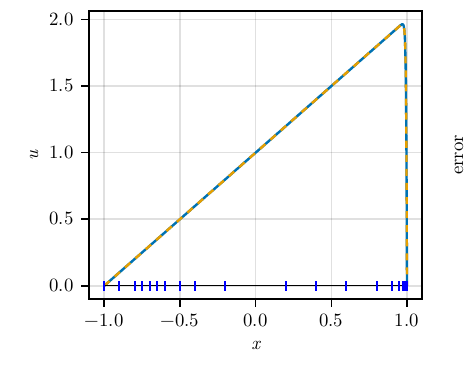
        }
        \caption{Exact and AQ solutions with final AQ partition.}
        \label{fig:dls-1d-adv-aq-sol-0.005}
    \end{subfigure}
    % \hfill
    \begin{subfigure}[t]{0.4\linewidth}
        \centering
        \includegraphics[width=\linewidth,trim={226 0 0 44},clip]{
            DLS-1D-Advection-Diffusion-BL-problem_aq-perf-study_aq-approx-plot-eps_0.005.pdf
        }
        \caption{AQ point-wise error.}
        \label{fig:dls-1d-adv-aq-error-0.005}
    \end{subfigure}

    \medskip

    \begin{subfigure}[t]{0.4\linewidth}
        \centering
        \includegraphics[width=\linewidth,trim={226 0 0 44},clip]{
            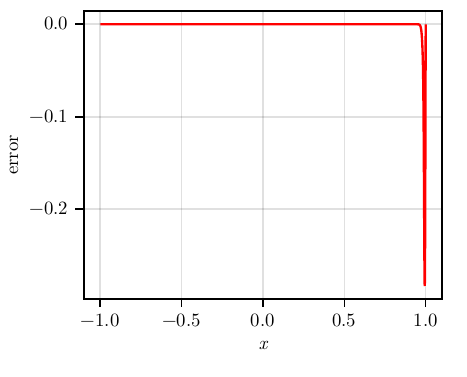
        }
        \caption{Uniform quadrature point-wise error.}
        \label{fig:dls-1d-adv-uq-error-0.005}
    \end{subfigure}
    % \hfill
    \begin{subfigure}[t]{0.4\linewidth}
        \centering
        \includegraphics[width=\linewidth,trim={226 0 0 44},clip]{
            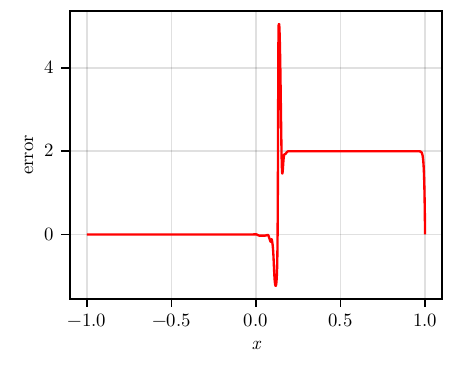
        }
        \caption{MC point-wise error.}
        \label{fig:dls-1d-adv-mc-error-0.005}
    \end{subfigure}
    \caption[Comparison of quadrature strategies for the 1D advection-diffusion problem with $\epsilon=0.005$]{Comparison of adaptive, uniform and MC quadrature for the
    1D advection-diffusion problem \protect\protect\eqref{eq:1d-adv-diff} with $\epsilon =
    0.005$.}
    \label{fig:dls-1d-adv-approx-0.005}
\end{figure}
The data efficiency of this approach is remarkable, requiring a
maximum of only $20$ partitions, amounting to just $140$ primal (training)
quadrature points in total. Only $6$ AQ refreshes were required within the
first $10\%$ of the full \num{5000} training epochs, with continued quadrature
error control throughout the entire training. In contrast,
the uniform composite quadrature, starting directly with the full
$20$ partitions, fails to capture the boundary layer accurately, as evident
from~\cref{fig:dls-1d-adv-uq-error-0.005}. Furthermore, \ac{mc} quadrature
with the full $140$ points results in localized overfitting (see
\cref{fig:dls-1d-adv-mc-error-0.005}), failing to even roughly capture the
solution. 

The adaptive quadrature strategy remains strikingly accurate even for the
difficult case of $\epsilon = 0.001$, maintaining very small $L^2$ and $H^1$
approximation errors over training; see
\cref{fig:dls-1d-adv-error-history-panel-0.001}. The final adaptive quadrature
mesh is even more concentrated at the
right boundary than in the $\epsilon = 0.005$ case shown in
\cref{fig:dls-1d-adv-approx-0.005}. The final refresh also slightly reduces the
number of mesh partitions while concentrating them predominantly near the right
boundary.
We re-emphasize the data-efficient nature of the adaptive quadrature strategy by
noting that a maximum of $25$ partitions (with the $90\%$ quantile being $23$
partitions) amounts to just $161$ primal (training) quadrature points in total.
We further highlight that only $7$ AQ refreshes were required within the first
$25\%$ of the full \num{5000} training epochs, maintaining continued
quadrature error control for the entire training thereafter.

In contrast, the uniform quadrature strategy leads to complete failure, as seen
in~\cref{fig:dls-1d-adv-loss-history-panel-0.001,fig:dls-1d-adv-error-history-panel-0.001}, by
seemingly overfitting to the right
Dirichlet boundary condition while complying with the convection part of the \ac{pde}
\protect\eqref{eq:1d-adv-diff} and ignoring the left Dirichlet boundary condition.
\ac{mc} quadrature exhibits similar behaviour to the $\epsilon = 0.005$
case; the corresponding plots are omitted for brevity because they look very
similar to those for the $\epsilon = 0.005$ case.
\begin{figure}
  \centering
  \begin{subfigure}[t]{1.0\linewidth}
    \centering
    \includegraphics[width=0.82\linewidth]{
        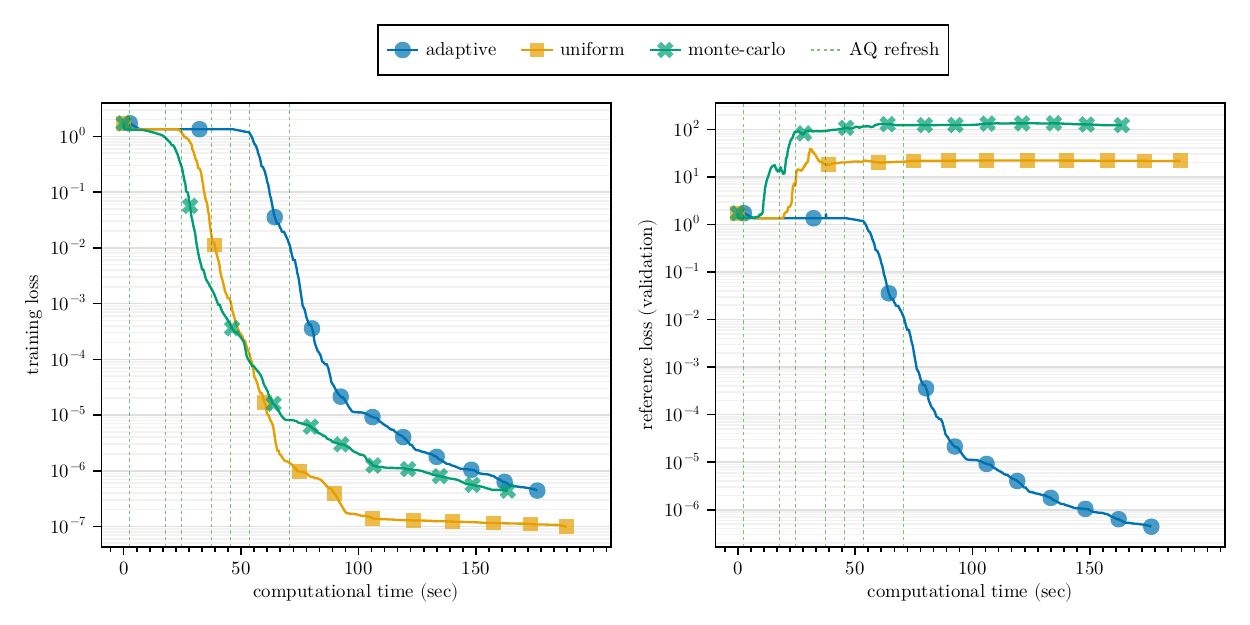
    }
    \caption{Training and reference loss history comparison across all
    considered quadrature strategies.}
    \label{fig:dls-1d-adv-loss-history-panel-0.001}
  \end{subfigure}
  \medskip
  \begin{subfigure}[t]{1.0\linewidth}
    \centering
    \includegraphics[width=0.82\linewidth,trim={0 0 0 40},clip]{
        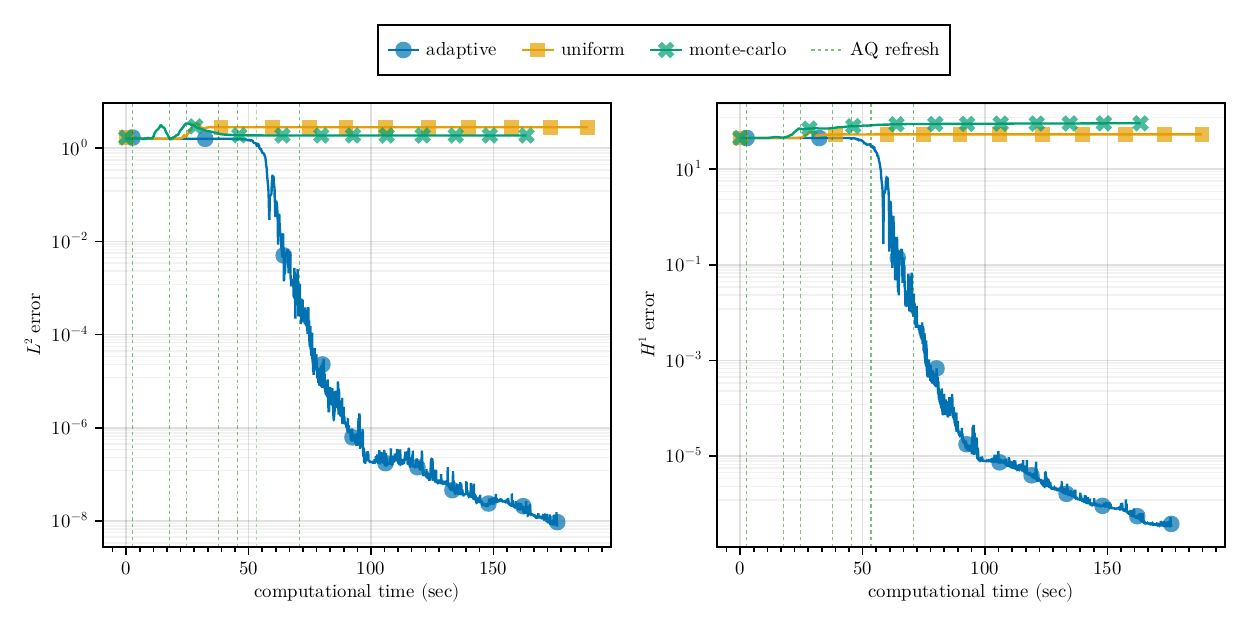
    }
    \caption{$L^2$ and $H^1$ error history comparison across all considered
    quadrature strategies.}
    \label{fig:dls-1d-adv-error-history-panel-0.001}
  \end{subfigure}
  \caption{Loss and error histories for the 1D advection-diffusion problem  \protect\protect\eqref{eq:1d-adv-diff} with $\epsilon=0.001$}  
  \label{fig:dls-1d-adv-loss-history-comparison-0.001}
\end{figure}
The loss histories in~\cref{fig:dls-1d-adv-loss-history-panel-0.001} highlight the
robustness of the adaptive quadrature strategy by showcasing a monotonic decrease
in the reference loss that is strongly correlated with the training
loss, thereby controlling overfitting that severely affects both uniform
quadrature and Monte Carlo quadrature strategies. Consequently, the $L^2$
and $H^1$ approximation errors in
\cref{fig:dls-1d-adv-error-history-panel-0.001} exhibit clear
superiority for the adaptive strategy compared to the meaningless results
obtained by the uniform and Monte Carlo strategies.
We conclude by stressing that the same hyperparameter choices, in particular
the \ac{nn} architecture, relative tolerance and refresh threshold, were
used for both $\epsilon = 0.005$ and $\epsilon = 0.001$. 

\subsection{The (1+1)D Viscous Burgers equation}
\label{subsec:burgers-problem}

Next, we consider the viscous Burgers' equation 
benchmark problem used in~\cite{PINNs2019,UrbanPINNOpt2025} with viscosity $0.01/\pi$, which exhibits shock formation. 
Although an exact solution
can be approximately constructed~\cite{PINNs2019}, we consider the discrete
approximate solution taken from~\cite{UrbanPINNOpt2025}, for fair comparison
with the state-of-the-art results reported therein. We emphasize that we use a \ac{nn} of
width 20 and depth 3 (with $941$ parameters), matching~\cite{UrbanPINNOpt2025}.
Notably, this is smaller than the \ac{nn} size ranges used in
\cite{KiyaniOpt2025}, where the smallest one has 20 neurons and 4 layers. We
also use the same number of training epochs as in~\cite{UrbanPINNOpt2025},
 taken to be \num{15000}.
\begin{figure}
  \centering
  \begin{subfigure}[t]{0.48\linewidth}
    \includegraphics[width=\linewidth]{
        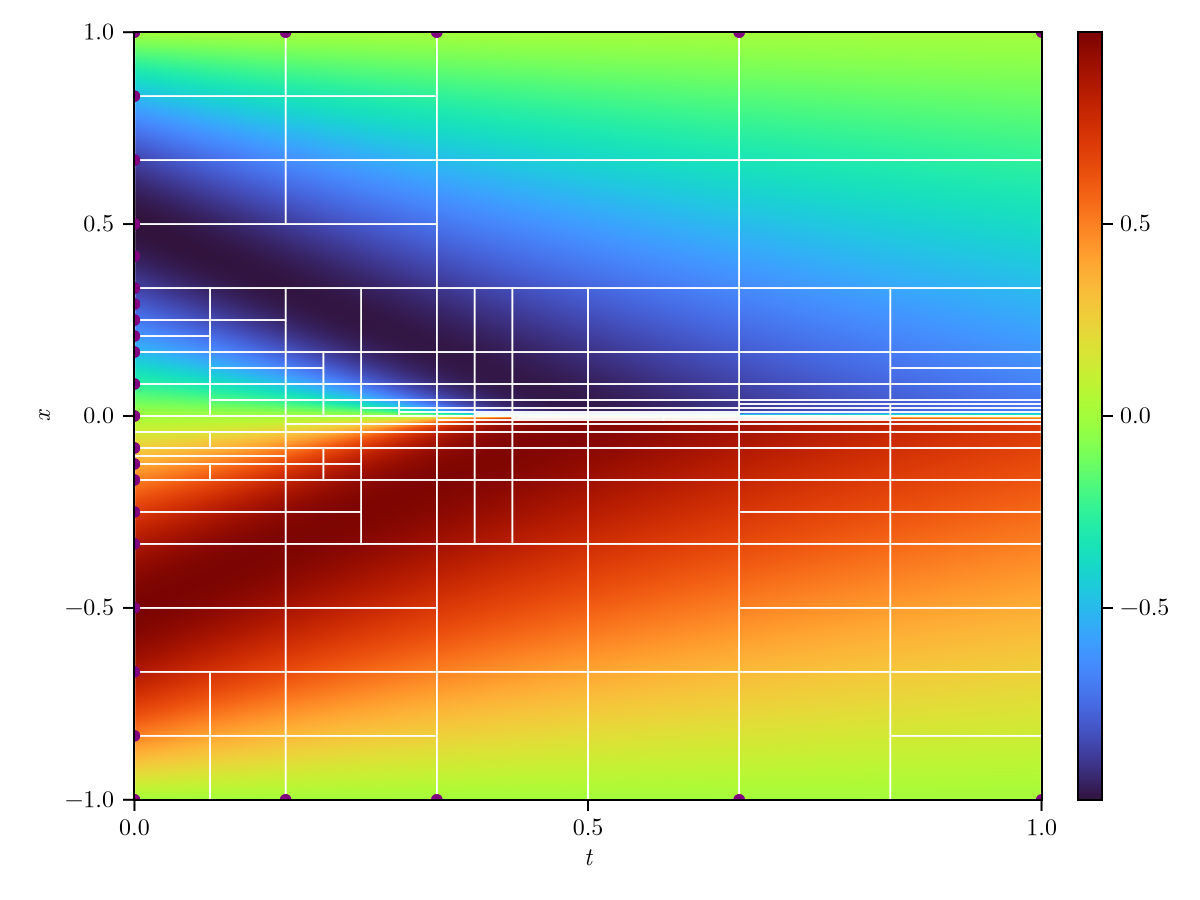
    }
    \caption{NN approximation with final AQ partition}
    \label{fig:burgers-aq-sol-mesh}
  \end{subfigure}
  \begin{subfigure}[t]{0.48\linewidth}
    \includegraphics[width=\linewidth]{
        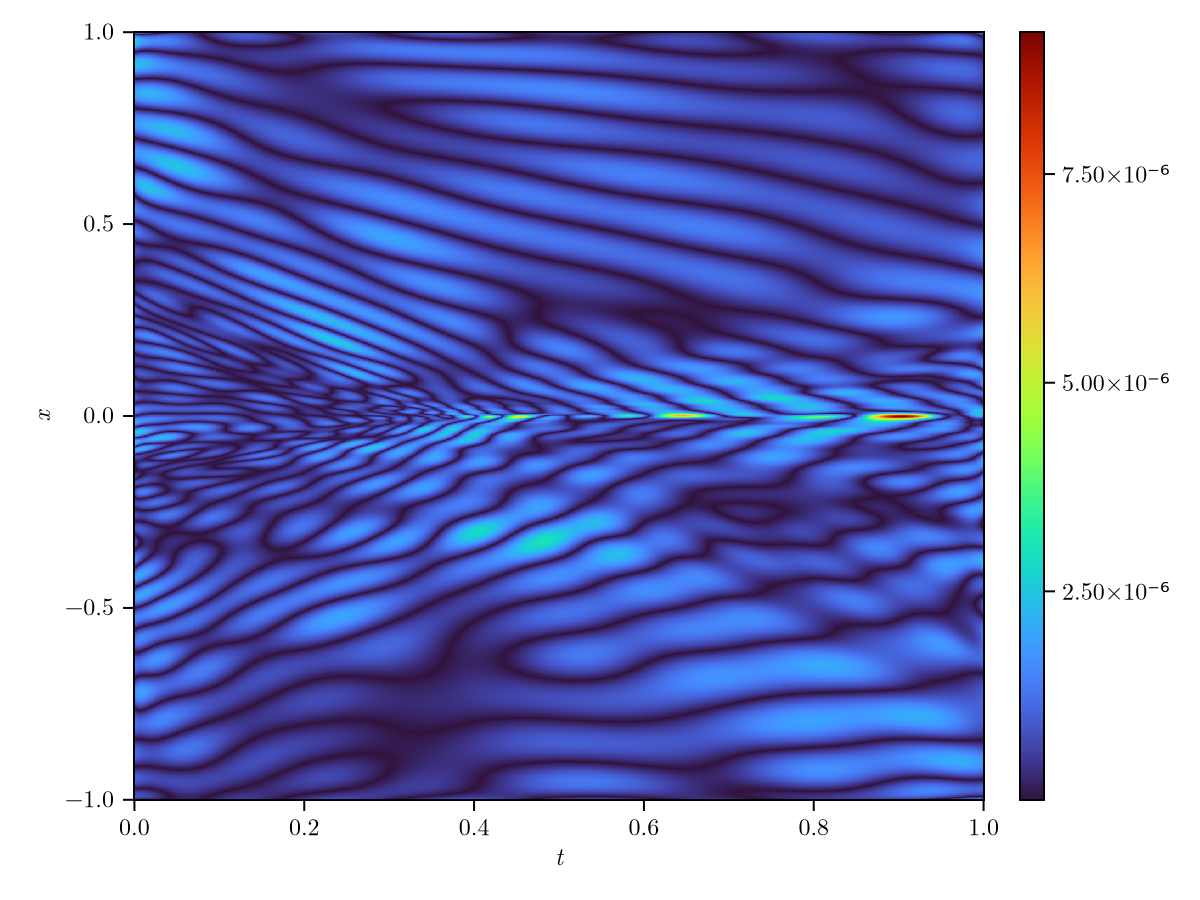
    }
    \caption{Absolute point-wise error}
    \label{fig:burgers-aq-errors}
  \end{subfigure}
  \caption{Adaptive quadrature solution, final adaptive partition and absolute point-wise errors for the viscous Burgers' problem using the AQ algorithm.} 
  \label{fig:burgers-aq-results}
\end{figure}

We begin the discussion with~\cref{fig:burgers-aq-results}, which demonstrates
the high accuracy of the solution obtained via the adaptive quadrature
strategy. Point-wise errors remain consistently low, even in the vicinity of
the shock. The final quadrature mesh successfully adapts to the shock's
progression, increasing partition density in the high-gradient region while
exhibiting clear anisotropy as the shock sharpens over time. Conversely, the
mesh remains sparse in regions away from the shock, particularly as time $t$
increases, which suggests near-optimal behaviour.

The point-wise errors in ~\cref{fig:burgers-aq-results} remain small across the domain, with no severe
high-frequency oscillations around the shock. Their maximum magnitude is about
half that reported in~\cite{UrbanPINNOpt2025}, and the final relative
$L^2$ error is $1.44 \times 10^{-6}$. This remains a substantial improvement,
especially because we enforce the boundary conditions through a penalty term
rather than the strong enforcement used in~\cite{UrbanPINNOpt2025,KiyaniOpt2025}. Furthermore, our
approach demonstrates a clear advantage in data efficiency: the maximum number
of training quadrature points used throughout the simulation was \num{8134}
(equivalent to $166$ partitions across the full space-time domain, achieved
with only $18$ \ac{aq} refreshes). By contrast, the setup in
\cite{UrbanPINNOpt2025} uses \num{10000} \ac{mc} points, which were resampled every \num{500}
epochs.
\begin{figure}
  \centering
  \begin{subfigure}[t]{1.0\linewidth}
    \centering
    \includegraphics[width=0.6\linewidth]{
        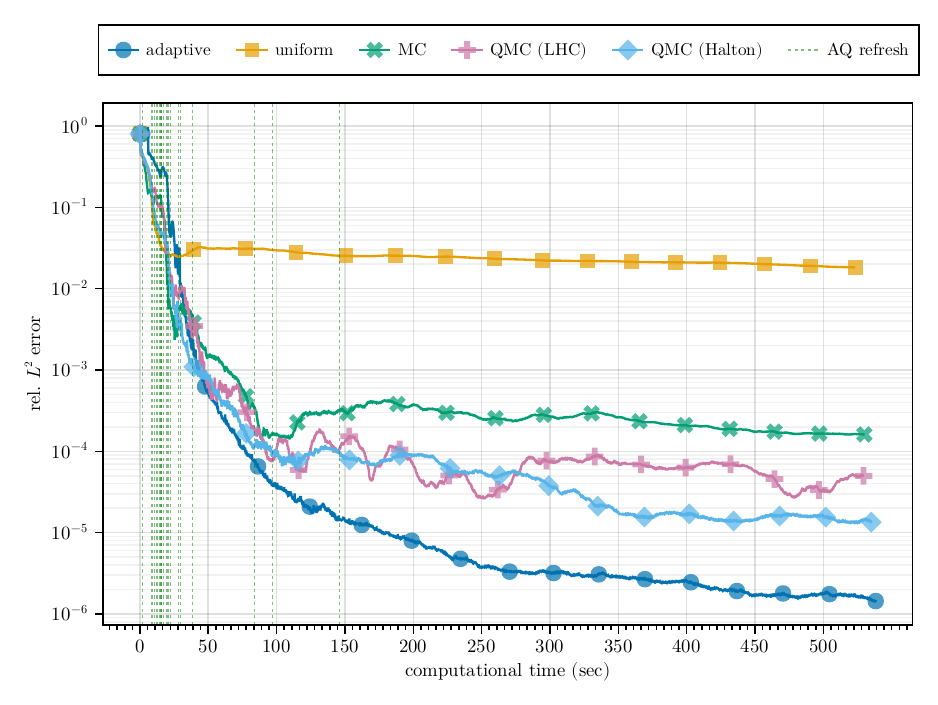
    }
    \caption{Relative $L^2$ error history comparison across all considered quadrature strategies.}
    \label{fig:burgers-error-history-panel}
  \end{subfigure}
  \medskip
  \begin{subfigure}[t]{1.0\linewidth}
    \centering
    \includegraphics[width=0.82\linewidth,trim={0 0 0 44},clip]{
        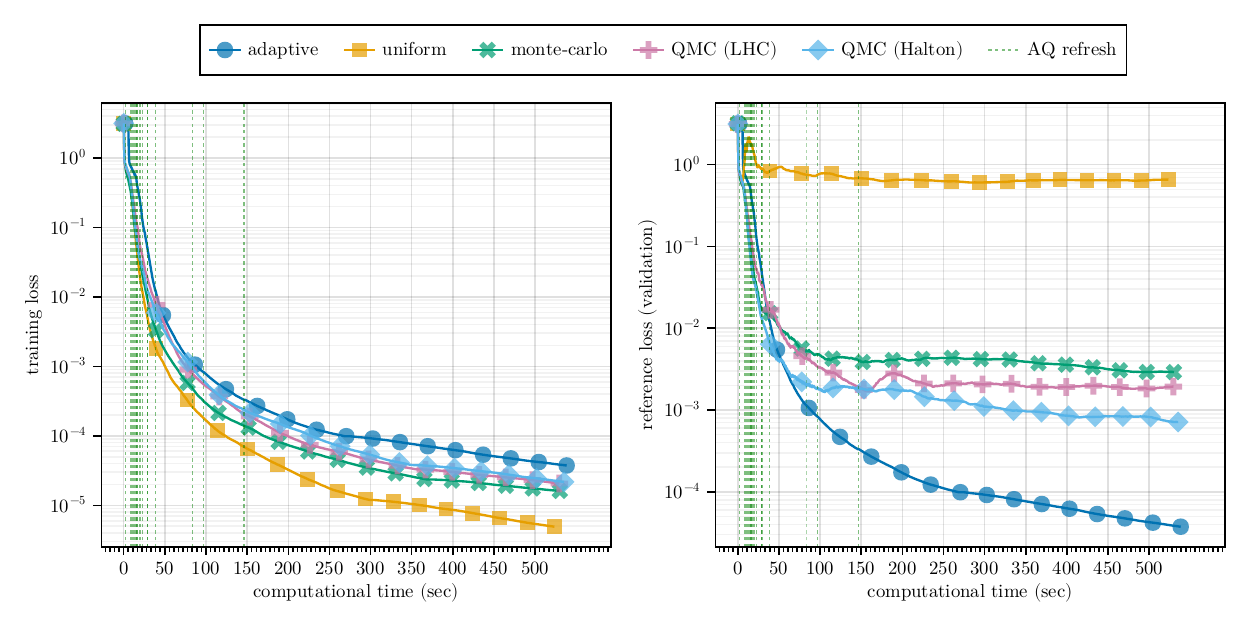
    }
    \caption{Training and reference loss history comparison across all considered quadrature strategies.}
    \label{fig:burgers-loss-history-panel}
  \end{subfigure}
  \caption{Loss and error histories for the viscous Burgers' problem.}
  \label{fig:burgers-error-history-comparison}
\end{figure}

A comparison of the considered quadrature strategies is presented in
\cref{fig:burgers-error-history-panel}. For the same total computational
expenditure, near the final stage where other approaches stagnate, our adaptive
quadrature strategy achieves a relative $L^2$ error that is an order of
magnitude lower—and continuing to improve—than the next best method, \ac{qmc}
(Halton) quadrature. The training and reference losses remain well aligned in
\cref{fig:burgers-loss-history-panel}, indicating good generalisation
throughout training. The contrast between the training and reference loss
histories also clearly illustrates the effect of \emph{overfitting}.
Notably, the strong correlation between the reference loss and the relative
$L^2$ error throughout the training process underscores the practical
importance of the reference loss; it serves not only to identify quadrature
insufficiency but also as a reliable metric for assessing overall convergence.
Furthermore, it is noteworthy that uniform quadrature exhibits the poorest
performance, suffering from overfitting and stagnating in accuracy very early
in the training process. This behaviour is consistent with the 2D
function-approximation benchmark in~\Cref{subsec:fa-2d-problem}, which also
features a sharp target profile. As in that case, \ac{qmc} (Halton) remains
the most effective non-adaptive strategy.

In summary, the proposed adaptive quadrature strategy demonstrates a clear
superiority over traditional static and quasi-random approaches, achieving an
order of magnitude higher precision while maintaining significantly greater
data efficiency.

\subsection{The (1+1)D Korteweg-De Vries (KdV) equation}
\label{subsec:kdv-problem}

Next, we consider the Korteweg-De Vries (KdV) equation
\begin{align} \label{eq:kdv-problem}
\alpha\frac{\partial u}{\partial t} + \beta u \frac{\partial u}{\partial x} + \gamma \frac{\partial^3 u}{\partial x^3} &= 0, %\quad x \in [0,L], \; t \in [0,5],\\
\end{align}
with the domain and boundary conditions specified in
\cite{UrbanPINNOpt2025}. This problem exhibits nonlinear dispersive
behaviour through a convection term similar to that of the Burgers equation,
combined with a third-order term that introduces dispersion. Because it lacks a
viscous second-order term, the problem provides a demanding test for numerical
solvers. 
The
system yields a complex solution characterised by the interaction of two
solitons, making it a demanding approximation problem. An analytical solution
is also available, as constructed in~\cite{UrbanPINNOpt2025}, which makes
the benchmark especially useful.
\begin{figure}
  \centering
  \begin{subfigure}[t]{0.48\linewidth}
    \includegraphics[width=\linewidth]{
        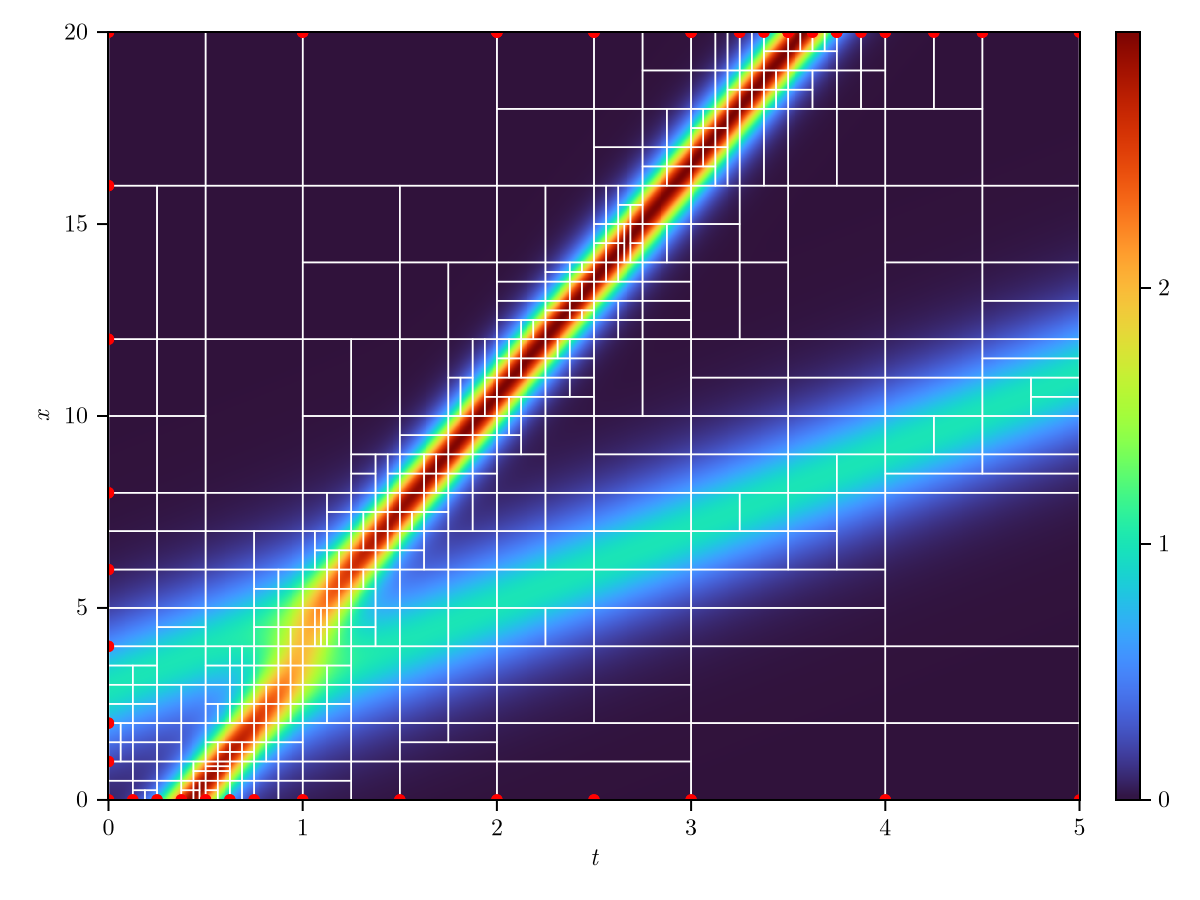
    }
    \caption{NN approximation with final AQ partition}
    \label{fig:kdv-aq-sol-mesh}
  \end{subfigure}
  \begin{subfigure}[t]{0.48\linewidth}
    \includegraphics[width=\linewidth]{
        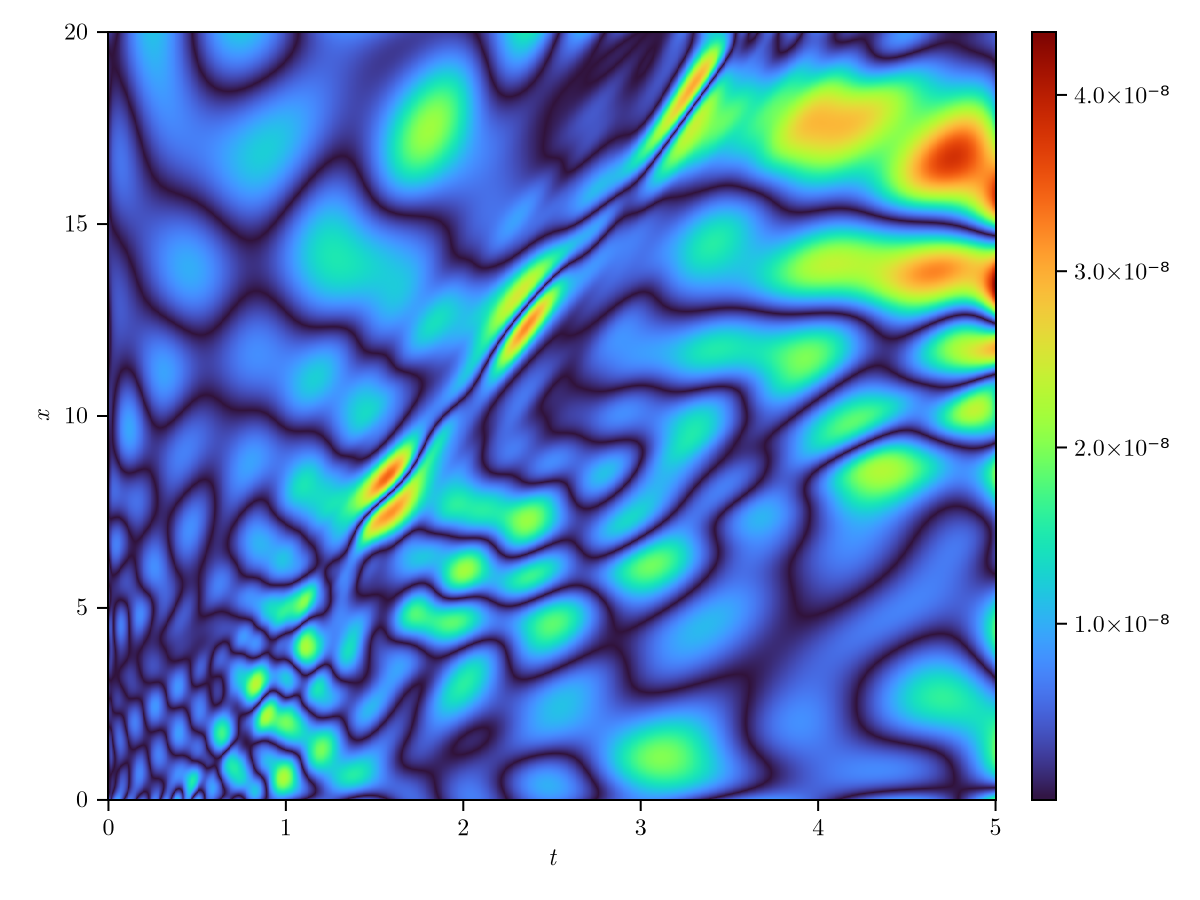
    }
    \caption{Absolute point-wise errors}
    \label{fig:kdv-aq-errors}
  \end{subfigure}
  \caption{Adaptive quadrature solution, final adaptive partition and absolute point-wise errors for the (1+1)D Korteweg-De Vries (KdV) problem
       ~\eqref{eq:kdv-problem} using the AQ algorithm.} 
  \label{fig:kdv-aq-results}
\end{figure}
Using the same \ac{nn} architecture (3 layers with 30 neurons) as in
\cite{UrbanPINNOpt2025} and training for only \num{20000} SSBroyden
optimiser epochs—omitting the Adam pre-training used in
\cite{UrbanPINNOpt2025}—we obtain an excellent approximation of the target
solution. 

As shown in~\Cref{fig:kdv-aq-sol-mesh}, the interaction of the waves
is captured particularly well, with the final adaptive quadrature mesh closely
tracking the solution profile. Supporting this, the maximum point-wise errors
(\Cref{fig:kdv-aq-errors}) are approximately two orders of magnitude lower than
those reported in~\cite[Fig. 15]{UrbanPINNOpt2025}. The final AQ partition is
also consistent with refinement first concentrating on the stronger soliton and
then resolving the weaker interaction region.
Furthermore, the maximum number of partitions reached is $452$, totalling
\num{11300} primal quadrature points, and this required only six AQ refreshes
from an initial $5 \times 5$ mesh. The final relative
$L^2$ and $H^1$ errors are $1.54 \times 10^{-8}$ and $1.85 \times 10^{-8}$,
respectively, with a total runtime of \num{1203} seconds. By comparison,
\cite[Table 5]{UrbanPINNOpt2025} reported a relative $L^2$ error of
approximately $6 \times 10^{-6}$ using \num{15000} MC points resampled every
500 epochs, with even larger errors cited in the references therein. These
results again highlight the superior performance and quadrature efficiency of
the adaptive strategy.
\begin{figure}
  \centering
  \begin{subfigure}[t]{1.0\linewidth}
    \centering
    \includegraphics[width=0.82\linewidth]{
        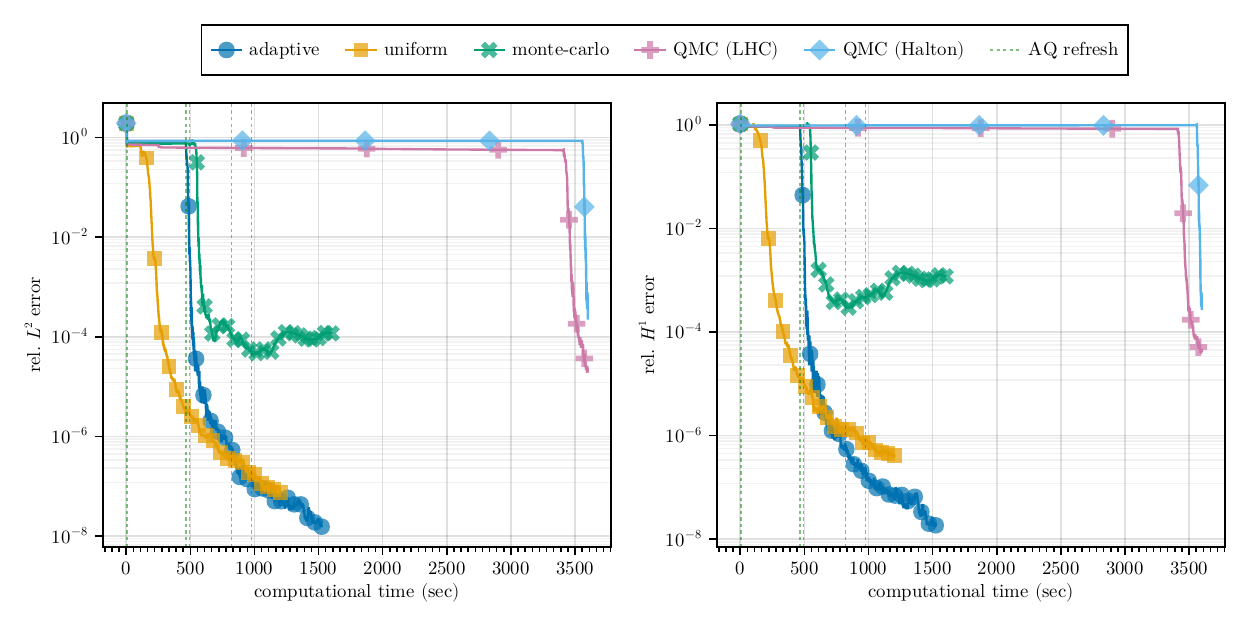
    }
    \caption{$L^2$ and $H^1$ error history comparison across all considered quadrature strategies.}
    \label{fig:kdv-error-history-panel}
  \end{subfigure}
  \medskip
  \begin{subfigure}[t]{1.0\linewidth}
    \centering
    \includegraphics[width=0.82\linewidth,trim={0 0 0 38},clip]{
        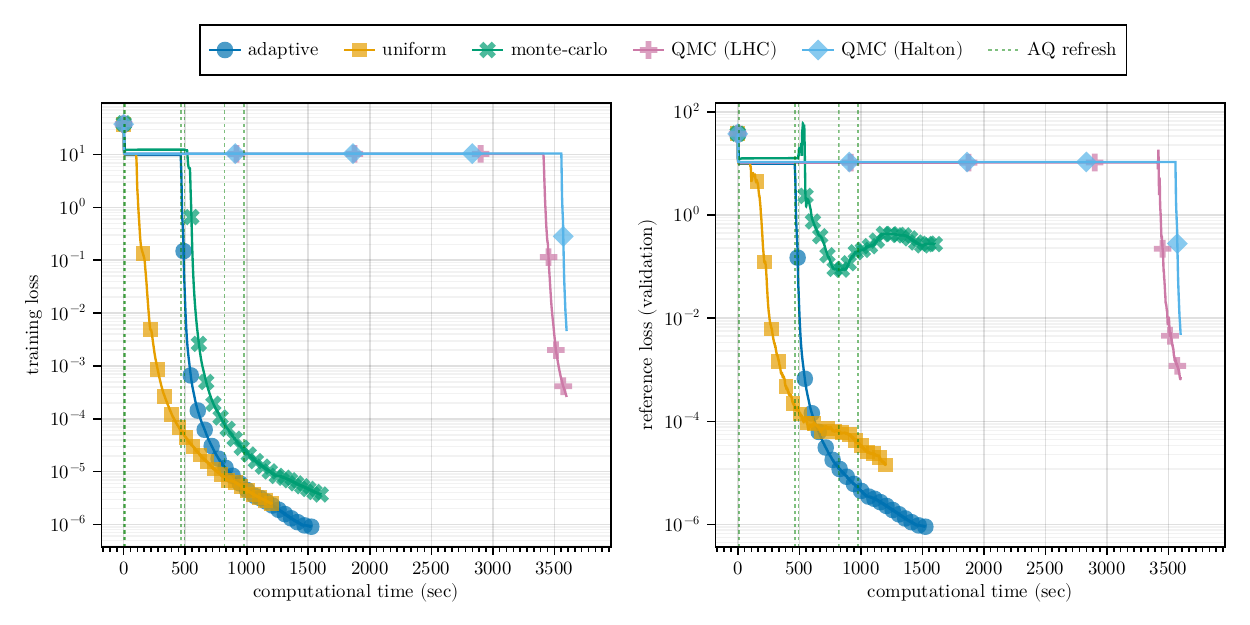
    }
    \caption{Training and reference loss history comparison across all considered quadrature strategies.}
    \label{fig:kdv-loss-history-panel}
  \end{subfigure}
  \caption{Loss and error histories for the 
  (1+1)D Korteweg-De Vries (KdV) problem.}
  \label{fig:kdv-error-history-comparison}
\end{figure}

A comparison of the histories in~\Cref{fig:kdv-error-history-panel,fig:kdv-loss-history-panel} shows that
the adaptive strategy spends
substantial wall-clock time during the first \num{1000} epochs. This is driven by larger
line-search counts per epoch, likely caused by the difficult early loss
landscape induced by the penalty treatment of the boundary conditions. The AQ
refreshes then help the optimiser leave this stagnant phase, after which the
errors and reference loss improve rapidly and monotonically. Moreover, only the
\ac{aq} strategy exhibits training and reference losses that match in
\Cref{fig:kdv-loss-history-panel}. Uniform quadrature exits this initial stage
earlier and attains lower early-time errors in
\Cref{fig:kdv-error-history-panel}, but this advantage is short-lived. Without
adaptation, the generalisation gap widens and the final errors remain above
those of AQ. In contrast, the Monte Carlo
(\ac{mc}) strategy undergoes a similarly slow initial phase, then improves
nonlinearly, first overfitting and only later recovering some
generalisation. Its progress remains markedly non-monotonic and it ends with
the largest final errors among the tested methods, suggesting convergence to a
suboptimal local minimum. The behaviour of both \ac{qmc} variants is more
extremely still: they remain in a slow-convergence regime for much longer before
their late error reduction near the \num{3600}-second runtime limit, as shown
in~\Cref{fig:kdv-error-history-panel,fig:kdv-loss-history-panel}.

Overall, the adaptive quadrature strategy achieves a superior approximation
compared to the next best result, which was attained by uniform quadrature. In
particular, the disparity in the relative $H^1$ errors is significant,
suggesting that the adaptive approach provides superior generalisation to the
exact solution. The marked improvement in $H^1$ accuracy, which accounts for
derivatives of the solution, underscores the ability of adaptive quadrature to
capture the underlying physics of the soliton interactions more faithfully than
static sampling methods.

\subsection{The (1+1)D Cahn-Hilliard problem}

Next, we consider the fourth-order Cahn-Hilliard equation, which models phase
separation in a binary fluid mixture. The solution exhibits multiple sharp
phase transitions and is generally more challenging computationally than
second-order phase-field models, primarily because of the fourth-order spatial
derivatives. We adopt a benchmark problem from
\cite{Wight2021AdaptiveCH}, where the
scalar field $u$ represents the relative concentrations of the binary
components. The problem is defined by the following governing equations,
supplemented with initial and periodic boundary conditions:
\begin{align} \label{eq:cahn-hilliard-problem}
u_t - (\gamma_2(u^3 - u) - \gamma_1 u_{xx})_{xx} &= 0, \quad x \in [-1,1], \; t \in [0,1],\\
u(0,x) &= - \cos(2\pi x), \\
u(t,-1) &= u(t,1), \\
u_x(t,-1) &= u_x(t,1),
\end{align}
where the parameters $\gamma_2 = 0.01$ and $\gamma_1 = 10^{-6}$.

For this problem, we employ a 4-layer \ac{nn} with 50 neurons per
hidden layer, a significantly narrower architecture than the
128-neuron-per-layer model used in~\cite{Wight2021AdaptiveCH}. Moreover,
while the approach in~\cite{Wight2021AdaptiveCH} uses curriculum learning together with the
L-BFGS optimiser, we adopt a full-batch space-time formulation solved with
SSBroyden.
\begin{figure}
  \centering
  \begin{subfigure}[t]{0.48\linewidth}
    \includegraphics[width=\linewidth]{
        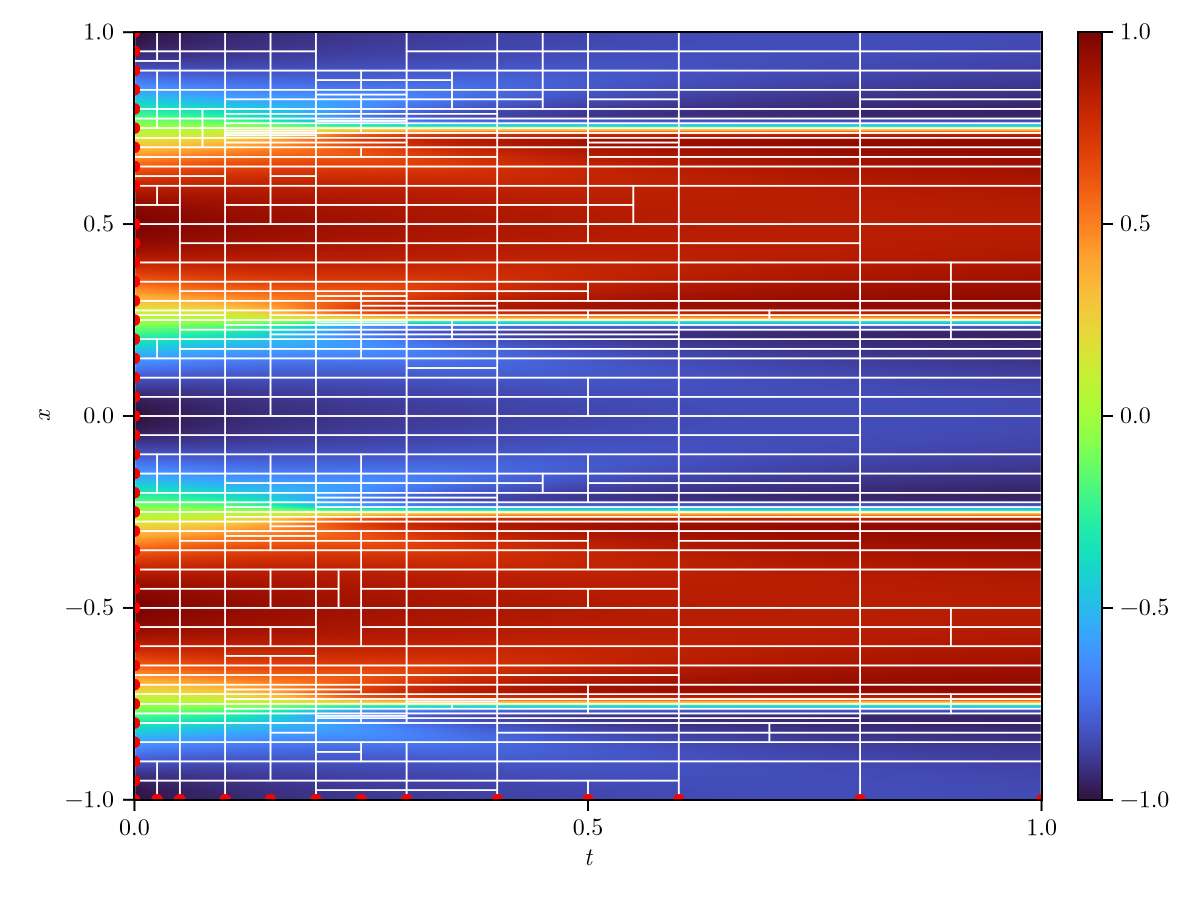
    }
    \caption{NN approximation with final AQ partition}
    \label{fig:cahn-hilliard-aq-sol-mesh}
  \end{subfigure}
  \begin{subfigure}[t]{0.48\linewidth}
    \includegraphics[width=\linewidth]{
        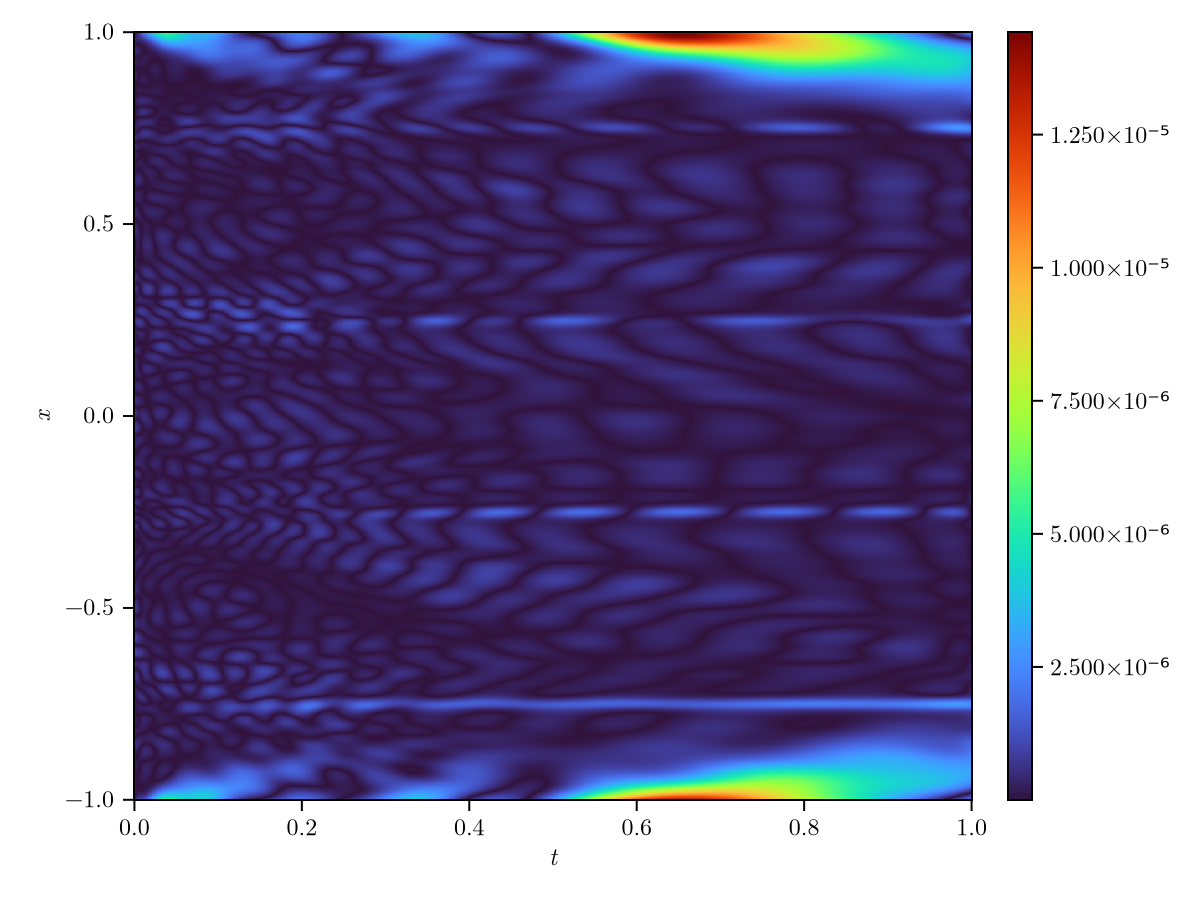
    }
    \caption{Absolute point-wise errors}
    \label{fig:cahn-hilliard-aq-errors}
  \end{subfigure}
  \caption{
    Adaptive quadrature solution, final adaptive partition and absolute point-wise errors for 
        the Cahn-Hilliard problem~\eqref{eq:cahn-hilliard-problem}
         using the AQ algorithm.} 
  \label{fig:cahn-hilliard-aq-results}
\end{figure}
The point-wise errors presented in~\Cref{fig:cahn-hilliard-aq-errors}
demonstrate the reliability of the approximation achieved by the adaptive
quadrature (AQ) strategy. As shown in~\Cref{fig:cahn-hilliard-aq-sol-mesh}, the
model captures the sharp phase separation with high fidelity. Notably, the
final AQ partition—obtained through 15 adaptation cycles starting from a $5 \times
5$ uniform grid—exhibits particularly high-aspect-ratio (thin) cells. This
anisotropy underscores the algorithm's focus on spatial refinement at sharp phase transitions, while largely bypassing temporal
adaptivity due to the relatively stable dynamics along the time axis.
\begin{figure}
  \centering
  \begin{subfigure}[t]{1.0\linewidth}
    \centering
    \includegraphics[width=0.5\linewidth]{
        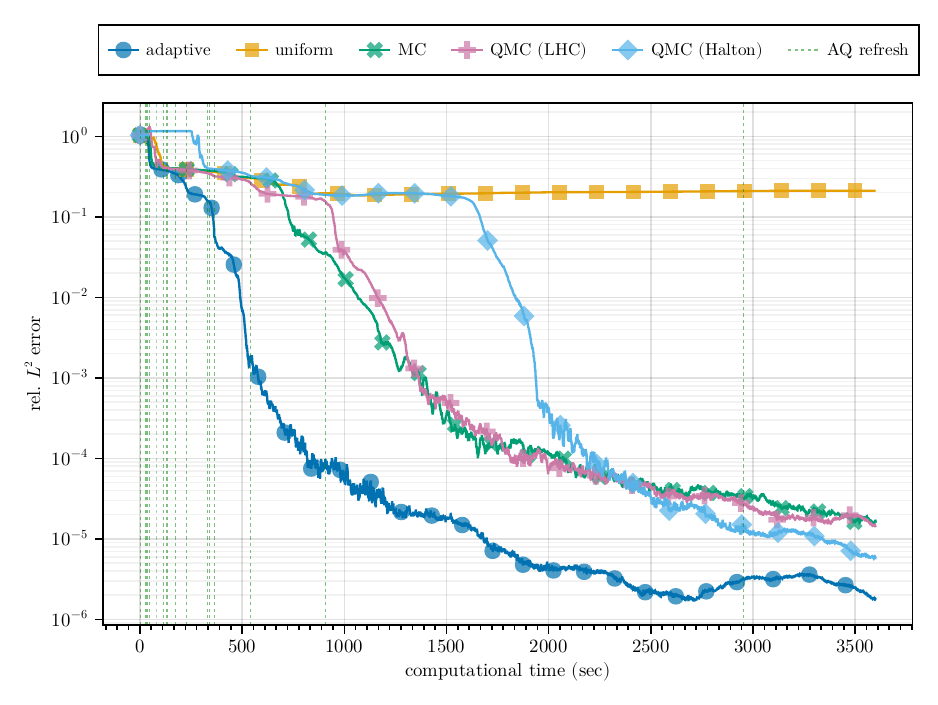
    }
    \caption{$L^2$ and $H^1$ error history comparison across all considered quadrature strategies.}
    \label{fig:cahn-hilliard-error-history-panel}
  \end{subfigure}
  \medskip
  \begin{subfigure}[t]{1.0\linewidth}
    \centering
    \includegraphics[width=0.82\linewidth,trim={0 0 0 38},clip]{
        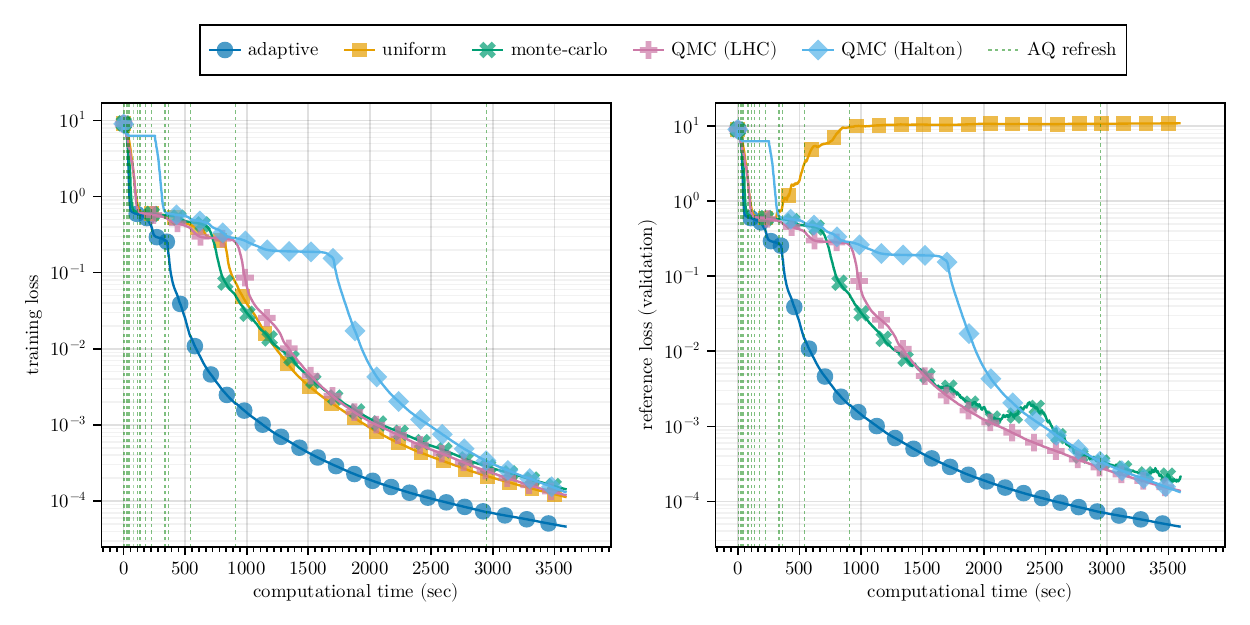
    }
    \caption{Training and reference loss history comparison across all considered quadrature strategies.}
    \label{fig:cahn-hilliard-loss-history-panel}
  \end{subfigure}
  \caption{
    Loss and error histories for the 
  Cahn-Hilliard problem~\eqref{eq:cahn-hilliard-problem}.}
  \label{fig:cahn-hilliard-error-history-comparison}
\end{figure}

A unified comparison of the various quadrature strategies, shown through the
relative error histories in~\Cref{fig:cahn-hilliard-error-history-panel} and
the loss histories in~\Cref{fig:cahn-hilliard-loss-history-panel},
demonstrates that the adaptive quadrature (AQ) strategy, supported by frequent
refreshes, successfully bypasses the stagnation phases encountered by the
other methods. Consequently, it achieves superior relative $L^2$ accuracy
significantly earlier in training.
With the exception of uniform quadrature, the other strategies show relatively
close alignment between training and reference losses. This does not,
however, translate into equally accurate solutions. The adaptive run reaches a
maximum of $695$ partitions, equivalent to \num{34055} primal quadrature
points, and still achieves substantially lower errors than the alternatives,
whereas uniform quadrature eventually succumbs to irreversible overfitting.

Here too, the reference loss history is consistent with the training loss but
not so strongly correlated with the relative $L^2$ error history. This
discrepancy is likely attributable to the penalty-based enforcement of
periodic boundary conditions.

\subsection{The 2D arc wavefront problem with sharp gradients}
\label{sec:poisson-aq}

We consider a Poisson problem with 
with forcing term and Dirichlet boundary conditions 
such that the solution of the problem is the manufactured solution
\begin{equation}\label{eq:poisson-exact-solution}
u(x,y) = \mathrm{atan}\left(100\left(\sqrt{(x + 0.05)^2 + (y + 0.05)^2} - 0.7\right)\right),
\end{equation}
 which
displays sharp gradients along an arc wavefront. The problem is
defined on the unit square domain $\Omega = [0,1]^2$. This problem is drawn
from the NIST-\ac{amr} benchmark collection~\cite{NISTAMR}, which is designed
to evaluate adaptive mesh refinement (\ac{amr}) strategies. It has also been
used in recent work on adaptive
finite element interpolated \acp{nn}~\cite{AdaptiveFEINNs2024}.
\begin{figure}
  \centering
  \begin{subfigure}[t]{0.48\linewidth}
    \includegraphics[width=\linewidth]{
        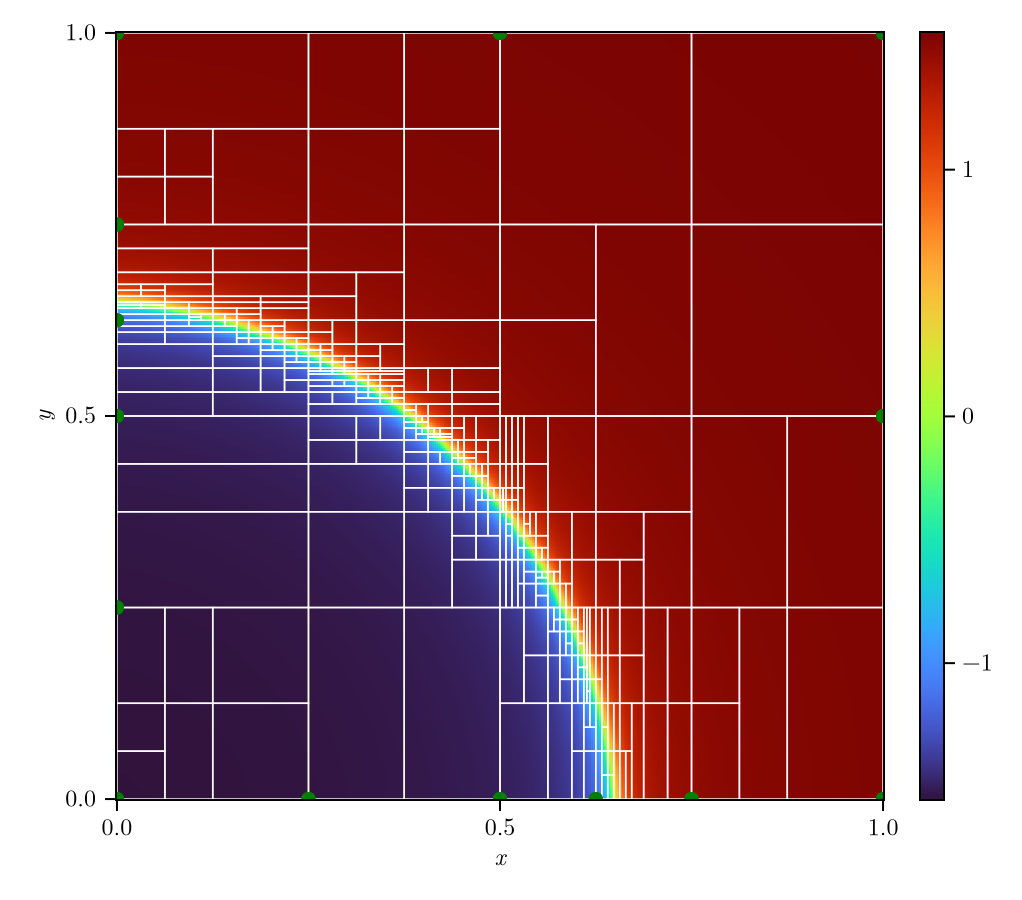
    }
    \caption{NN approximation with final AQ partition}
    \label{fig:poisson-aq-sol-mesh}
  \end{subfigure}
  \begin{subfigure}[t]{0.48\linewidth}
    \includegraphics[width=\linewidth]{
        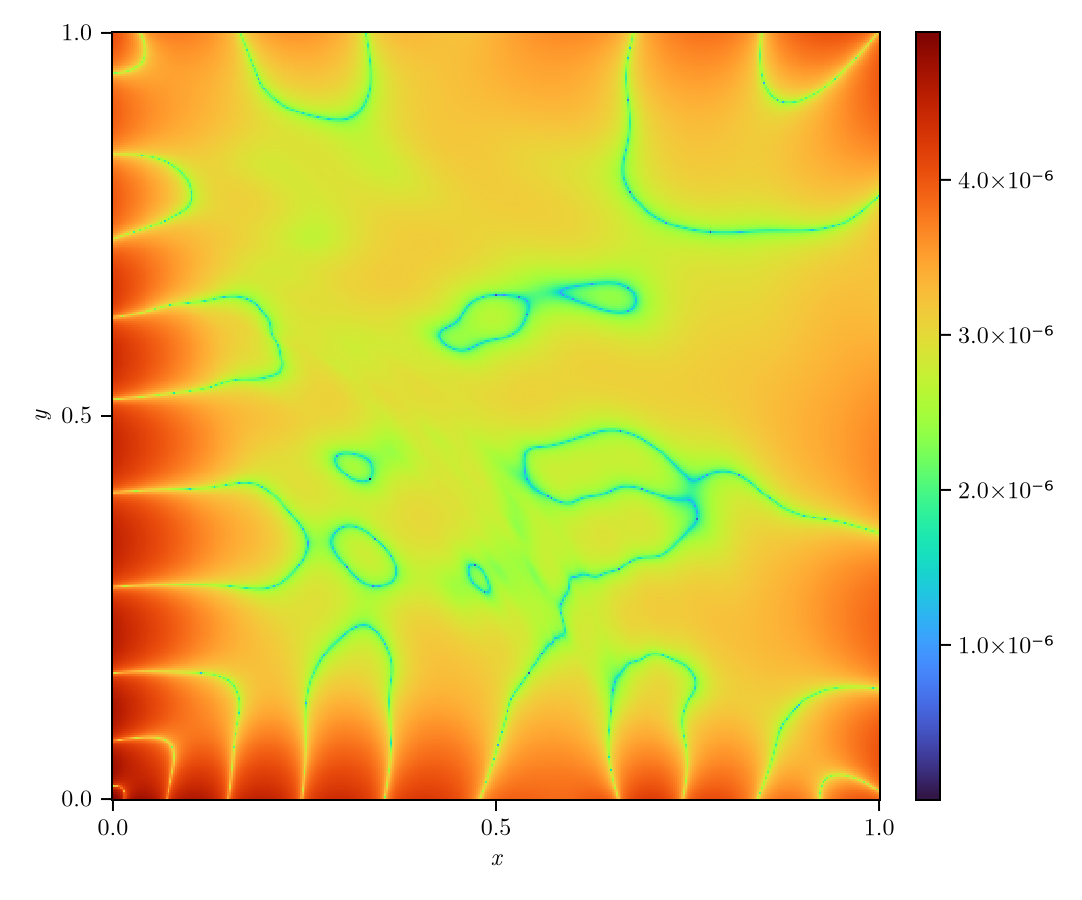
    }
    \caption{Absolute point-wise errors}
    \label{fig:poisson-aq-errors}
  \end{subfigure}
  \caption{Adaptive quadrature solution, final adaptive partition and absolute point-wise errors for 
        the 2D arc wavefront Poisson problem using the AQ algorithm.} 
  \label{fig:poisson-aq-results}
\end{figure}
For this problem, we employ a quadrature order combination of $(7,10)$, using
a base mesh that covers the entire domain without initial partitioning. The
\ac{aq} algorithm is configured with tolerances of
$\texttt{rtol} = 10^{-3}$ and $\texttt{refresh\_tol} = 10^{-2}$. The trial
network architecture comprises 4 hidden layers, each with a width of 50
neurons and the $\tanh$ activation function. The model is trained for
\num{15000} epochs with a fixed Dirichlet penalty $\gamma_D = 10.0$.

We begin with~\Cref{fig:poisson-aq-errors}, which shows remarkably low
point-wise errors that remain nearly uniform across the domain. Errors increase
slightly near the boundaries, reflecting the intrinsic difficulty of capturing
boundary data. The figure also highlights the ability of the $h$-adaptive
process to resolve the high-gradient region efficiently, capturing the
anisotropy of the sharp arc and concentrating effort where the solution is
less regular. The adaptive run requires only 11 refreshes, all within the first
$\num{5100}$ epochs. After that, the quadrature stabilises and no further
refinement is needed to maintain the prescribed tolerances.
\begin{figure}
  \centering
  \begin{subfigure}[t]{1.0\linewidth}
    \centering
    \includegraphics[width=0.82\linewidth]{
        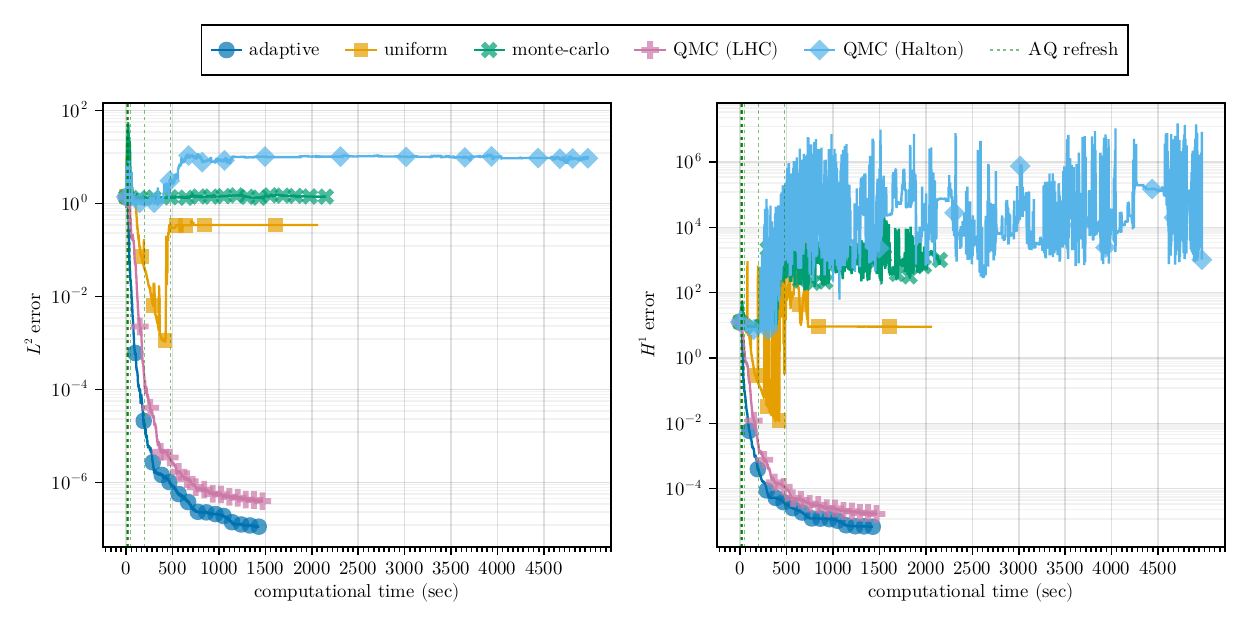
    }
    \caption{$L^2$ and $H^1$ error history comparison across all considered quadrature strategies.}
    \label{fig:poisson-error-history-panel}
  \end{subfigure}
  \medskip
  \begin{subfigure}[t]{1.0\linewidth}
    \centering
    \includegraphics[width=0.82\linewidth,trim={0 0 0 38},clip]{
        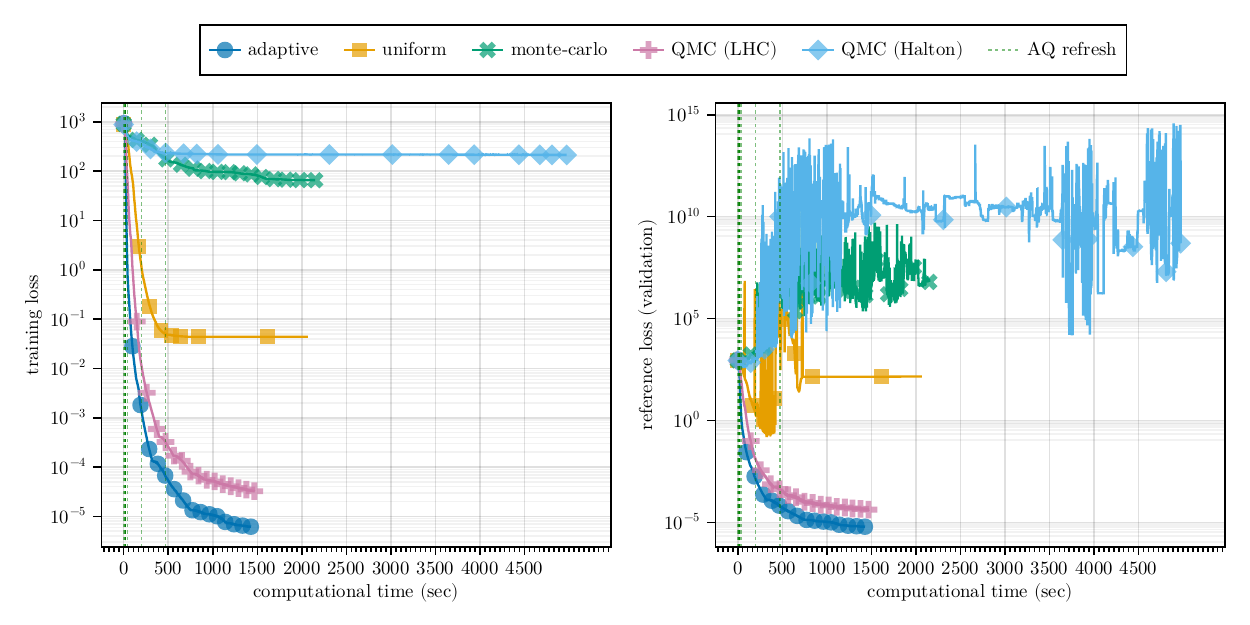
    }
    \caption{Training and reference loss history comparison across all considered quadrature strategies.}
    \label{fig:poisson-loss-history-panel}
  \end{subfigure}
  \caption{Loss and error histories for the 2D arc wavefront 
  diffusion problem.}
  \label{fig:poisson-error-history-comparison}
\end{figure}
Finally, a comparison across the various quadrature strategies in
\Cref{fig:poisson-error-history-panel} shows that the \ac{aq} strategy
significantly outperforms uniform, \ac{mc}, and \ac{qmc} (Halton)
quadratures. These alternative schemes exhibit clear overfitting, while their
trajectories suggest a highly non-convex or difficult loss landscape, a point
reinforced by the loss histories in~\Cref{fig:poisson-loss-history-panel} and
their longer run times. In contrast, the \ac{qmc} (\ac{lhc}) approach
remains competitive, yielding error levels only slightly above the \ac{aq}
case within a similar temporal budget.

\subsection{The 2D L-shaped Poisson problem with corner singularity}
\label{sec:lshaped-aq}

We next consider a Poisson problem with Dirichlet boundary conditions on the
2D non-convex L-shaped domain $\Omega = [-1,1]^2\backslash[-1,0]^2$, together
with the manufactured solution
\begin{equation}\label{eq:lshaped-exact-solution}
u(x) = r^{\frac{2}{3}} \sin(2 \theta / 3 + \pi / 3)
\end{equation}
which determines the boundary data and the source term. Although $u$
belongs only to $H^{3/2 - \epsilon} (\Omega)$, the source term vanishes in
$\Omega$, so the primal deep least squares formulation remains well posed.

For this experiment, we adopt a configuration close to that of the previous
problem (\Cref{sec:poisson-aq}). The main difference is the network
architecture, which now has 5 hidden layers while maintaining a width of 50
neurons. The \ac{aq} hyperparameters remain unchanged, but the initial base
mesh is partitioned into 3 blocks aligned with the L-shaped geometry. Training
is run for \num{10000} iterations.
\begin{figure}
  \centering
  \begin{subfigure}[t]{0.48\linewidth}
    \includegraphics[width=\linewidth]{
        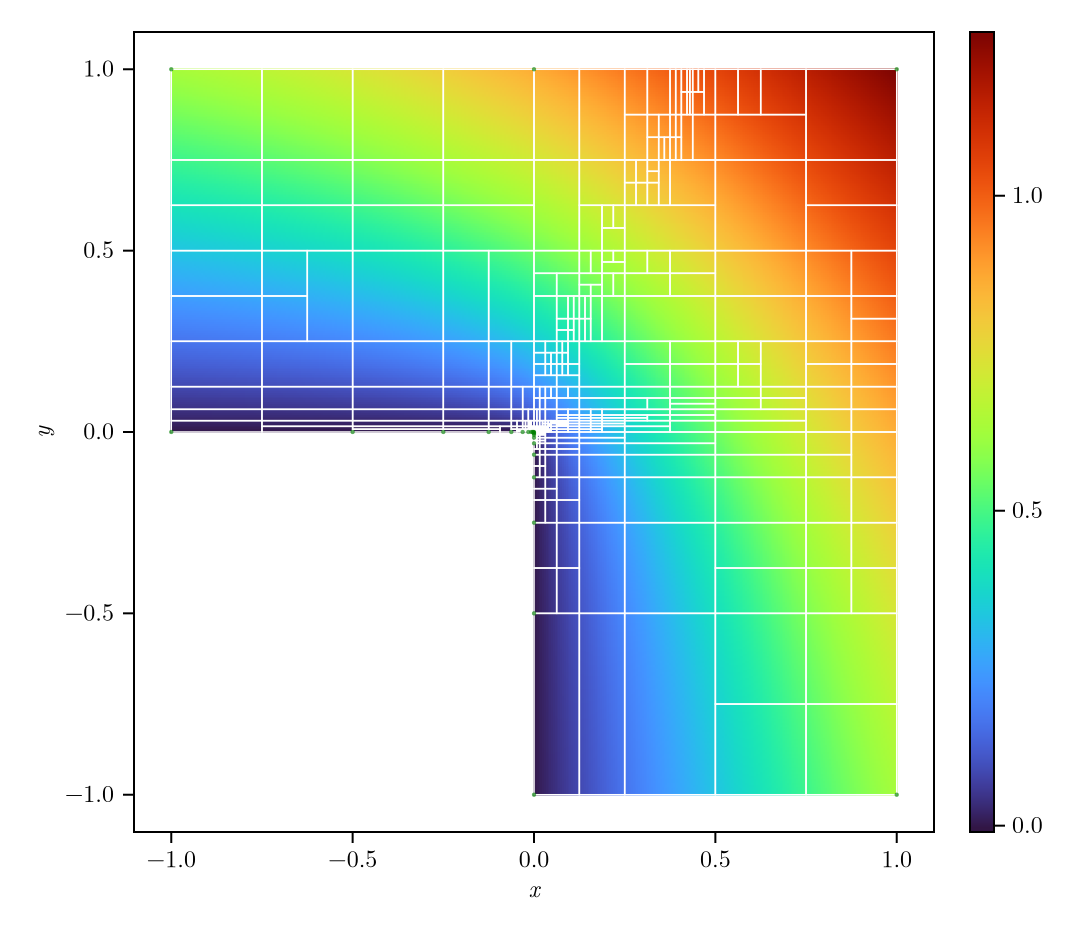
    }
    \caption{NN approximation with final AQ partition}
    \label{fig:lshaped-aq-sol-mesh}
  \end{subfigure}
  \begin{subfigure}[t]{0.48\linewidth}
    \includegraphics[width=\linewidth]{
        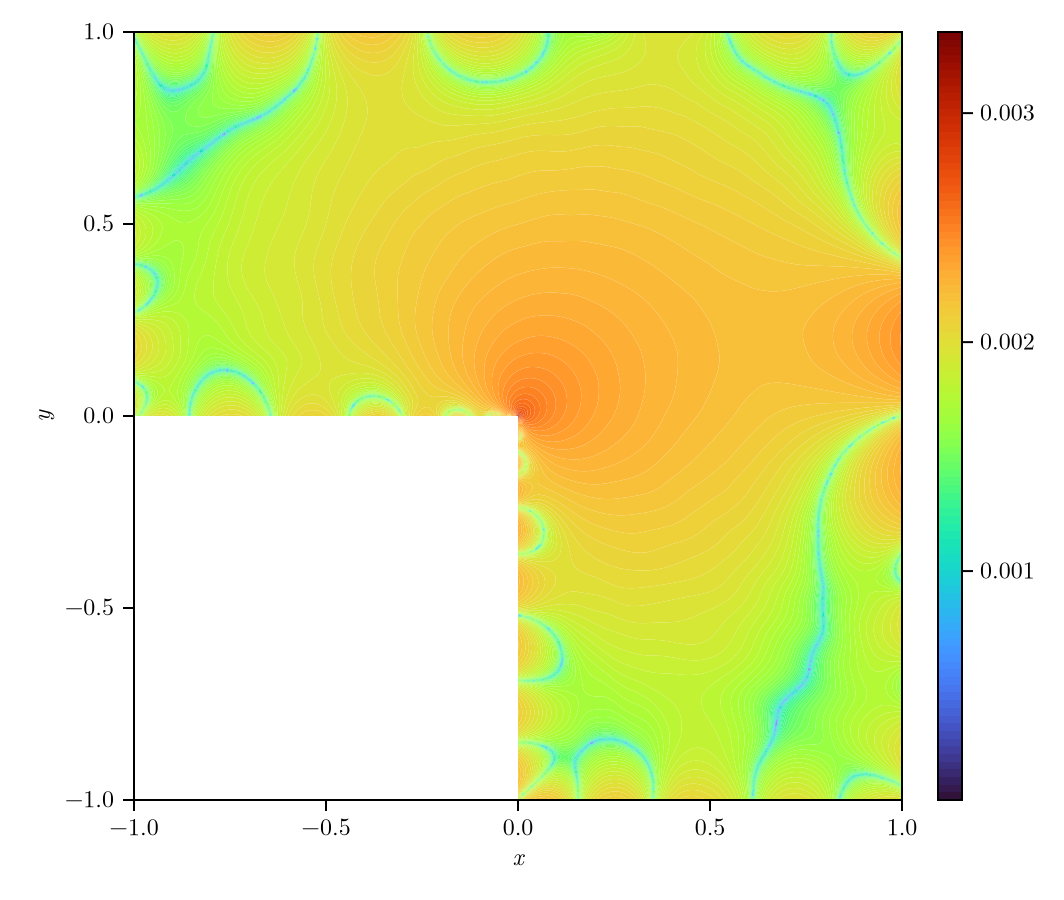
    }
    \caption{Absolute point-wise errors}
    \label{fig:lshaped-aq-errors}
  \end{subfigure}
  \caption{
  Adaptive quadrature solution, final adaptive partition and absolute point-wise errors for the 
        the L-shaped Poisson problem using the AQ algorithm.} 
  \label{fig:lshaped-aq-results}
\end{figure}
\Cref{fig:lshaped-aq-results} shows that the \ac{aq} approximation is highly
accurate overall. The point-wise errors are concentrated near the corner
singularity, while the remaining boundary regions are captured with high
fidelity. The training history records 17 mesh adaptation steps. These
refreshes become less frequent as training progresses, while the quadrature
error exhibits the expected drops at each \ac{aq} refinement point.

The autonomous nature of these \ac{aq} refreshes is particularly noteworthy
given the inherent difficulty of this problem, which typically challenges
traditional numerical methodologies due to the corner singularity and
non-convexity. The resulting refinement pattern concentrates quadrature
elements around the re-entrant corner, both in the domain interior and along
the singular boundary segments, so that the sharp gradients are resolved where
they matter most.
\begin{figure}
  \centering
  \begin{subfigure}[t]{1.0\linewidth}
    \centering
    \includegraphics[width=0.82\linewidth]{
        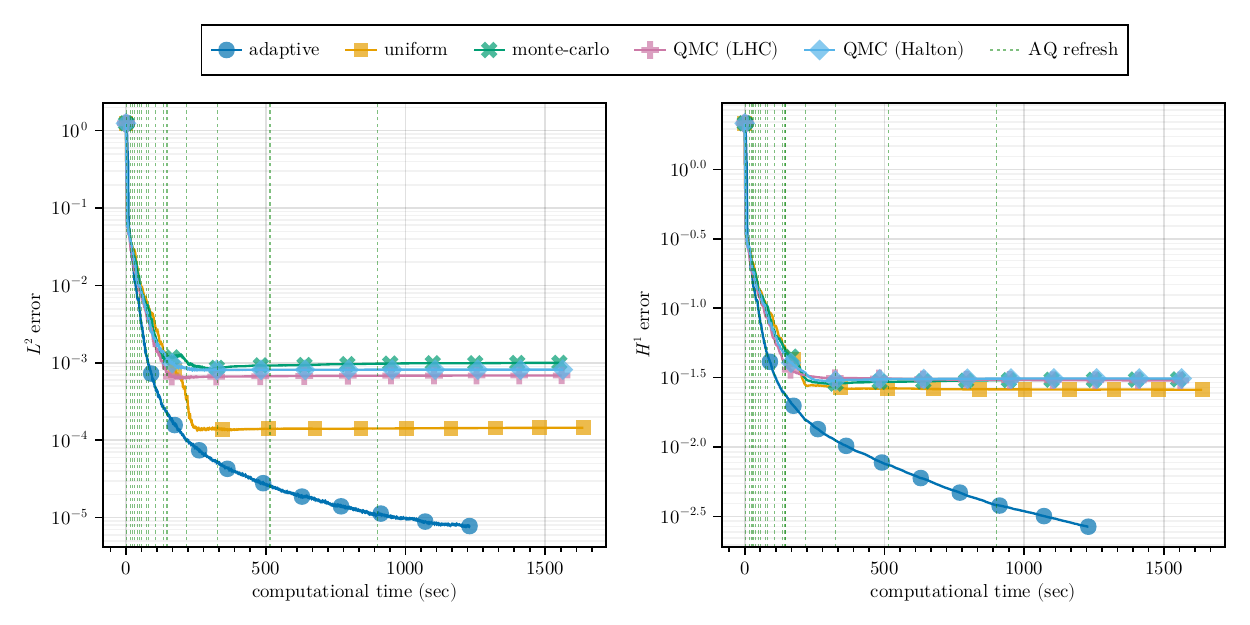
    }
    \caption{$L^2$ and $H^1$ error history comparison across all considered quadrature strategies.}
    \label{fig:lshaped-error-history-panel}
  \end{subfigure}
  \medskip
  \begin{subfigure}[t]{1.0\linewidth}
    \centering
    \includegraphics[width=0.82\linewidth,trim={0 0 0 38},clip]{
        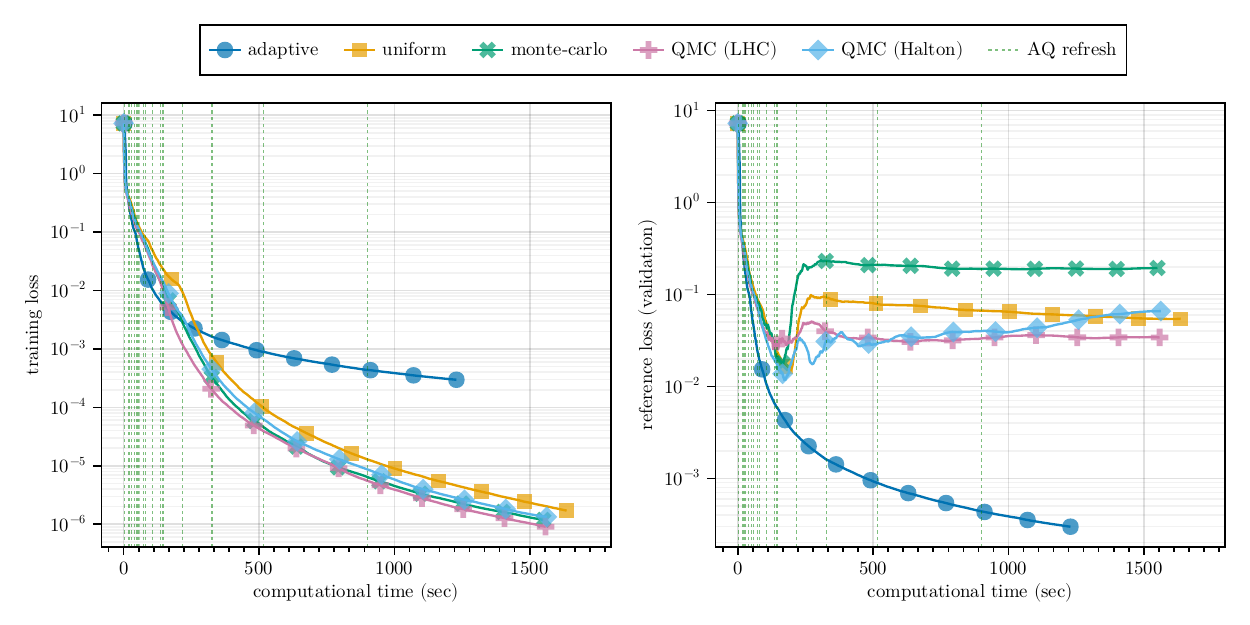
    }
    \caption{Training and reference loss history comparison across all considered quadrature strategies.}
    \label{fig:lshaped-loss-history-panel}
  \end{subfigure}
  \caption{Loss and error histories for the L-shaped Poisson problem.}
  \label{fig:lshaped-error-history-comparison}
\end{figure}
More importantly, the \ac{aq} strategy offers a clear advantage in
approximation accuracy while requiring at least $20\%$ less training time,
highlighting its quadrature efficiency. It maintains stable generalisation
throughout optimisation, with both $L^2$ and $H^1$ errors decreasing
monotonically. In contrast, the alternative approaches stall early. This
overfitting is clearly visible in~\Cref{fig:lshaped-loss-history-panel}. Only
the \ac{aq} strategy exhibits training and reference losses with the same
qualitative trend, up to a scaling factor. For the alternative approaches, the
reference loss plateaus despite the continued reduction of the training loss.

\subsection{Approximation on a non-trivial domain}
\label{sec:rhombi-aq}

We now consider a convection–diffusion–reaction problem on a non-convex double-rhombi
domain, as described in~\cite{MagueresseBadiaAdaptive2024}. The problem is defined as follows:
\begin{align} \label{eq:conv_diff}
- \pmb{\nabla}\cdot(\kappa \pmb{\nabla} u) + (\pmb{\beta} \cdot \pmb{\nabla}) u &= f \quad \hbox{in }  \Omega \subset \mathbb{R}^2,\\ 
u &= g \quad \hbox{on }  \partial \Omega. 
\end{align}
where the variable diffusion coefficient is given by $\kappa(x, y) = 2 +
\sin(xy)$ and the constant convection velocity is $\pmb{\beta} = [1, 2]^\top$.
The domain $\Omega$ is partitioned using an unstructured mesh (see~\Cref{fig:rhombi-aq-sol-mesh}). 
The source term $f$ and Dirichlet boundary data $g$ are determined via the
method of manufactured solutions, using:
\begin{equation*}
    u(x, y) = \tanh\left(25\left(x^2 + y^2 - 0.25\right)\right),
\end{equation*}
This solution profile exhibits sharp gradients concentrated along a circular
arc, which specifically intersects the top and bottom interior corners of the
double-rhombi domain, presenting a significant resolution challenge for
non-adaptive schemes.

For this problem, we employ a trial network architecture consisting of 4 hidden
layers with 30 neurons each, utilizing the $\tanh$ activation function. The
\ac{aq} algorithm is configured with a quadrature order combination of
$(7,10)$, with tolerances set to $\texttt{rtol} = 5 \times 10^{-3}$ and
$\texttt{refresh\_tol} = 5 \times 10^{-2}$. The network is trained for
$\num{5000}$ iterations, maintaining a constant Dirichlet penalty parameter of
$\gamma_D = 10.0$ throughout the optimisation process.
\begin{figure}
  \centering
  \begin{subfigure}[t]{0.48\linewidth}
    \includegraphics[width=\linewidth]{
        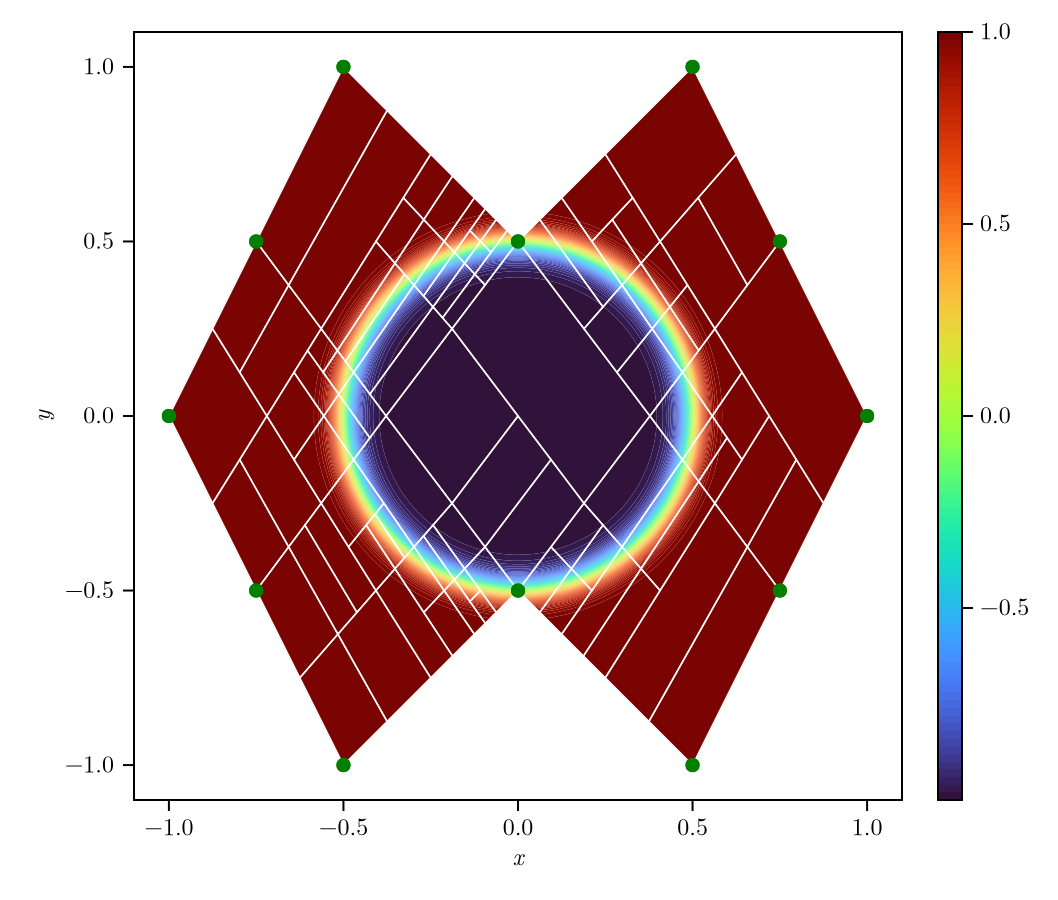 }
        \caption{NN approximation with final AQ partition}
    \label{fig:rhombi-aq-sol-mesh}
  \end{subfigure}
  \begin{subfigure}[t]{0.48\linewidth}
    \includegraphics[width=\linewidth]{
        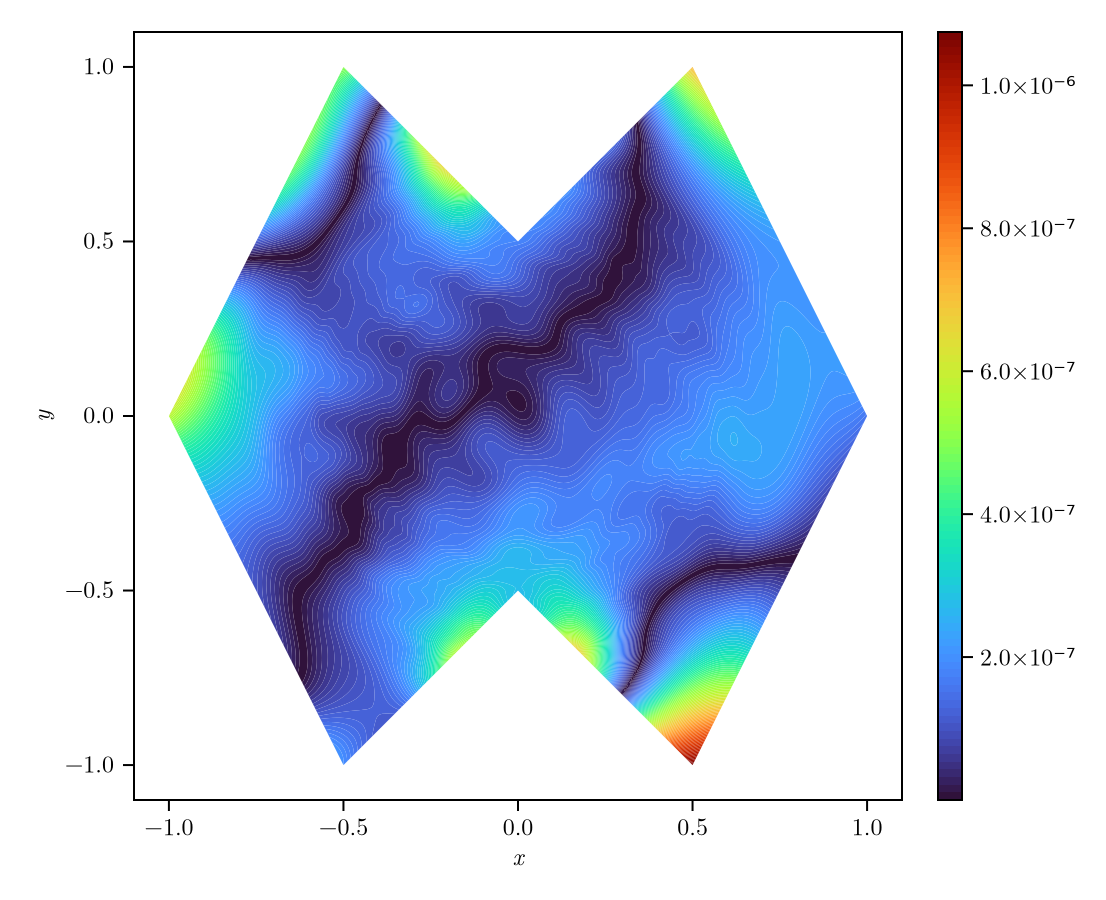
    }
    \caption{Absolute point-wise errors}
  \end{subfigure}
  \caption[Adaptive-quadrature solution for the convection--diffusion--reaction problem on the complex domain]{Adaptive-quadrature solution and final mesh for the
  convection--diffusion--reaction problem~\protect\eqref{eq:conv_diff} on the non-trivial domain using the AQ algorithm.} 
  \label{fig:rhombi-aq-results}
\end{figure}

We begin with the resolution of the circular gradient transition achieved
through the adaptive mesh updates. As shown in~\Cref{fig:rhombi-aq-results},
the final mesh remains relatively coarse yet tracks the circular feature
effectively. The \ac{aq} process is also efficient, reaching a maximum
partition size of 90 in only 7 adaptation steps, all within the first half of
the run. 
The point-wise error plot underscores the
framework's effectiveness: the error remains remarkably uniform across the domain. The boundaries require
only minimal refinement beyond the initial domain edges, indicating that the
$h$-adaptive quadrature localizes effort on the high-gradient interior feature
rather than over-refining the Dirichlet boundaries.
\begin{figure}
  \centering
  \begin{subfigure}[t]{1.0\linewidth}
    \centering
    \includegraphics[width=0.82\linewidth]{
        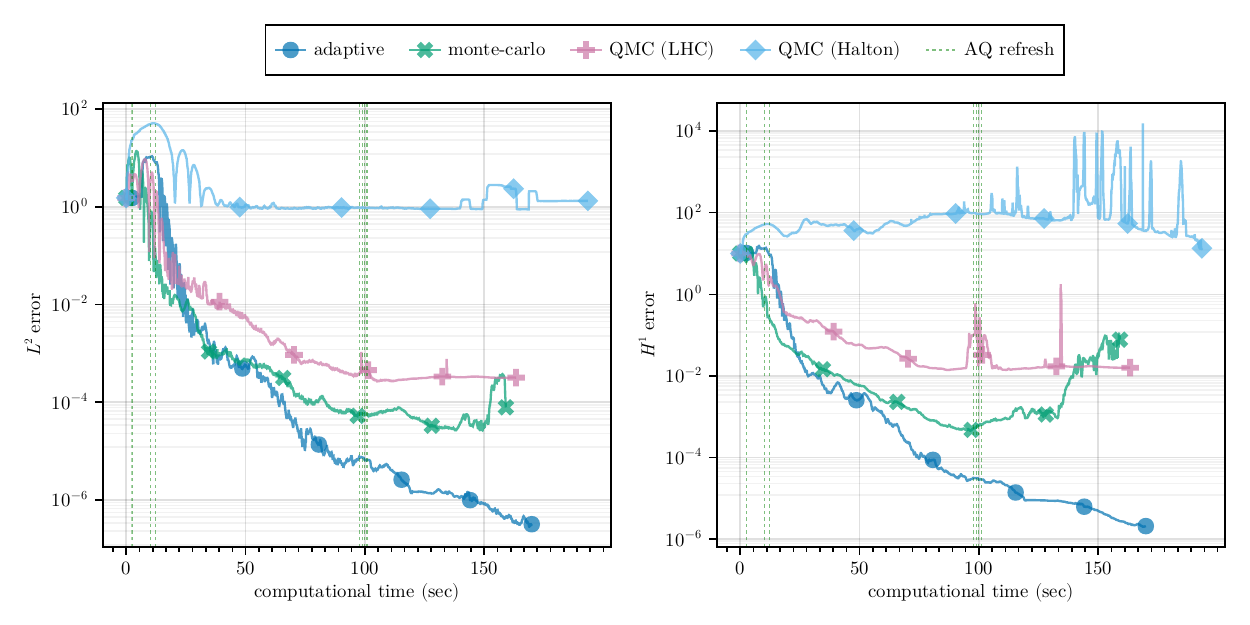
    }
    \caption{$L^2$ and $H^1$ error history comparison across all considered quadrature strategies.}
    \label{fig:rhombi-error-history-panel}
  \end{subfigure}
  \medskip
  \begin{subfigure}[t]{1.0\linewidth}
    \centering
    \includegraphics[width=0.82\linewidth,trim={0 0 0 38},clip]{
        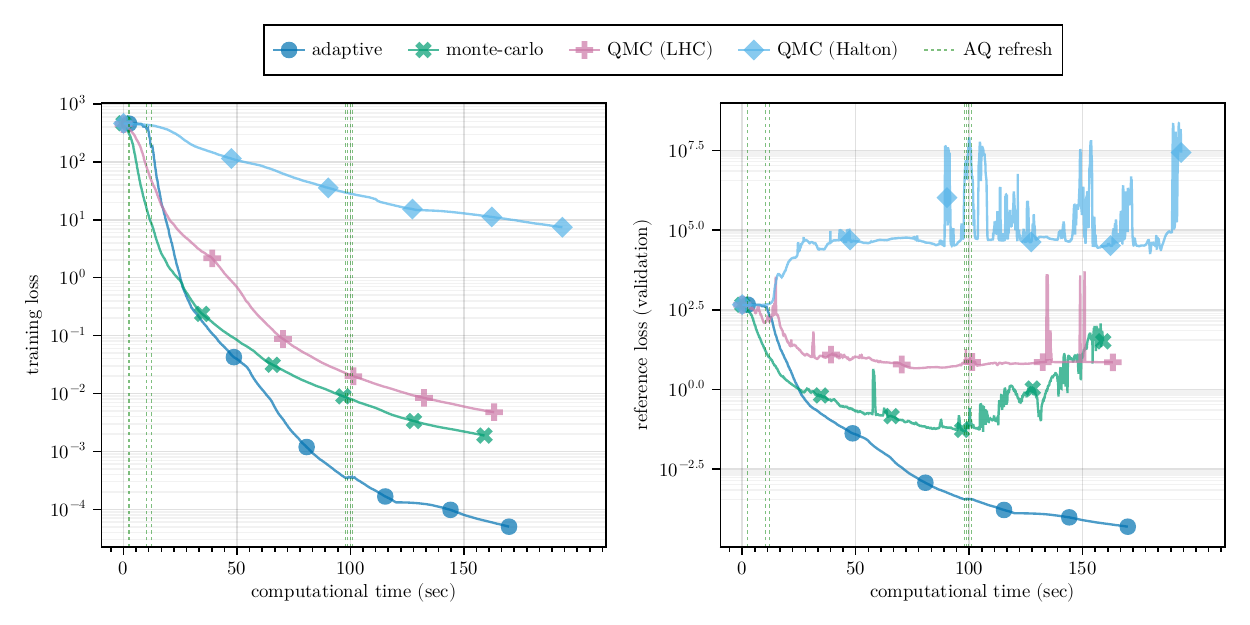
    }
    \caption{Training and reference loss history comparison across all considered quadrature strategies.}
    \label{fig:rhombi-loss-history-panel}
  \end{subfigure}
  \caption[Error and loss histories for the convection--diffusion--reaction problem on the non-trivial domain]{Loss and error histories for the convection--diffusion--reaction problem~\protect\eqref{eq:conv_diff} on the non-trivial domain.}
  \label{fig:rhombi-error-history-comparison}
\end{figure}

Furthermore, the \ac{aq} strategy achieves a markedly better approximation than
the competing methodologies. As shown in
\Cref{fig:rhombi-error-history-panel}, it attains an $L^2$ error more than two
orders of magnitude lower, and an $H^1$ error four orders of magnitude lower,
than any alternative. The error histories are strongly correlated with the
reference loss trajectories, indicating stable training and robust
generalisation (see~\Cref{fig:rhombi-loss-history-panel}). Moreover, only the
\ac{aq} strategy exhibits training and reference losses with the same
qualitative trend, up to a scaling factor. In contrast, the other approaches
exhibit clear overfitting or stagnation, most visibly when their training
histories continue to decrease while the reference loss does not. Notably, the
\ac{qmc} (Halton) method performs worst, which is consistent
with its behaviour in earlier experiments involving sharp gradients or
discontinuities.

\subsection{2D Navier-Stokes lid-driven cavity benchmark}
\label{sec:ns-ldc-aq}

We next consider the original
2D incompressible Navier-Stokes lid-driven cavity benchmark problem with a
Reynolds number $3200$ on $\Omega = [0,1]^2$. We enforce no-slip boundary conditions for the velocity. On the top boundary, we enforce a velocity profile given by 
$\pmb{u}(\pmb{x})=(u_T,0)$ with $u_T(x)$ chosen to be 
\begin{equation} \label{eq:ldc-top-profile}
    u_T(x) = 1 - \frac{\cosh{\left(C_0\left(x-0.5\right)\right)}}{\cosh{\left(0.5 C_0 \right)}}
\end{equation}
with $C_0=50$. This problem setting has been extracted from~\cite{PirateNets2024}, chosen to  
avoid discontinuities in the horizontal velocity at the top two corners of the cavity. 
We adopt the stream function-pressure approach and directly solve
the problem at $\mathrm{Re} = 3200$ (unlike the curriculum learning approach
\cite{PINNFailureModes2021} adopted in~\cite{PirateNets2024}), aimed at testing
the robustness of our approach in tackling the possible instabilities and
convergence to erroneous solutions reported in~\cite{Wang2023PINNsGuide,
Wang2023multiplicity}.

The high Reynolds number of $\mathrm{Re} = 3200$ presents a formidable challenge,
as the flow field is characterized by a hierarchy of vortices across multiple
scales, coupled with complex singular behaviour at the upper corners of the
cavity. To address this, we employ a two-output NN architecture
featuring 6 hidden layers and a width of 50 neurons per layer, using the
$\tanh$ activation function. 
The \ac{aq} algorithm is initialized from a uniform $5 \times 5$ base
partition, with tolerances configured to $\texttt{rtol} = 10^{-5}$ and
$\texttt{refresh\_tol} = 5 \times 10^{-3}$. Notably, the model is trained for a
total of $\num{25000}$ epochs. This represents a substantial reduction in
computational budget compared to the $\num{100000}$ iterations required by the
the state-of-the-art PirateNets~\cite{PirateNets2024}. This choice reflects our empirical
observation that, for this problem, training yields only marginal incremental
benefits beyond a certain iteration threshold. Throughout the optimisation, the
Dirichlet penalty parameter
$\gamma_D$ is held constant at $10$.
\begin{figure}
  \centering
  \begin{subfigure}[t]{0.48\linewidth}
    \includegraphics[width=\linewidth]{
        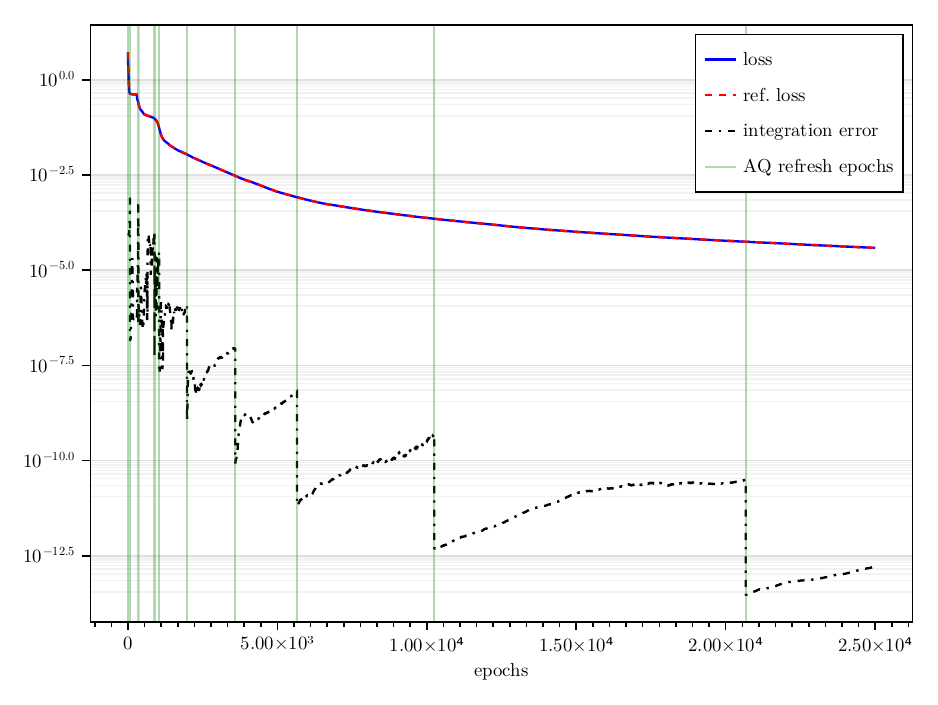
    }
    \caption{Loss curves with adaptive quadrature}
    \label{fig:ns-ldc-aq-training}
  \end{subfigure}
  \begin{subfigure}[t]{0.48\linewidth}
    \includegraphics[width=\linewidth]{
        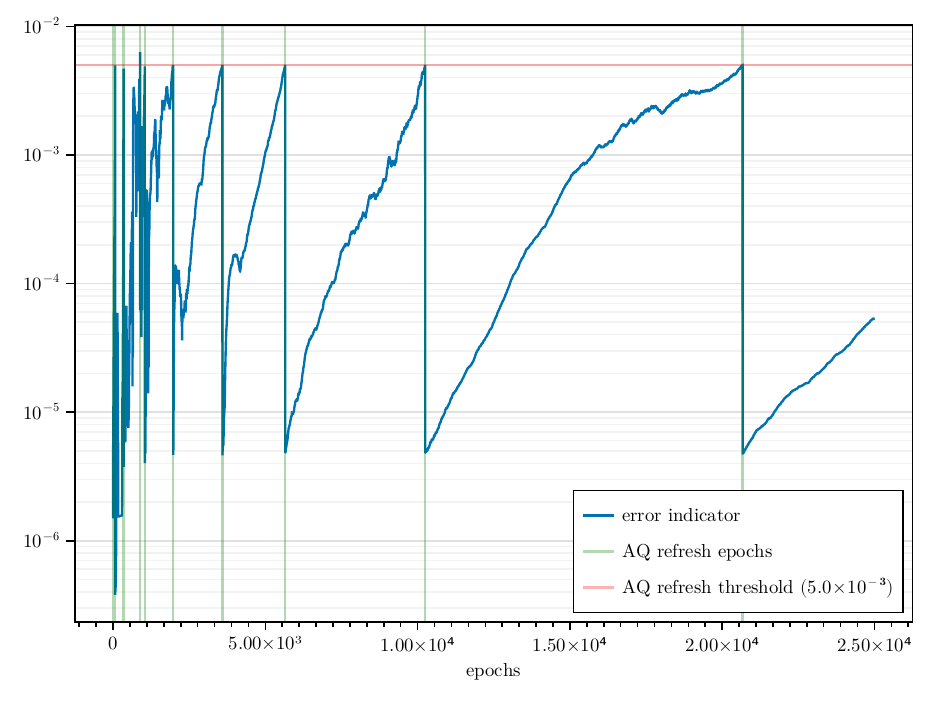
    }
    \caption{Error indicator \& AQ refresh history}
    \label{fig:ns-ldc-aq-error-ind}
  \end{subfigure}
  \caption[Adaptive-quadrature training diagnostics for the 2D Navier-Stokes lid-driven cavity benchmark]{Adaptive-quadrature training diagnostics for the 2D Navier-Stokes
  lid-driven cavity benchmark problem.}
  \label{fig:ns-ldc-aq-training-results}
\end{figure}

We begin by noting that the \ac{aq} approximation remains highly reliable, as
evidenced by the sustained control of the quadrature error shown in
\Cref{fig:ns-ldc-aq-training-results} and the close match between the training
and reference losses throughout the run. The process requires only $10$
\ac{aq} refreshes, most of which occur during the initial training phase. As
the optimisation
progresses, the frequency of these adaptations decreases, ultimately resulting
in a maximum of $480$ interior partitions and $83$ boundary partitions. Given
the strong performance and robust convergence of the \ac{aq} strategy on this
benchmark, we omit further comparisons with less reliable alternatives here.
\begin{figure}
  \centering
  \begin{subfigure}[t]{0.48\linewidth}
    \includegraphics[width=\linewidth]{
        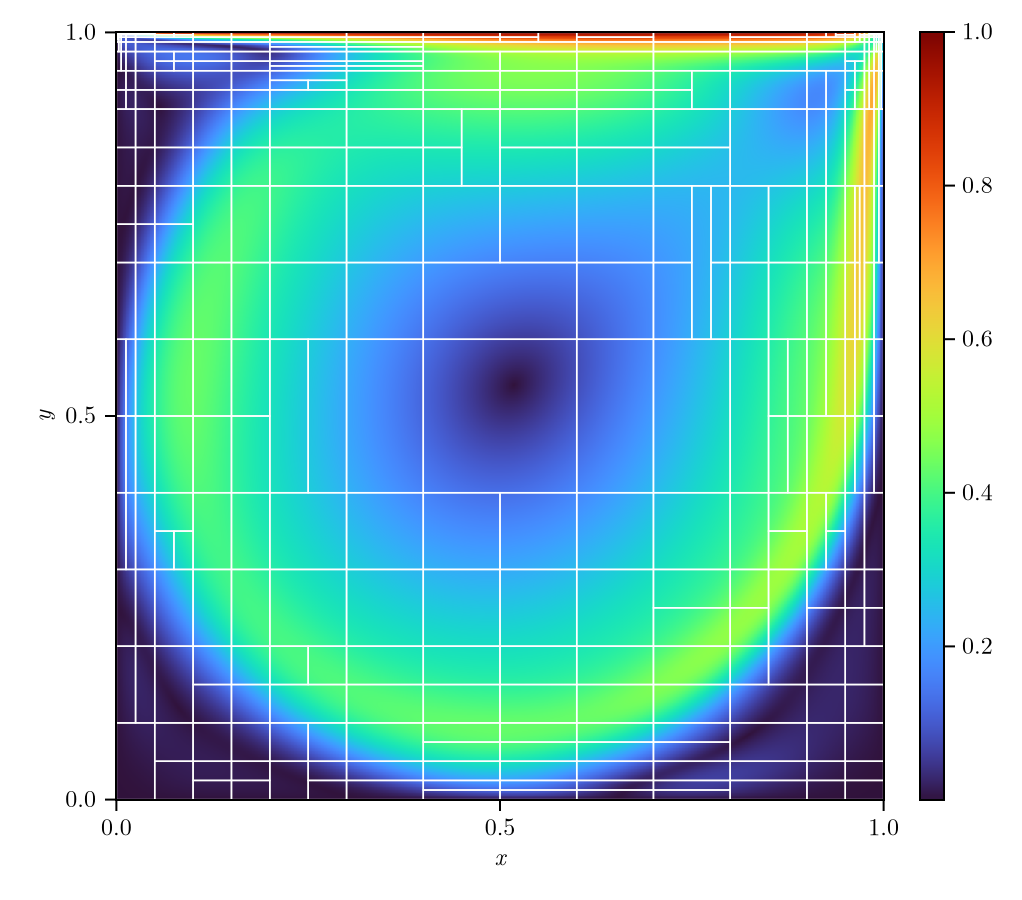
        } \caption{NN approximation with final AQ partition}
    \label{fig:ns-ldc-aq-sol-mesh}
  \end{subfigure}
  \begin{subfigure}[t]{0.48\linewidth}
    \includegraphics[width=\linewidth]{
        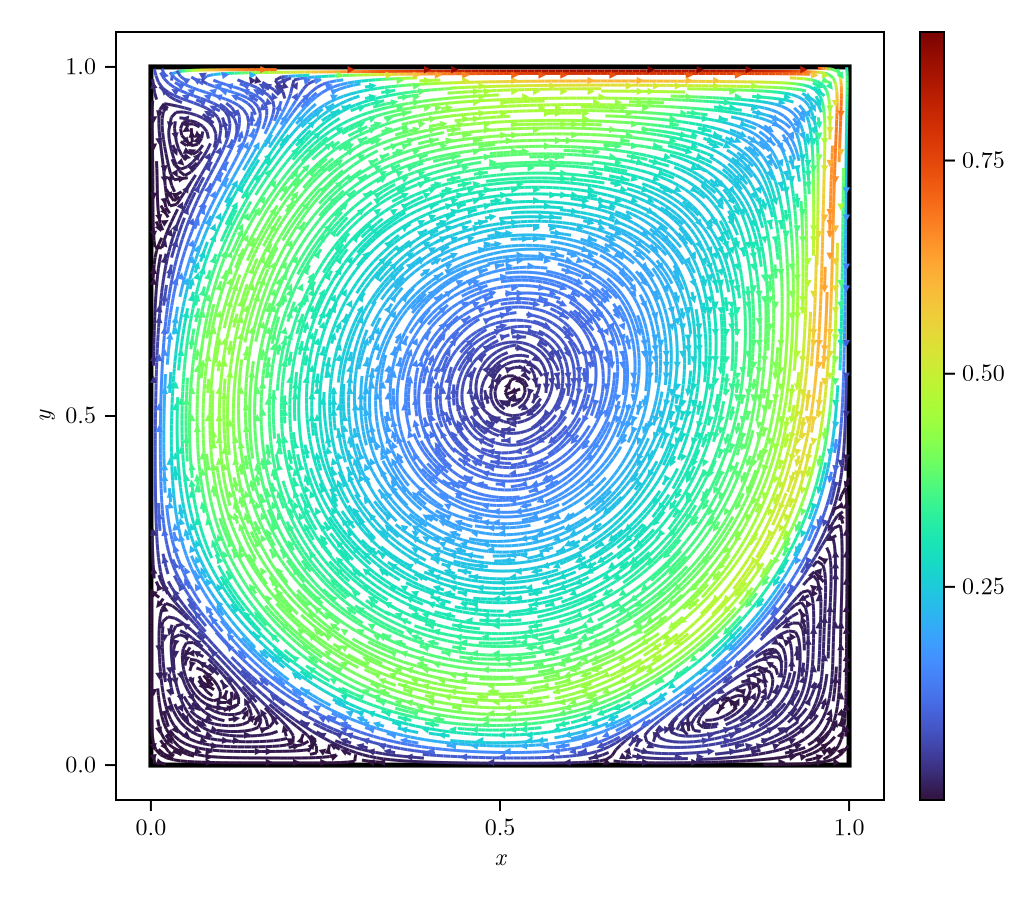
    }
    \caption{Streamlines of the NN approximation}
    \label{fig:ns-ldc-aq-sp}
  \end{subfigure}
  \medskip
  \begin{subfigure}[t]{0.48\linewidth}
    \includegraphics[width=\linewidth]{
        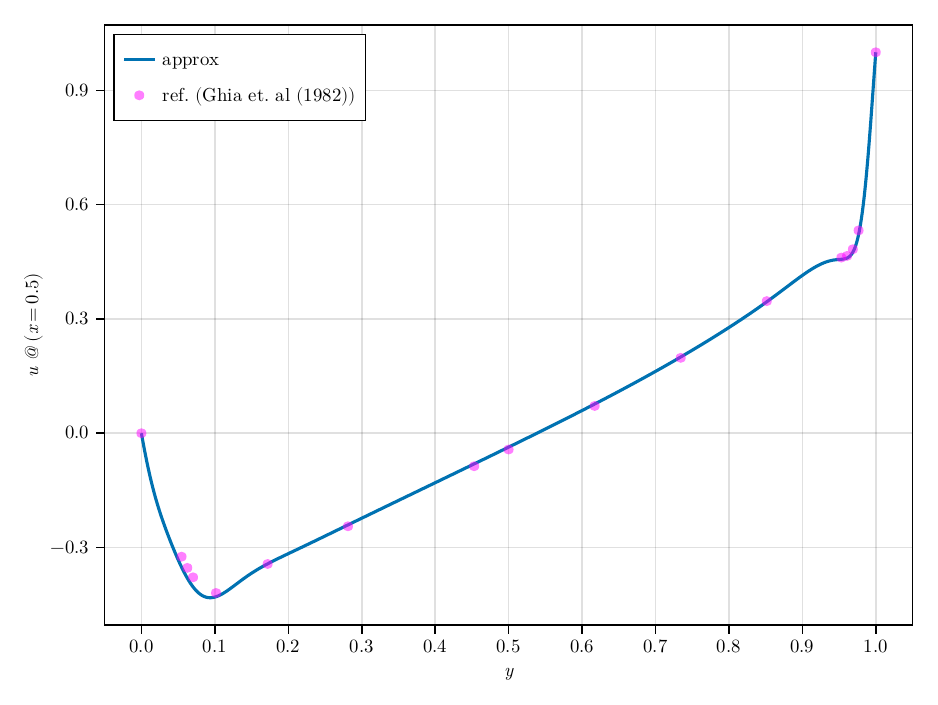
    }
        \caption{Transverse horizontal velocity comparison}
    \label{fig:ns-ldc-transu-comparison}
  \end{subfigure}
  \begin{subfigure}[t]{0.48\linewidth}
    \includegraphics[width=\linewidth]{
        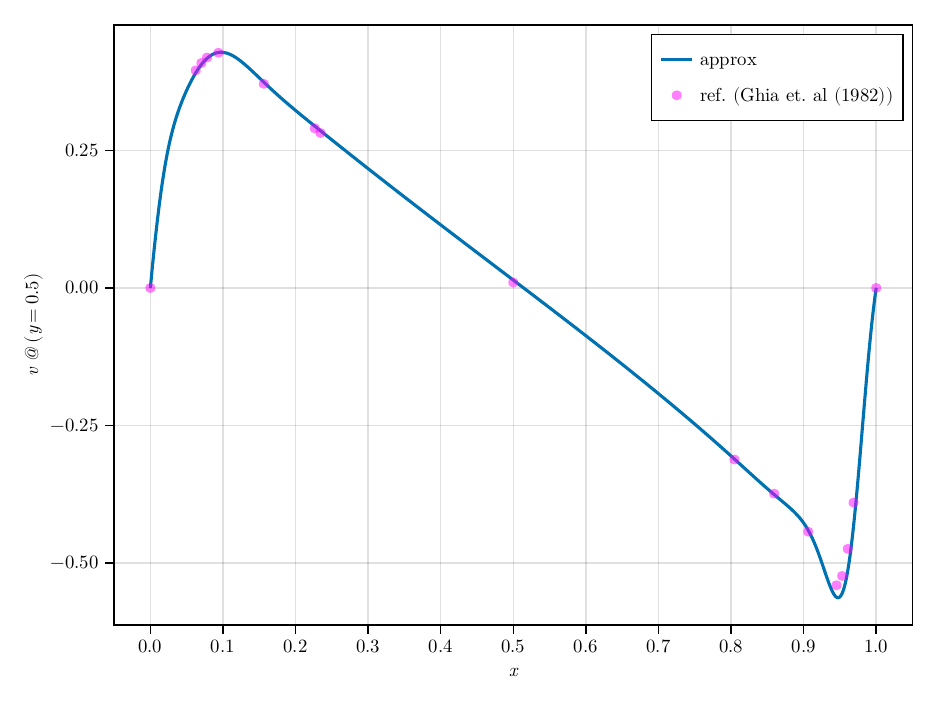
    }
    \caption{Transverse vertical velocity comparison}
    \label{fig:ns-ldc-transv-comparison}
  \end{subfigure}
  \caption{Fluid-velocity results for the 2D Navier-Stokes lid-driven
  cavity and transverse velocity comparison
  against the benchmark data of~\cite{Ghia1982}.}
  \label{fig:ns-ldc-aq-results}
\end{figure}
Examining the quality of the approximation achieved in
\Cref{fig:ns-ldc-aq-results}, we observe that the approach captures the global
flow characteristics well. Panel (a) shows that the \ac{aq} strategy
prioritizes resolution at the singular top corners while maintaining a coarser
discretisation in the central domain. Panel (b) provides further evidence of
solution quality by resolving vortices of various sizes and locations across
the cavity. More importantly, panels (c) and (d) compare the transverse
horizontal and vertical velocity profiles along the central axes against the
standard reference data for transverse velocities taken from~\cite{Ghia1982}.
These
profiles demonstrate a close agreement between the approximation obtained via
the \ac{aq} strategy and the reference results.

The approximation quality is further supported quantitatively by the achieved
relative $L^2$ error of $3.71 \times 10^{-2}$, representing around a $12\%$
improvement over the state-of-the-art result of $4.21 \times 10^{-2}$ reported by
PirateNets~\cite{PirateNets2024}. Notably, the PirateNets results were obtained
using a significantly deeper residual-based architecture and a training budget
of $\num{100000}$ epochs, whereas our approach utilizes a more compact network
and $\num{25000}$ iterations. This comparison highlights the superior
efficiency and accuracy offered by the proposed \ac{aq} algorithm and training
strategy, which effectively extracts higher accuracy from smaller networks and
standard architectures.

\subsection{2D Stokes lid-driven wedge benchmark}
\label{sec:stokes-aq}
To assess the proposed adaptive quadrature strategy on a challenging fluid
problem, we solve the Stokes equations for viscous flow in a lid-driven wedge
as described in~\cite{KiyaniOpt2025}. 
This problem exhibits the multiscale hierarchy of
Moffatt vortices induced by the geometry. 
To make the vortices and mesh partitions easier to inspect, all illustrations
in this section stretch the $x$-axis by a factor of $2$, yielding an aspect
ratio of $1$.
\begin{figure}
  \centering
    \includegraphics[width=\linewidth]{
        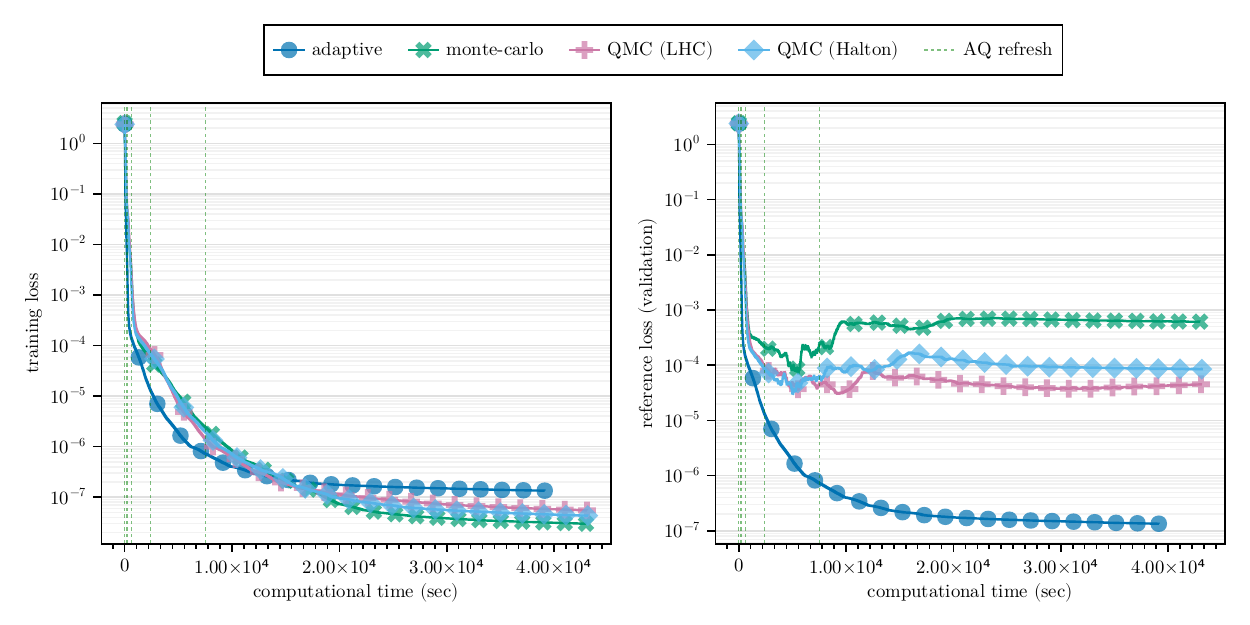
        } \caption{Training and reference loss history comparison for the 2D
        Stokes (Moffatt) problem across all considered quadrature strategies.}
  \label{fig:stokes-loss-history-comparison}
\end{figure}
Although this benchmark is linear, it remains a
significant challenge because the vortex hierarchy weakens rapidly near the
bottom corners; see~\cite[Section 2.6]{KiyaniOpt2025} for a detailed
discussion. We therefore employ a larger network architecture with width 64
and depth 5. The \ac{aq}
algorithm is configured with tolerances of $\texttt{rtol} = 10^{-5}$ and
$\texttt{refresh\_tol} = 5 \times 10^{-3}$, and the model is trained for
$\num{200,000}$ epochs. Throughout the optimisation, the Dirichlet penalty
parameter $\gamma_D$ is held constant at $10$.

For this problem, the comparison between the training and reference
loss histories in~\Cref{fig:stokes-loss-history-comparison} is
crucial. The training losses are initially similar across all approaches, so
they are not by themselves informative about solution quality. However,
the reference loss histories differ substantially. While the
\ac{aq} strategy maintains stable generalisation, the alternative approaches
show early overfitting, with the reference loss diverging and then
stagnating. 
The data efficiency of the \ac{aq} approach is also noteworthy. Starting from a
base mesh of 20 partitions, it requires only 6 adaptation steps, all within
the first quarter of training. The algorithm reaches a maximum of 159 interior
partitions and 40 boundary partitions, with refinement concentrated near the
top two corners in both the interior and boundary discretisation.
\begin{figure}
    \centering
    \begin{subfigure}[t]{0.48\textwidth}
        \centering
        \includegraphics[width=\textwidth]{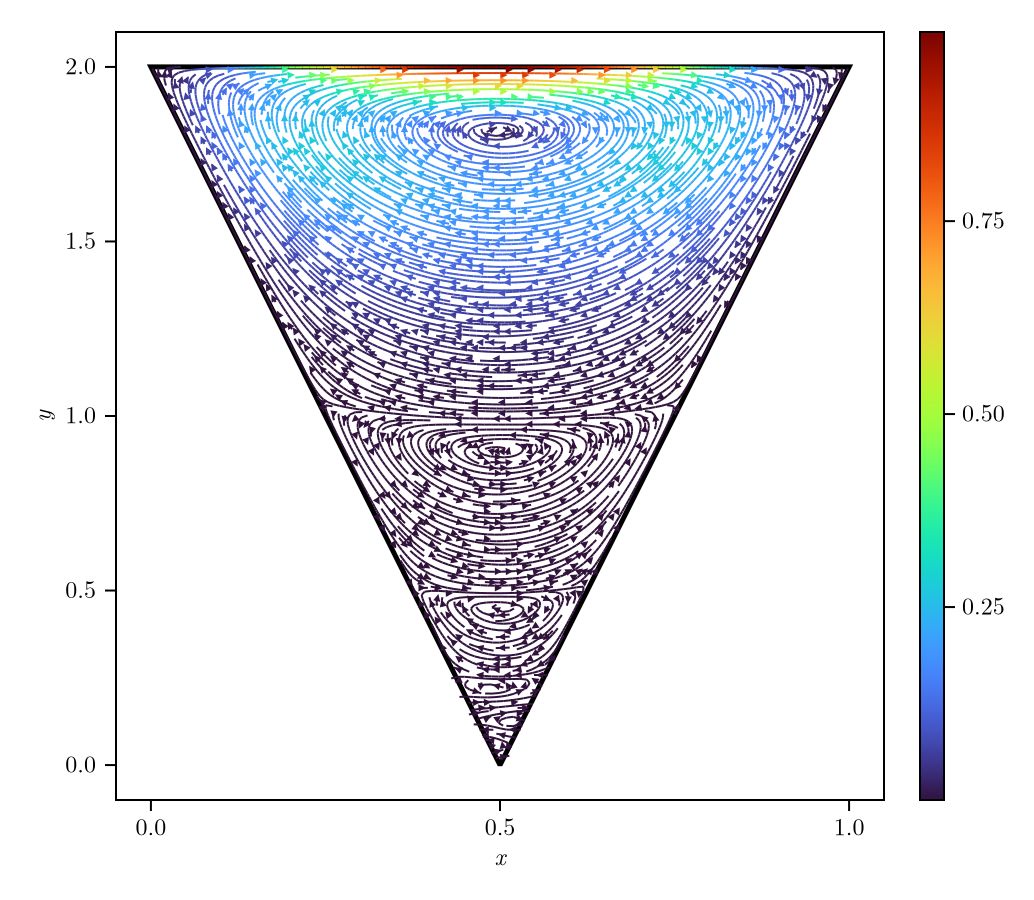}
        \caption{Adaptive quadrature}
        \label{fig:stokes-streamlines-aq}
    \end{subfigure}
    \begin{subfigure}[t]{0.48\textwidth}
        \centering
        \includegraphics[width=\textwidth]{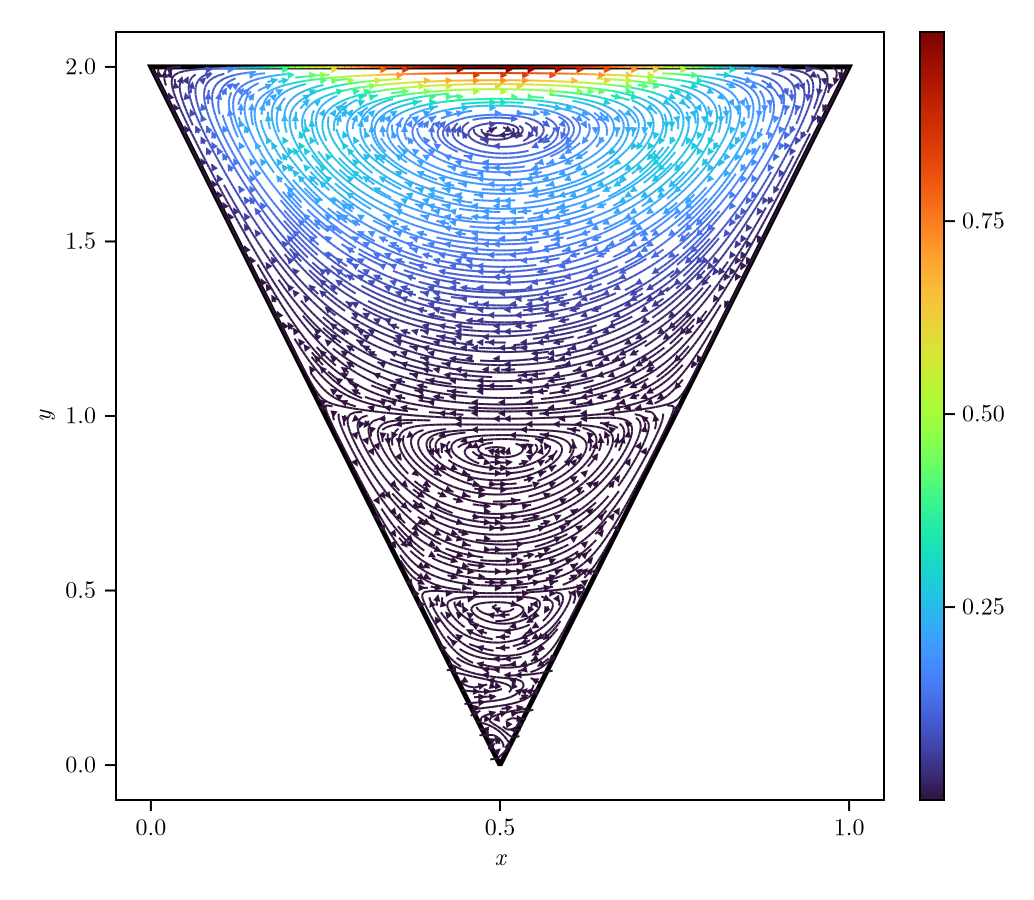}
        \caption{QMC (LHC)}
        \label{fig:stokes-streamlines-qmc-lhc}
    \end{subfigure}
    
    \caption{Streamlines for 2D Stokes (Moffatt) when trained with adaptive
    quadrature and QMC
    (LHC).}
    \label{fig:stokes-streamlines}
\end{figure}

Finally, we compare solution fidelity using only the strongest non-adaptive
baseline, namely QMC (LHC). Both methods appear to outperform the
results reported in~\cite[Figure 23(d)]{KiyaniOpt2025}, despite that study using a
deeper architecture (8 hidden layers) and a much larger training budget of
around $\num{450000}$ epochs. That approach also relied on
$\num{120000}$ residual points and attention-based sampling
\cite{Anagnostopoulos2026}, whereas our \ac{aq} strategy uses only
$\num{7791}$ quadrature points. Moreover, while the reference study employs
hard boundary constraints, we use a penalty formulation to better reflect
practical domains and general boundary conditions, where hard enforcement is
often infeasible.

For a more complete assessment of solution quality, we examine the streamlines
in~\Cref{fig:stokes-streamlines}. These show that the \ac{aq} strategy
captures the vortex hierarchy near the bottom vertex more completely than the
best alternative, where the smaller vortices are only partially resolved. Away
from the bottom corner, the two solutions are already broadly similar, so the
main visible difference is the more complete recovery of the fine vortex
structure by AQ. The adaptive flow structures are 
better formed than the best solution reported in
\cite[Figure 23(a)]{KiyaniOpt2025}. This again highlights the ability of the
\ac{aq} framework to resolve fine-scale multiscale features with a fraction of
the computational resources and a simpler network architecture.

\subsection{Parametric 2D Nonlinear Convection-Diffusion-Reaction problem}
\label{sec:parametric-prob}

Before turning to a fully 3D problem, we consider a 2D parametric benchmark
from~\cite{BerroneVPINNs2022}. The goal is to
illustrate the advantage of adaptive quadrature in a parametric setting
without resorting to specialised architectures for nonlinear operator learning
\cite{DeepONet2021, FNO2021}. The problem reads: find $u$ such that 
\begin{align} \label{eq:parametric-2d-problem}
-\nabla \cdot (\mu \nabla u) + \pmb{\beta}\cdot \nabla u + \sigma\,  {\rm e}^{-pu^2} &= f \quad \text{in \;\; } \Omega\,, \\
u &= g \quad \text{on \; } \partial\Omega \,, 
\end{align}
where $p$ ranges over a prescribed parameter set
$\Omega_p \subset \mathbb R$ on which we seek a reliable approximation of the
PDE. The coefficients $\mu$, $\pmb{\beta}$ and $\sigma$ are independent of
$p$.

For this benchmark, we take $\Omega = (0,1)^2$, $\mu = 1$,
$\pmb{\beta} = [2,3]^T$ and $\sigma = 4$, and choose $f = f(\cdot,p)$ and
$g = g(\cdot,p)$ so that the manufactured solution is
\begin{equation}
u(x,y;p) = \frac{\cos\left(5\left(px+\frac y2\right)\right)}{1+p} + \left(x+\frac y2\right)^2.
\end{equation}
We enlarge the parametric domain from the original range
$\Omega_p = [0.5, 2]$ in~\cite{BerroneVPINNs2022} to $\Omega_p = [0.5, 10]$ in
order to test robustness over a substantially wider parameter interval. This turns the problem into an effectively 3D one in
$(p,x,y)$ space, with computational domain
$\Omega \times \Omega_p$. Dirichlet boundary conditions are
 imposed only on the physical boundary $\partial\Omega \times \Omega_p$ .

We use the tensor-product Gauss-Legendre order pair $(4,5)$ for the adaptive quadrature. 
We start from a uniform mesh
with 5 partitions in each direction and set $\texttt{rtol} = 5 \times 10^{-3}$ and
$\texttt{refresh\_tol} = 5 \times 10^{-2}$. The \ac{nn} has 7 hidden
layers of width 40 with $\tanh$ activation. Training is run for at most
\num{10000} epochs, with the Dirichlet penalty parameter fixed at $\gamma_D = 10$.

The resulting adaptive run requires only three quadrature refreshes. It reaches
a maximum of 874 interior partitions, amounting to \num{55936} quadrature
points, together with 127 boundary partitions, corresponding to \num{2032}
points on the Dirichlet boundary. The final relative $L^2$ and $H^1$ errors
over the full spatial-parametric domain are $1.11 \times 10^{-5}$ and
$1.86 \times 10^{-4}$, respectively.
\begin{figure}
  \centering
  \begin{subfigure}[t]{1.0\linewidth}
    \centering
    \includegraphics[width=0.82\linewidth]{
        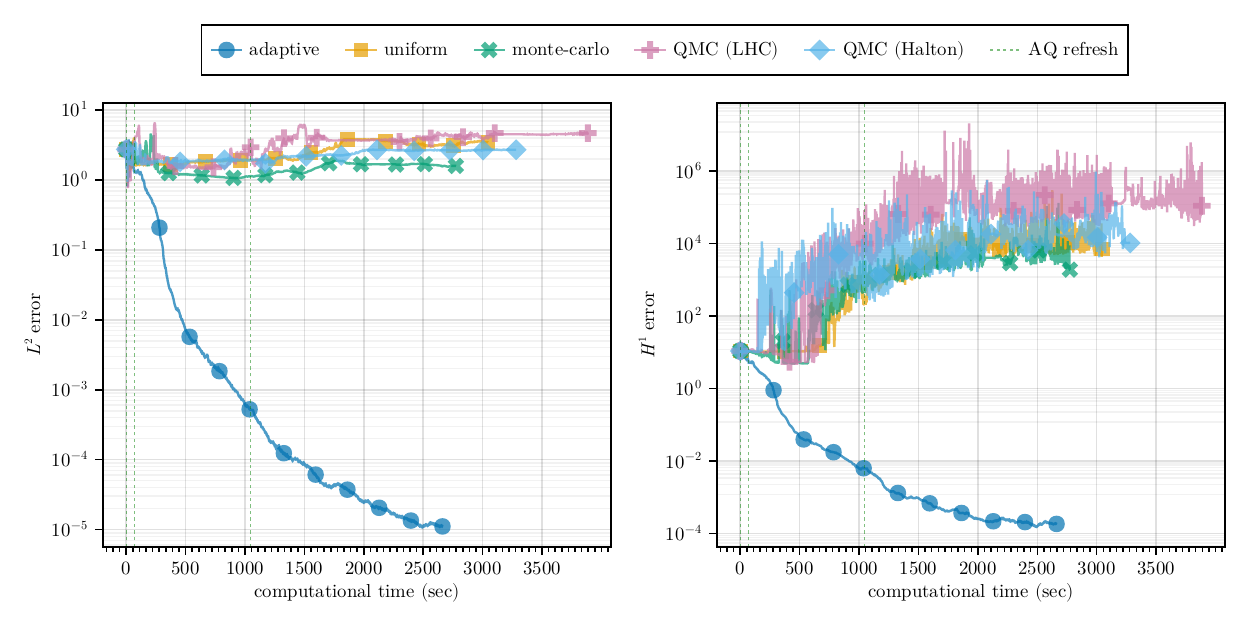
    }
    \caption{$L^2$ and $H^1$ error history comparison across all considered quadrature strategies.}
    \label{fig:parametric-2d-error-history-panel}
  \end{subfigure}
  \medskip
  \begin{subfigure}[t]{1.0\linewidth}
    \centering
    \includegraphics[width=0.82\linewidth,trim={0 0 0 38},clip]{
        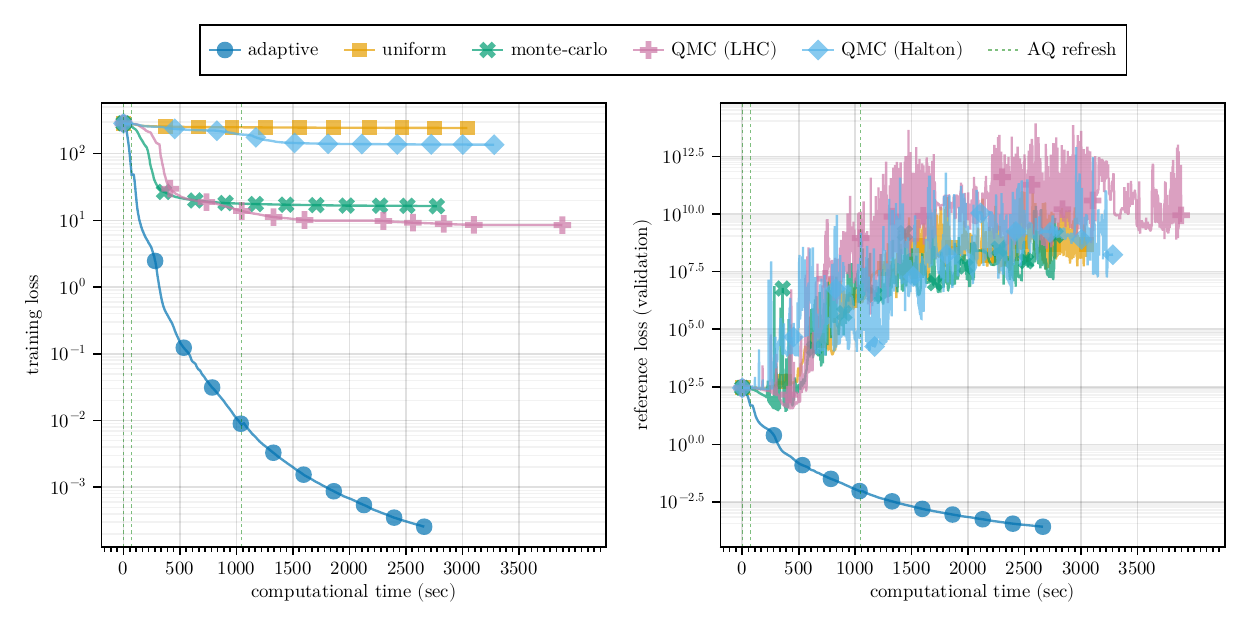
    }
    \caption{Training and reference loss history comparison across all considered quadrature strategies.}
    \label{fig:parametric-2d-loss-history-panel}
  \end{subfigure}
  \caption{Loss and error histories for the nonlinear convection–diffusion–reaction problem in~\eqref{eq:parametric-2d-problem}.}
  \label{fig:parametric-2d-error-history-comparison}
\end{figure}
The comparison in~\Cref{fig:parametric-2d-error-history-panel} shows a clear
advantage for the \ac{aq} strategy. Since the reported errors are measured over
the entire spatial-parametric domain, they directly reflect generalisation
across the full parameter range. In contrast, the non-adaptive alternatives
diverge completely: their error histories stagnate very early, while the
training-reference loss comparison in
\Cref{fig:parametric-2d-loss-history-panel} shows growing discrepancy between
training and reference losses. This shows that adaptive quadrature
is essential for obtaining robust approximations in parametric nonlinear
problems of this type.

\subsection{The (2+1)D Parabolic Poisson equation}

We now consider a challenging (2+1)D parabolic heat equation from the
NIST-\ac{amr} benchmark collection~\cite{NISTAMR}, involving a moving circular
wave front. This benchmark is relevant because it probes whether adaptive
algorithms can control the number of \acp{dof} required in a finite element setting. 
The problem is posed with homogeneous initial condition and
Dirichlet boundary conditions.
The domain is $\Omega = (x_0, x_1) \times (y_0, y_1)$ and the source term is chosen such that the solution of the problem is the
manufactured solution given by
\begin{equation}
u(x,y,t) = \frac{1}{C} (x - x_0)(x - x_1)(y - y_0)(y - y_1)\tan^{-1}(t)\left(\frac{\pi}{2} - \tan^{-1}\left(\alpha (r -t) \right) \right),
\end{equation}
with $r = \sqrt{\left(x - x_c\right)^2 + \left(y - y_c\right)^2}$.
Specifically, we take $\left(x_c , y_c\right) = (0,0)$, $\alpha = 100$,
$C = 10000$, and
$(x_0, x_1) \times (y_0, y_1) \times (0,T) = (0,10) \times (-5,5) \times (0,10)$.

This benchmark is effectively a 3D space-time problem for the \ac{nn}
approximation. We therefore use the Witherden-Vincent optimised quadrature rule
\cite{WitherdenVincent2015} with degree pair $(9,11)$, which is more
data-efficient than standard tensor-product rules. The \ac{aq} base mesh is a
uniform grid with three partitions in each direction. To keep the total number
of quadrature points within the memory limits of the V100 GPU, we use the
larger tolerances $\texttt{rtol} = 5 \times 10^{-2}$ and
$\texttt{refresh\_tol} = 10^{-1}$. The \ac{nn} has 7 hidden layers of
width 25 with $\tanh$ activation. Training is run for at most \num{15000}
epochs, with the initial condition enforced through a penalty term analogous to
the boundary data and the Dirichlet penalty parameter fixed at 10.

\begin{figure}
  \centering
  \begin{subfigure}[t]{1.0\linewidth}
    \centering
    \includegraphics[width=0.82\linewidth]{
        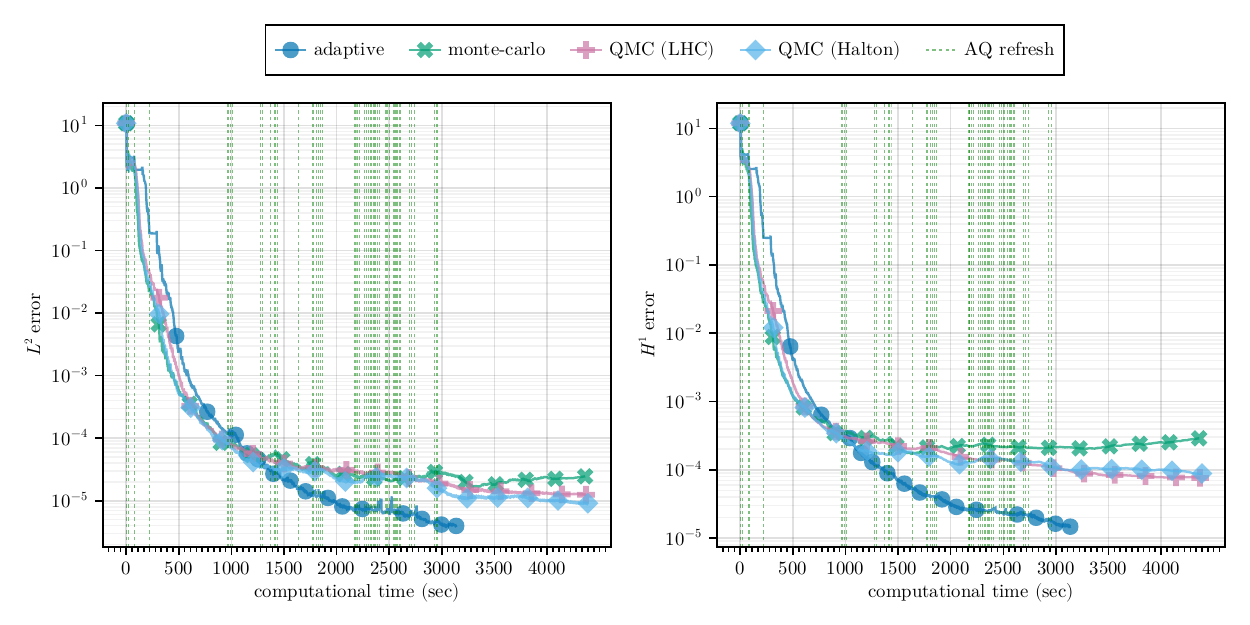
    }
    \caption{$L^2$ and $H^1$ error history comparison across all considered quadrature strategies.}
    \label{fig:movingarc-error-history-panel}
  \end{subfigure}
  \medskip
  \begin{subfigure}[t]{1.0\linewidth}
    \centering
    \includegraphics[width=0.82\linewidth,trim={0 0 0 38},clip]{
        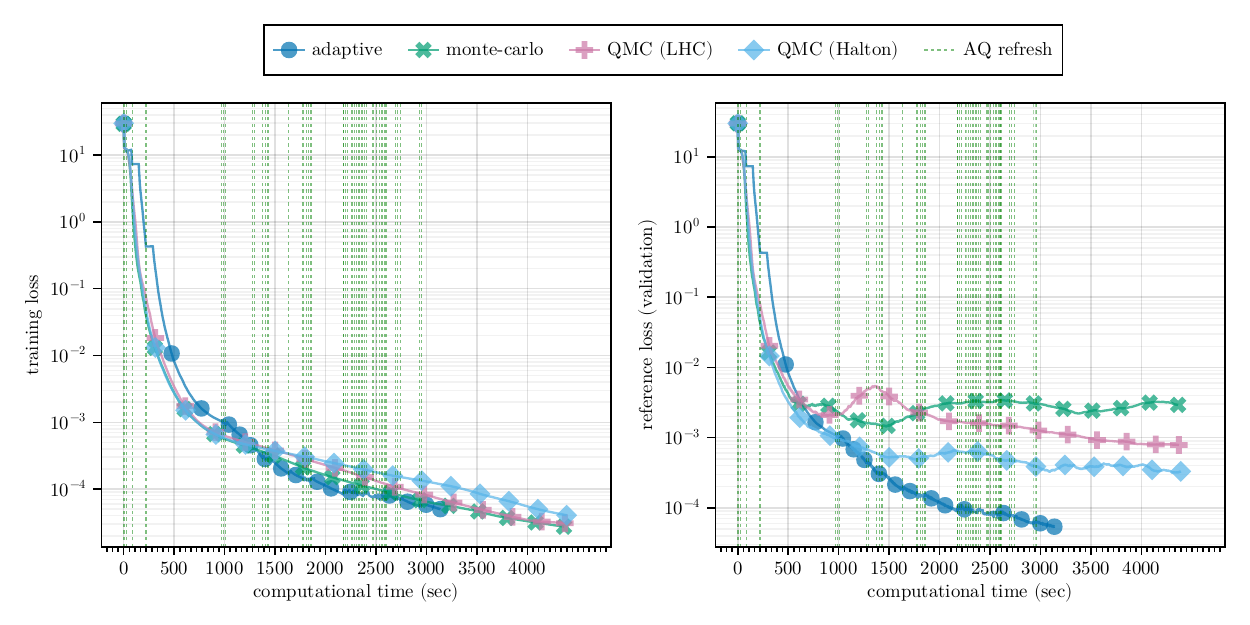
    }
    \caption{Training and reference loss history comparison across all considered quadrature strategies.}
    \label{fig:movingarc-loss-history-panel}
  \end{subfigure}
  \caption{Loss and error histories for the (2+1)D
  Parabolic Poisson problem.}
  \label{fig:movingarc-error-history-comparison}
\end{figure}
\Cref{fig:movingarc-error-history-comparison} shows that, for \ac{aq}, the
reference loss decreases steadily without stagnation
(\Cref{fig:movingarc-loss-history-panel}), and the $L^2$ and $H^1$ error
histories improve overall in a compatible manner
(\Cref{fig:movingarc-error-history-panel}), indicating sustained generalisation.
The non-adaptive alternatives exhibit oscillations, stagnation or slow
generalisation. The \ac{aq} strategy also enjoys a clear time advantage,
reaching the maximum number of epochs in 3,136 seconds, compared with the next
best runtime of 4,362 seconds, for a gain of more than $28\%$. Its final
$L^2$ and $H^1$ errors are $3.96 \times 10^{-6}$ and $1.53 \times 10^{-5}$,
respectively. These are roughly three and five times smaller than the next
best approximation, obtained with \ac{qmc} (Halton), which yields
$9.35 \times 10^{-6}$ and $8.48 \times 10^{-5}$.
For the quadrature comparison in this experiment, we omit uniform quadrature
because, as seen in earlier experiments, it cannot resolve sharp solution
profiles adequately. The same limitation also applies to the subsequent 3D
convection-diffusion problem with a boundary layer.

\subsection{The 3D convection-diffusion problem with boundary layer}

We finally consider a purely 3D stationary convection-diffusion problem. It
features a steep boundary layer and therefore presents a significant challenge
for both traditional numerical methods and \ac{nn} approximations. As
discussed earlier in~\Cref{sec:1d-adv-diff}, the 1D analogue already showed a
clear benefit from adaptive quadrature. Similar to the previous experiment,
this benchmark is drawn from the NIST-\ac{amr} collection~\cite{NISTAMR} and
is defined as follows:
\begin{align} \label{eq:adv-diff-3d-problem}
- \epsilon \Delta u + \pmb{\beta}\cdot \nabla u &= f \quad \text{in \;\; } \Omega\, \subset \mathbb{R}^3\\
u &= g \quad \text{on \; } \partial\Omega \,.
\end{align}
We take $\epsilon = 10^{-2}$ and $\pmb{\beta} = (2,1,1)$. The domain is
$\Omega = (-1,1)^3$, and $f$ and $g$ are chosen so that the manufactured
solution is
\begin{equation}
u(x,y,z) = \left(1 - e^{-(1-x)/\epsilon}\right) \left(1 - e^{-(1-y)/\epsilon}\right) \left(1 - e^{-(1-z)/\epsilon}\right) \cos\left(\pi\left(x + y + z\right)\right).
\end{equation}
For this problem, we use the Xiao-Gimbutas quadrature
\cite{XiaoGimbutas2010} with degree pair $(9,11)$, again to improve
data-efficiency relative to standard tensor-product rules. This mirrors the
choice made in the preceding parabolic Poisson experiment. The \ac{aq} base
mesh is a uniform grid with 5 partitions in each direction. To keep the total
number of quadrature points within the memory limits of the V100 GPU, we use
$\texttt{rtol} = 1 \times 10^{-2}$ together with the tighter
$\texttt{refresh\_tol} = 5 \times 10^{-2}$. The \ac{nn} has 4 hidden
layers of width 50 with $\tanh$ activation. Training is run for at most
\num{15000} epochs, with the Dirichlet penalty parameter fixed at 10 for
consistency with the earlier experiments.

\begin{figure}
  \centering
  \begin{subfigure}[t]{1.0\linewidth}
    \centering
    \includegraphics[width=0.82\linewidth]{
        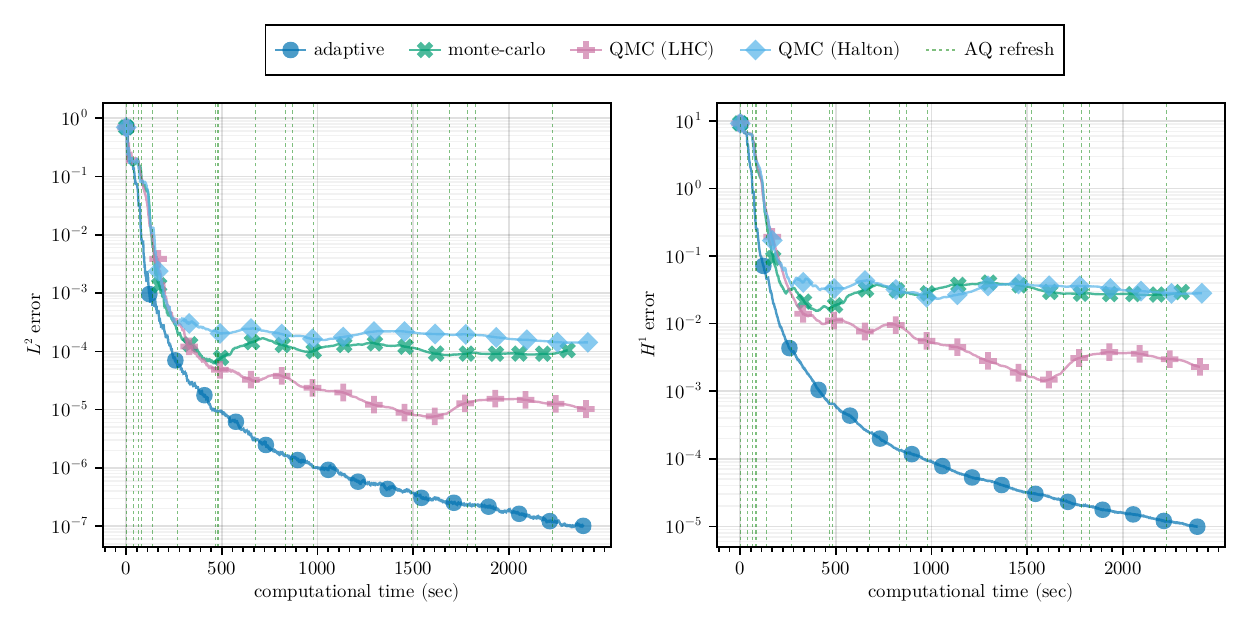
    }
    \caption{$L^2$ and $H^1$ error history comparison across all considered quadrature strategies.}
    \label{fig:adv-diff-3d-error-history-panel}
  \end{subfigure}
  \medskip
  \begin{subfigure}[t]{1.0\linewidth}
    \centering
    \includegraphics[width=0.82\linewidth,trim={0 0 0 38},clip]{
        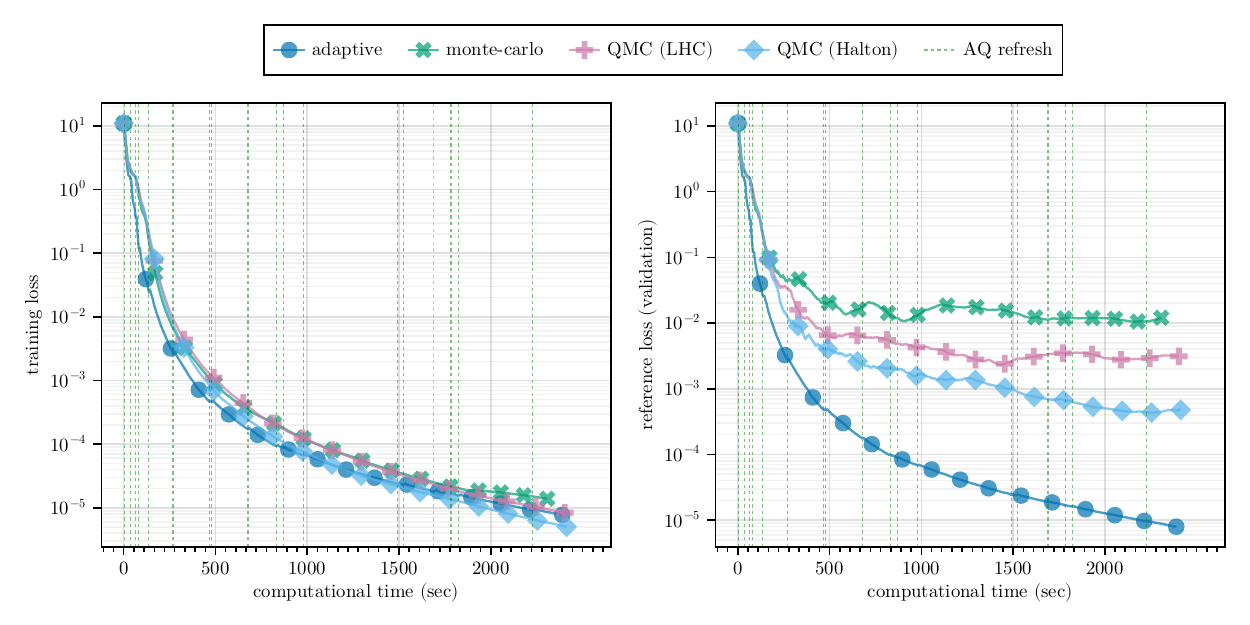
    }
    \caption{Training and reference loss history comparison across all considered quadrature strategies.}
    \label{fig:adv-diff-3d-loss-history-panel}
  \end{subfigure}
  \caption{Loss and error histories for the 3D convection-diffusion  problem.}
  \label{fig:adv-diff-3d-error-history-comparison}
\end{figure}
The comparison in~\Cref{fig:adv-diff-3d-error-history-panel} and
\Cref{fig:adv-diff-3d-loss-history-panel} further reinforces the advantage of
the \ac{aq} strategy. It yields $L^2$ and $H^1$ errors that are slightly more
than two orders of magnitude lower than the next best results, obtained by
\ac{qmc} (\ac{lhc}). As in the earlier experiments, the gain in $H^1$ is
slightly more pronounced than in $L^2$. Consistently across the experiments with sharp solution profiles,
\ac{qmc} (\ac{lhc}) performs significantly better than \ac{mc} and
\ac{qmc} (Halton). Moreover, although the training losses behave similarly
across the considered approaches, the reference loss reveals much slower
convergence for the non-adaptive methods. This creates a substantial
generalisation gap and leads to errors that are orders of magnitude worse.
These findings again highlight the importance of global quadrature error
control for good generalisation, ensuring that the \ac{nn}'s
approximation power is used to resolve the solution rather than wasted on
overfitting.

\section{Conclusion} \label{sec:conclusion}

We studied the role of numerical quadrature in residual-minimisation methods
for neural approximation of \acp{pde}. Our analysis separates approximation,
quadrature and optimisation errors, and the resulting nonlinear Strang-type
estimate shows that controlling the perturbation of the discrete residual loss
is important for obtaining accurate approximations.

Motivated by this analysis, we proposed an anisotropic, bisection-based
$h$-adaptive composite quadrature with non-nested primal and reference rules,
together with a refresh-based training strategy that rebuilds the quadrature
only when an online error indicator exceeds a prescribed threshold. This
combines explicit control of quadrature-induced perturbations with moderate
computational cost throughout training.

Across a broad set of benchmark problems, including elliptic, time-dependent
and parametric \acp{pde} in one, two and three spatial dimensions, the
proposed approach reduced the gap between training and reference losses, used
quadrature points more efficiently than non-adaptive strategies and delivered
strong approximation accuracy in the tested problems. The experiments also
indicate that only a small number of quadrature refreshes is needed relative to
the total number of training iterations, making the approach practical for
residual-based neural \ac{pde} solvers.

\printbibliography 

\end{document}